\documentclass[11pt]{amsart}

\usepackage{fullpage}
\usepackage{amsmath,amssymb,mathrsfs}
\usepackage{stmaryrd}
\usepackage{hyperref}
\usepackage{tropical}
\usepackage{color}
\usepackage{graphicx}
\usepackage[3D]{movie15}
\usepackage{tikz}
\usepackage{hypergraph}
\usepackage[inline]{enumitem}

\usepackage{tikzexternal} 
\tikzsetexternalprefix{figures/}
\providecommand{\arxiv}[2][]{\href{http://www.arXiv.org/abs/#2}{arXiv:#2}}

\newtheorem{theorem}{Theorem}
\newtheorem{proposition}[theorem]{Proposition}
\newtheorem{lemma}[theorem]{Lemma}
\newtheorem{corollary}[theorem]{Corollary}
\newtheorem{definition}[theorem]{Definition}
\newtheorem{propositiondefinition}[theorem]{Proposition-Definition}
\newtheorem{remark}[theorem]{Remark}
\newtheorem{example}[theorem]{Example}

\newcommand{\R}{\mathbb{R}}

\newcommand{\B}{\mathbb{B}}
\newcommand{\GG}{\mathcal{G}}
\newcommand{\BB}{\mathscr{B}}
\renewcommand{\HH}{\mathscr{H}}
\renewcommand{\CC}{\mathscr{C}}
\renewcommand{\DD}{\mathscr{D}}
\renewcommand{\SS}{\mathscr{S}}

\renewcommand{\tangent}{\mathscr{T}}
\newcommand{\reachgraph}{\mathcal{R}}
\newcommand{\NN}{\mathscr{N}}
\newcommand{\mcA}{\mathcal{A}}
\newcommand{\mcB}{\partial\CC}
\newcommand{\mcC}{\mathcal{C}}
\newcommand{\ipolar}[2][i]{{#2}^\circ_{#1}}
\newcommand{\mydef}{:=}
\newcommand{\mpinverse}[1]{#1^{-}}
\newcommand{\vectinverse}[1]{#1^{-}}
\newcommand{\compl}[1]{[n]\setminus #1}

\newcommand{\projective}{\mathbb{P}}
\newcommand{\proj}{\mathbb{P}}
\newcommand{\projmax}{\mathbb{P}_{\max}}
\newcommand{\projspace}[1][n-1]{\proj^{#1}}
\newcommand{\projmaxspace}[1][n-1]{\projmax^{#1}}
\newcommand{\myemph}[1]{\emph{#1}}
\renewcommand{\equiv}{\sim}
\DeclareMathOperator{\apex}{\mathsf{apex}}
\DeclareMathOperator{\sect}{\mathsf{sect}}

\DeclareMathOperator{\type}{\mathsf{type}}

\newcounter{tikzfigures}

\title{Minimal external representations of tropical polyhedra}
\author{Xavier {A}llamigeon}
\address{INRIA and CMAP, \'Ecole Polytechnique, 91128 Palaiseau Cedex France}
\email{xavier.allamigeon@inria.fr}
\author{Ricardo D. Katz}
\address{CONICET. Postal address:\!\! Instituto de Matem\'atica 
``Beppo Levi'',\! Universidad Nacional de Rosario, 
Avenida Pellegrini 250, 2000 Rosario, Argentina.}
\email{rkatz@fceia.unr.edu.ar}
\date{\today}
\catcode`\@=11
\@namedef{subjclassname@2010}{%
  \textup{2010} Mathematics Subject Classification}
\catcode`\@=12
\date{December 21, 2012}

\begin{document}
\maketitle
	
\begin{abstract}
Tropical polyhedra are known to be representable externally, as intersections of finitely many tropical half-spaces. However, unlike in the classical case, the extreme rays of their polar cones provide external representations containing in general superfluous half-spaces. In this paper, we prove that any tropical polyhedral cone in $\R^n$ (also known as ``tropical polytope'' in the literature) 
admits an essentially unique minimal  
external representation. The result is obtained by establishing a (partial) anti-exchange property of half-spaces. Moreover, we show that the apices of the half-spaces appearing in such 
non-redundant external representations are vertices of the cell complex associated with the polyhedral cone. We also establish a necessary condition for a vertex of this cell complex to be the apex of a non-redundant half-space. It is shown that this condition is sufficient for a dense class of polyhedral cones having ``generic extremities''.
\end{abstract}

\section{Introduction}

Tropical convex geometry consists in the study of the analogues of convex sets in tropical algebra. In this paper, we consider the max-plus semiring $\maxplus$ instantiation of tropical algebra, dealing with the set $\R\cup \{ -\infty \}$ equipped with $x \mpplus y \mydef \max \{ x,y \}$ as addition and $x \mptimes y \mydef x + y$ as multiplication. 
Thus, in the max-plus semiring, $-\infty$ is the neutral element for addition and $0$ is the neutral element for multiplication. 
The $n$-fold product space $\maxplus^n$ carries the structure of a semimodule over $\maxplus$ when equipped with the tropical scalar multiplication $(\lambda,x) \mapsto \lambda x \mydef (\lambda + x_1, \ldots ,\lambda + x_n)$ and the component-wise tropical addition $(x,y) \mapsto x \mpplus y \mydef (\max \{ x_1,y_1 \}, \ldots ,\max \{ x_n,y_n \})$. 

Since any number is ``non-negative'' in the max-plus semiring 
(\ie{}, greater than or equal to the neutral element for addition $-\infty$), 
following the analogy with ordinary convexity, 
a subset $\CC \subset \maxplus^n$ is said to be a 
\myemph{tropical convex set} if 
\begin{equation}\label{DefTropConvex}
\lambda x \mpplus \mu y \in \CC \makebox{ for all } x,y \in \CC \makebox{ and } \lambda, \mu \in \maxplus \makebox{ such that } \lambda \mpplus  \mu = 0.
\end{equation}
Similarly, $\CC$ is called a \myemph{tropical (convex) cone} when~\eqref{DefTropConvex} is satisfied without the condition $\lambda \mpplus  \mu = 0$. 

A tropical cone $\CC$ is said to be \myemph{polyhedral} when it can be generated by finitely many vectors, meaning that there exists 
$\{v^1,\ldots , v^p\}\subset \maxplus^n$ such that
\begin{equation}\label{DefPolyCone}
\CC = \left\{ v \in \maxplus^n \mid v =\lambda_1 v^1 \mpplus \dots \mpplus \lambda_p v^p 
\makebox{ for some } \lambda_1,\ldots ,\lambda_p \in \maxplus \right\} \; . 
\end{equation}  
In this paper, we mainly deal with the situation in which all the generators $v^i$ belong to $\R^n$. In this case, we say that $\CC$ is a \myemph{real polyhedral cone}. 

It is worth mentioning that our terminology differs from the one introduced by Develin and Sturmfels in~\cite{DS}, where tropical cones are called \myemph{tropical convex sets}, and real polyhedral cones are referred to as \myemph{tropical polytopes}. 

Tropical cones are closed under tropical scalar multiplication. In consequence, it turns out to be convenient to identify a real polyhedral cone with its image in the \myemph{real projective space}
\[
\projspace \mydef \R^n / (1, \dots, 1)\R \; .
\] 
We adopt this approach here and, for visualization purposes, represent a vector 
$(x_1,\ldots ,x_n)\in \R^n$ by the point $(x_2-x_1,\ldots,x_n-x_1)$ of $\R^{n-1}$. 
For example, the real polyhedral cone generated by $v^1=(0,1,3)$, $v^2=(0,4,1)$ and $v^3=(0,9,4)$ is depicted in Figure~\ref{FigTropicalCone}. This tropical cone is given by the bounded gray region together with the line segments joining the points $v^1$ and $v^3$ to it.
\begin{figure}
\begin{center}
\begin{tikzpicture}[convex/.style={draw=black,ultra thick,fill=lightgray,fill opacity=0.7},point/.style={blue!50},>=triangle 45,
shsborder/.style={green!60!black,dashdotted,ultra thick},
shs/.style={draw=none,pattern=north west lines,pattern color=green!60!black,fill opacity=0.5},
sapex/.style={green!60!black},
]
\draw[gray!30,very thin] (-0.5,-0.5) grid (10.5,6.5);
\draw[gray!50,->] (-0.5,0) -- (10.5,0) node[color=gray!50,below] {$x_2-x_1$};
\draw[gray!50,->] (0,-0.5) -- (0,6.5) node[color=gray!50,above] {$x_3-x_1$};

\node[coordinate] (v1) at (1,3) {};
\node[coordinate] (v12) at (4,3) {};
\node[coordinate] (v2) at (4,1) {};
\node[coordinate] (v23) at (6,1) {};
\node[coordinate] (v3) at (9,4) {};
\node[coordinate] (v31) at (8,3) {};

\filldraw[convex] (v1) -- (v12) -- (v2) -- (v23) -- (v3) -- (v31) -- cycle;
\filldraw[point] (v1) circle (3pt) node[left=3pt] {$v^1$};
\filldraw[point] (v2) circle (3pt) node[below left=1.5pt] {$v^2$};
\filldraw[point] (v3) circle (3pt) node[right=3pt] {$v^3$};

\draw[shsborder] (-0.5,1) -- (6,1) -- (10.5,5.5);
\filldraw[shs] (-0.5,1) -- (6,1) -- (10.5,5.5) -- (10.5,6.2) -- (5.8,1.5) -- (-0.5,1.5) -- cycle;
\filldraw[sapex] (6,1) circle (3.5pt) node[below right=1.5pt] {$(0,6,1)$};
\end{tikzpicture}	
\end{center}
\caption{The real polyhedral cone generated by $v^1=(0,1,3)$, $v^2=(0,4,1)$ and $v^3=(0,9,4)$ (gray), together with a tropical half-space with apex $(0,6,1)$ (green).}\label{FigTropicalCone}
\end{figure}
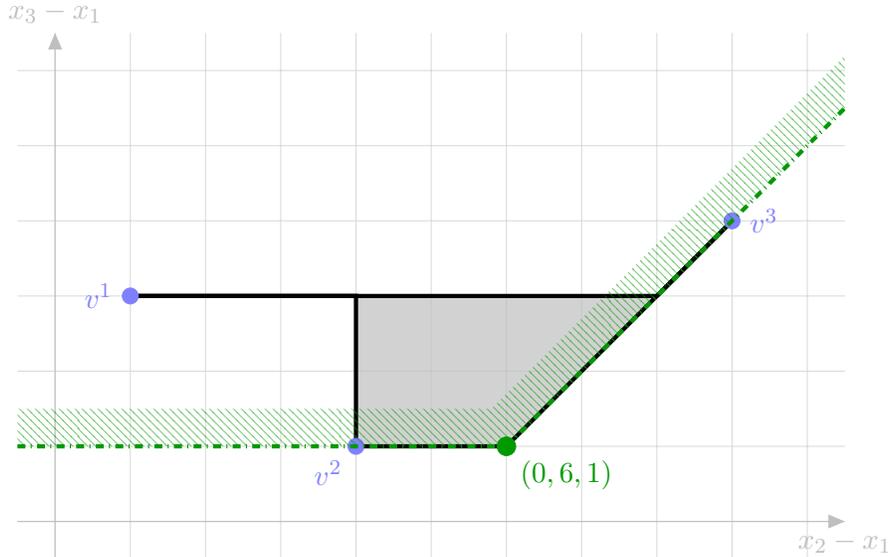  

As in classical convexity, any tropical polyhedral cone admits an ``external'' representation, as the intersection of finitely many tropical 
half-spaces~\cite{GK09}. A \myemph{tropical half-space} can be defined as a set of the form:
\[
\HH = \Bigl\{ x \in \maxplus^n \mid \max_{i \in I} \{x_i - \alpha_i\} \geq \max_{j \in J} \{x_j - \alpha_j\} \Bigr\} \;,
\]
where $I$ and $J$ are non-empty disjoint subsets of $\oneto{n}\mydef \{1, \dots, n\}$ and $\alpha_h \in \R$ for $h \in I \cup J$. 
We say that $\HH$ is \myemph{non-degenerate} when $I \cup J = [n]$. In this case, the vector $(\alpha_1,\ldots ,\alpha_n)$ is an element of $\R^n$, and is called the \myemph{apex} of $\HH$. Note that apices are defined up to a scalar multiple. Consequently, we think of them as elements of~$\projspace$. 
For instance, the tropical half-space $\{ x \in \maxplus^3 \mid  x_3 -1 \geq \max\{x_1, x_2 - 6\} \}$ is non-degenerate, and its apex is the point $(0,6,1)$. It corresponds to the region located above the green dashed half-lines in Figure~\ref{FigTropicalCone}. We warn the reader that, unless explicitly specified, every tropical half-space considered in the sequel will be non-degenerate.

Tropical convex sets, which were introduced by K. Zimmermann~\cite{zimmerman77} when studying discrete optimization problems, have been the topic of many works coming from different fields. 
Tropical cones have been studied in idempotent analysis. It came from an observation of Maslov implying the solutions of a Hamilton-Jacobi equation associated with a deterministic optimal control problem belong to structures similar to convex cones, called semimodules or idempotent linear spaces~\cite{litvinov00,cgq02}. Besides, the invariant spaces that appear in the study of some discrete event systems  
are naturally equipped with structures of 
tropical cones, see~\cite{ccggq99}. This motivated the study of
tropical cones or semimodules by Cohen, 
Gaubert and Quadrat~\cite{cgq00,cgq02}, following the
algebraic approach to discrete event systems initiated
by Cohen, Dubois, Quadrat and Viot~\cite{cohen85a}. 
Another interest in the tropical analogues of convex sets comes from 
abstract convex analysis~\cite{ACA}, see for instance~\cite{cgqs04,NiticaSinger07a}. With the same motivation, the notion of $\B$-convexity (which is another name for tropical convexity) was introduced and studied by Briec, Horvath and Rubinov~\cite{BriecHorvath04,BriecHorvRub05}. 
The theory of tropical convexity has recently been developed
in relation to tropical geometry. Real (tropical) polyhedral cones were considered by Develin and Sturmfels~\cite{DS}. They developed a combinatorial approach, thinking of these tropical cones as polyhedral complexes in the usual sense. 
This was at the origin of several works 
(see for example~\cite{joswig04,blockyu06,DevelinYu}) 
by the same authors and by Joswig, Block and Yu, to quote but a few. 

In classical convex geometry, any full-dimensional polyhedron admits a unique minimal external representation, which is provided by the facet-defining half-spaces. No analogues of facets nor faces are currently known for tropical polyhedra, see the work of Develin and Yu~\cite{DevelinYu}. Minimal (inclusion-wise) tropical half-spaces containing a given tropical polyhedral cone have been studied by Joswig~\cite{joswig04}, Block and Yu~\cite{blockyu06}, and Gaubert and Katz~\cite{GK09}. They can be seen as the tropical counterparts of supporting half-spaces. 
In~\cite{GK09}, the authors show the surprising result that there can exist an infinite number of minimal half-spaces, as soon as $n \geq 4$. In a joint work with Allamigeon~\cite{AGK-10}, they have provided a characterization of the extreme vectors of the polars of tropical polyhedral cones. They pointed out that these vectors do not provide in general minimal representations by half-spaces. They have also introduced a method to eliminate redundant half-spaces via a reduction to a problem in game theory (solving mean payoff games). In particular, it has been observed that the greedy elimination of superfluous half-spaces produces different non-redundant representations, depending on the order in which the half-spaces are considered. As far as we know, the question whether there exists some structure behind the different non-redundant representations of a tropical polyhedral cone has remained open. 

The purpose of this work is to study the 
\myemph{minimal external representations}, 
also called \myemph{non-redundant external representations}, 
of a real polyhedral cone $\CC \subset \projspace$. Without loss of generality, we  restrict our attention to external representations composed of tropical half-spaces 
with apices located in $\CC$. Indeed, as shown in Section~\ref{subsec:saturation}, we can always replace any half-space containing $\CC$ by a smaller half-space (inclusion-wise), whose apex belongs to $\CC$. In particular, the apices of minimal 
half-spaces are elements of $\CC$. 

Under this assumption on the apices, we show that a real polyhedral cone has an essentially unique non-redundant external representation. More precisely, we prove the following theorem.  

\begin{theorem}\label{Th-Main}
For each real polyhedral cone $\CC \subset \projspace$ there exist a subset $\mcA$ of $\CC$, and for each $a\in \mcA$, a collection $\mcC_{a}$ of disjoint sets of tropical half-spaces with apex $a$, satisfying the following property:

A finite set of tropical half-spaces with apices in
$\CC$ is a non-redundant external representation of $\CC$ if, 
and only if, it is composed of precisely one
half-space in each set of the collection $\mcC_{a}$ for each $a\in \mcA$.
\end{theorem}

As a consequence of this result, any non-redundant external representation 
(composed of finitely many tropical half-spaces with apices in the cone) 
precisely involves the same set $\mcA$ of apices. They are referred to as \myemph{non-redundant apices}. 
Moreover, 
the multiplicity of each non-redundant apex $a$ 
(\ie , the number of half-spaces with this apex) 
is identical in any non-redundant external representation. It is equal to the cardinality of the collection $\mcC_a$. Theorem~\ref{Th-MaximalSCC} below specifies that the sets in the collection $\mcC_{a}$ are given by some strongly connected components of a certain directed graph (see Section~\ref{subsec:non_redundant_with_same_apex} for details). Consequently, two non-redundant external representations only differ on the choice of the representative of each strongly connected component.

This paper is organized as follows. The next section is devoted to recalling 
basic notions and results concerning tropical convexity. 
Moreover, given an 
external representation of a real polyhedral cone, we present a method to replace its half-spaces by half-spaces with apices in the cone (Proposition~\ref{Prop-Saturation}). Note that this method handles arbitrary half-spaces, including degenerate ones.

In Section~\ref{sec:criterion} we establish a combinatorial criterion to determine whether a half-space $\HH$ is redundant in a given external representation of a tropical polyhedral cone $\CC$. It is expressed as a certain reachability problem in a directed hypergraph, and plays a fundamental role in the subsequent results. It applies to the case where the half-space $\HH$ is non-degenerate, and its apex belongs to $\CC$. In contrast, the half-spaces in the external representation of $\CC$ can be arbitrary.

The main result of this paper, Theorem~\ref{Th-Main} above, 
is proved in Section~\ref{SectionNon-redundantRepre}. Thus, 
in this section we consider only non-degenerate half-spaces containing 
a fixed real polyhedral cone $\CC$, and whose apices belong to $\CC$. The proof consists of several steps. Firstly, we show an anti-exchange result which applies to half-spaces with distinct apices (Theorem~\ref{prop:partial-anti-exchange}). Secondly, we prove that the set of apices arising in non-redundant external representations is always equal to a certain set $\mcA$ (Theorem~\ref{th:non_redundant_apices}). Finally, we fix an apex $a \in \mcA$, and study which half-spaces with apex $a$ appear in non-redundant representations. This leads to the characterization of the collections $\mcC_a$ (Theorem~\ref{Th-MaximalSCC}).

Section~\ref{sec:cell_decomposition} studies the relationship between non-redundant apices and vertices of the cell complex associated with the cone. Theorem~\ref{Theorem-Nonredundant-Distinguised} establishes that all the non-redundant apices belong to 
a particular subset of vertices. We then provide a sufficient condition for a vertex 
in this subset to be a non-redundant apex of the cone (Theorem~\ref{th:sufficient_condition_non_redundant_apices}). Finally, we show (Theorem~\ref{th:generic_extremities}) that this sufficient condition is always satisfied when the cone has ``generic extremities'', meaning that each of its extreme vectors belongs to a closed ball of positive radius (for the tropical projective Hilbert metric) contained in the cone.

\section{Preliminaries}\label{SectionPreliminaries}

\subsection{Basic notions in tropical convexity.}\label{SectionBasicNotions}  

Henceforth, we will use concatenation $xy$ to denote tropical multiplication $x\otimes y$ of two scalars $x,y \in \maxplus$. When $x,y$ are vectors of $\maxplus^n$, 
$x y$ 
represents the tropical inner product of $x$ and $y$, 
\ie 
\[
x y\mydef \mpplus_{i\in \oneto{n}}x_i y_i \; .
\]
To emphasize the semiring structure of $\maxplus$, we denote by $\mpzero$ the neutral element for addition, \ie~$\mpzero \mydef -\infty$, and by $\mpone$ the neutral element for multiplication, \ie~$\mpone \mydef 0$. 
The $i$th (tropical) unit vector will be denoted by $e^i$, 
\ie~$e^i\in \maxplus^n$ is the vector defined by 
\[
e^i_i\mydef \mpone \; \makebox{ and }\; e^i_j\mydef \mpzero \; \makebox{ for }\; j\neq i \;. 
\]
The multiplicative inverse of a non-zero (in the tropical sense) scalar $\lambda \in \maxplus$, \ie~ $-\lambda$, will be represented by $\mpinverse{\lambda}$. When $x \in \projspace$, we denote by $\vectinverse{x}$ the vector whose coordinates are $\mpinverse{x_i}$. 
Given $I\subset \oneto{n}$, the vector $\vectinverse{x}_I$ is defined by 
\[
(\vectinverse{x}_I)_i \mydef \vectinverse{x}_i \; \makebox{ if  } \; i\in I \; \makebox{ and } \; (\vectinverse{x}_I)_i \mydef \mpzero \; \makebox{ otherwise}. 
\]

The identification of a real polyhedral cone with its image in the real projective space $\projspace$ can be generalized to any tropical cone $\CC \subset \maxplus^n$ provided that we consider the \myemph{tropical projective space}
\[
\projmaxspace \mydef \bigl(\maxplus^n \setminus \{ (-\infty, \dots, -\infty)\} \bigr) / (1, \dots, 1)\R \; .  
\]  
We define the \myemph{tropical projective Hilbert metric} over $\projspace$ by:
\[ 
d_H(x,y)\mydef \max_{i \in \oneto{n}} (x_i - y_i) - \min_{i\in \oneto{n}} (x_i - y_i) \; .
\]
It can be extended to $\projmaxspace$ by setting 
\[
d_H(x,y) \mydef
\begin{cases}
\max_{y_i \neq \mpzero} (x_i - y_i) - \min_{y_i \neq \mpzero} (x_i - y_i)& \text{when} \ \{ i \mid x_i \neq \mpzero \} = \{ i \mid y_i \neq \mpzero \}, \\
+\infty & \text{otherwise.}
\end{cases}
\]
The sets $\projspace$ and $\projmaxspace$ are both endowed with the topology induced by the metric $d_H$. In the sequel, closed balls for this metric will be called \myemph{(closed) Hilbert balls}. 

\subsubsection*{Extreme vectors of tropical cones.}

A  (non-zero) vector $x$ of a tropical cone $\CC \subset \projmaxspace$ 
is said to be \myemph{extreme} in $\CC$ if for all $y, z \in \CC$,
\[
x = y \mpplus z \ \text{implies either} \ x=y  \ \text{or}\  x=z  . 
\] 
The tropical version of Minkowski theorem~\cite{GK06a,GK} in the case of cones 
shows that a tropical polyhedral cone $\CC$ is generated 
by a set $V \subset \projmaxspace$ if, and only if, 
$V$ contains the extreme vectors of $\CC$. Thus, 
a tropical polyhedral cone $\CC$ has a unique minimal generating set 
(as a subset of $\projmaxspace$).

\subsubsection*{Tropical half-spaces and hyperplanes.}

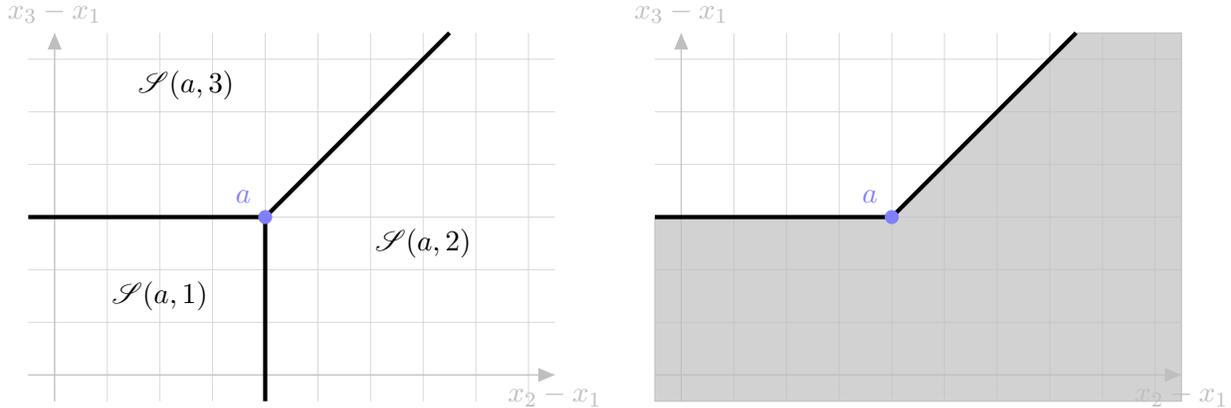
\begin{figure}
\begin{center}
\begin{tikzpicture}[convex/.style={draw=lightgray,fill=lightgray,fill opacity=0.7},convexborder/.style={ultra thick},point/.style={blue!50},>=triangle 45,scale=0.7]
\begin{scope}
\draw[gray!30,very thin] (-0.5,-0.5) grid (9.5,6.5);
\draw[gray!50,->] (-0.5,0) -- (9.5,0) node[color=gray!50,below] {$x_2-x_1$};
\draw[gray!50,->] (0,-0.5) -- (0,6.5) node[color=gray!50,above] {$x_3-x_1$};

\node[coordinate] (a) at (4,3) {};
\node[coordinate] (v1) at (-0.5,3) {};
\node[coordinate] (v2) at (4,-0.5) {};
\node[coordinate] (v3) at (7.5,6.5) {};
\draw[convexborder] (a) -- (v1) (a) -- (v2) (a) -- (v3);
\filldraw[point] (a) circle (3.5pt) node[above left=1.5pt] {$a$};
\node at (2,1.5) {$\SS(a,1)$};
\node at (7,2.5) {$\SS(a,2)$};
\node at (2.5,5.5) {$\SS(a,3)$};
\end{scope}
\begin{scope}[xshift=11.9cm]
\draw[gray!30,very thin] (-0.5,-0.5) grid (9.5,6.5);
\draw[gray!50,->] (-0.5,0) -- (9.5,0) node[color=gray!50,below] {$x_2-x_1$};
\draw[gray!50,->] (0,-0.5) -- (0,6.5) node[color=gray!50,above] {$x_3-x_1$};

\node[coordinate] (a) at (4,3) {};
\node[coordinate] (v1) at (-0.5,3) {};
\node[coordinate] (c1) at (-0.5,-0.5) {};
\node[coordinate] (c2) at (9.5,-0.5) {};
\node[coordinate] (c3) at (9.5,6.5) {};
\node[coordinate] (v3) at (7.5,6.5) {};
\filldraw[convex] (a) -- (v1) -- (c1) -- (c2) -- (c3) -- (v3) -- cycle;	
\draw[convexborder] (v1) -- (a) -- (v3);
\filldraw[point] (a) circle (3.5pt) node[above left=1.5pt] {$a$};
\end{scope}
\end{tikzpicture}
\end{center}
\caption{A tropical hyperplane with the corresponding sectors (left) and a tropical half-space (right), both of them with apex $a=(0,4,3)$.}
\label{fig:halfspace}
\end{figure}

With the notation introduced above, observe that any non-degenerate 
half-space can be written in the form
\begin{equation}
\HH = \{ x \in \projmaxspace \mid \vectinverse{a}_{I} x \geq \vectinverse{a}_{\compl{I}} x \} \; , \label{EqHalfSpace}
\end{equation}
where $a \in \projspace$ and $I$ is a non-empty proper subset of $[n]$. 
In what follows, we shall also shortly denote such half-space by 
$\HH(a,I)$. 

Non-degenerate half-spaces are related to the notion of tropical hyperplanes. The \myemph{(max-plus) tropical hyperplane with apex $a$} is defined as the set of vectors $x \in \projmaxspace$ such that the maximum 
\[
\vectinverse{a} x = \max \{x_1 - a_1, \dots, x_n - a_n\}
\]
is attained at least twice. The complement of such hyperplane is the disjoint union of $n$ regions, 
the topological closure of which 
\[
\SS(a,i)\mydef \left\{x \in \projmaxspace \mid \vectinverse{a_i} x_i \geq  \vectinverse{a_j} x_j \makebox{ for all }
j \in \oneto{n} \right\} \; , 
\]   
are special tropical half-spaces called (closed) \myemph{sectors}, 
see the left-hand side of Figure~\ref{fig:halfspace}. Note that the half-space $\HH$ in~\eqref{EqHalfSpace} coincides with the union, for $i \in I$, of the sectors $\SS(a,i)$ supported by the hyperplane with apex $a$. This is illustrated in the right-hand side of Figure~\ref{fig:halfspace}.
Besides, observe that the apex $a$ and the set of sectors $I$ are both uniquely determined by $\HH$, and so they will be denoted by $\apex(\HH)$ and $\sect(\HH)$ respectively. 
We refer the reader to~\cite{joswig04} for more information on hyperplanes and half-spaces, but we warn that the results of~\cite{joswig04} are in the setting of the (real) min-plus semiring $(\R , \min ,+)$, which is however equivalent to the setting considered here. 

\subsubsection*{Cell decomposition.}

We now recall basic definitions and properties concerning the natural cell decomposition of $\projspace$ induced by a finite set of vectors $\{v^r\}_{r\in \oneto{p}} \subset \projspace$. For a complete presentation in the equivalent setting of the (real) min-plus semiring 
$(\R , \min ,+)$, we refer the reader to~\cite{DS}. 

Given $x \in \projspace$, the \myemph{type} of $x$ relative to $\{v^r\}_{r\in \oneto{p}}$ is the $n$-tuple $\type(x)=(S_1(x),\ldots , S_n(x))$ of subsets of $\oneto{p}$ defined as follows: 
\[
S_j(x) \mydef \left\{r \in \oneto{p} \mid \vectinverse{x}_j v^r_j \geq 
\vectinverse{x}_i v^r_i \text{ for all } i \in \oneto{n}\right\} \; , 
\]
for $j\in \oneto{n}$.  An $n$-tuple $(S_1,\ldots , S_n)$ 
of subsets of $\oneto{p}$ is said to be a \myemph{type} if it arises in this way. 

With each $n$-tuple $S=(S_1,\ldots ,S_n)$ of subsets of $\oneto{p}$, 
it can be associated the set $X_S$ of all the vectors whose type contains $S$, 
\ie 
\[
X_S\mydef \left\{ x\in \projspace \mid S_j\subset S_j(x),\text{ for all } j\in \oneto{n} \right\} \; .
\] 
Lemma~10 of~\cite{DS} shows that these sets are given by  
\[
X_S=\left\{ x\in \projspace \mid x_j v^r_i\leq x_iv^r_j ,\text{ for all } i,j\in \oneto{n} 
\makebox{ and }r\in S_j \right\} \; , 
\] 
and so they are both closed convex polyhedra (in the usual sense) and tropical polyhedral cones. The \myemph{natural cell decomposition of $\projective^{n-1}$ induced 
by $\{v^r\}_{r\in \oneto{p}}$} is defined as the collection of convex polyhedra 
$X_S$, where $S$ ranges over all the possible types. 

A simple geometric construction of the natural cell decomposition of 
$\projspace$ induced by $\{v^r\}_{r\in \oneto{p}}$ can be obtained if we consider the min-plus hyperplanes whose apices are these vectors. 
Recall that given $a \in \projspace$, the \myemph{min-plus hyperplane} with apex $a$ is the set of vectors $x$ such that the minimum   
\[
\min \{x_1 - a_1,\dots,x_n - a_n\}
\] 
is attained at least twice.  
By Proposition~16 of~\cite{DS}, the cell decomposition induced 
by $\{v^r\}_{r\in \oneto{p}}$ is the common refinement of the fans defined by  
the $p$ min-plus hyperplanes whose apices are the vectors $v^r$, 
for $r\in \oneto{p}$. 

Given a cell $X_S$, if we define the undirected graph $G_S$ with set of nodes 
$\oneto{n}$ and an arc connecting nodes $i$ and $j$ 
if, and only if, $S_i\cap S_j\neq \emptyset$, 
then by Proposition~17 of~\cite{DS} the dimension of $X_S$ 
(in the projective space) is one less than 
the number of connected components of $G_S$. 
A zero-dimensional cell is called a \myemph{vertex} 
of the natural cell decomposition. 

When $\CC$ is the tropical cone generated by $\{v^r\}_{r\in \oneto{p}}$, 
the natural cell decomposition of $\projspace$ induced by 
$\{v^r\}_{r\in \oneto{p}}$ has in particular the property that $\CC$ 
is the union of its bounded cells, 
see~\cite{DS} for details. Corollary~12 of~\cite{DS} also shows that a cell $X_S$ is bounded if, and only if, $S_j \neq \emptyset$ for all $j \in [n]$. 
It follows that $x\in \CC$ if, and only if, 
$S_j(x) \neq \emptyset$ for all $j \in [n]$. 

\begin{example}
The natural cell decomposition of 
$\projspace[2]$ induced by $v^1=(0,1,3)$, $v^2=(0,4,1)$ and $v^3=(0,9,4)$ 
is illustrated in Figure~\ref{FigCellDecomp}.  
As explained above, it can be obtained by drawing three 
min-plus hyperplanes (dotted lines in Figure~\ref{FigCellDecomp}) 
whose apices are the vectors $v^1$, $v^2$ and $v^3$. 
This cell decomposition consists of six zero-dimensional cells 
(vertices), fifteen one-dimensional cells 
(nine unbounded and six bounded) 
and ten two-dimensional cells 
(nine unbounded and only one bounded). 

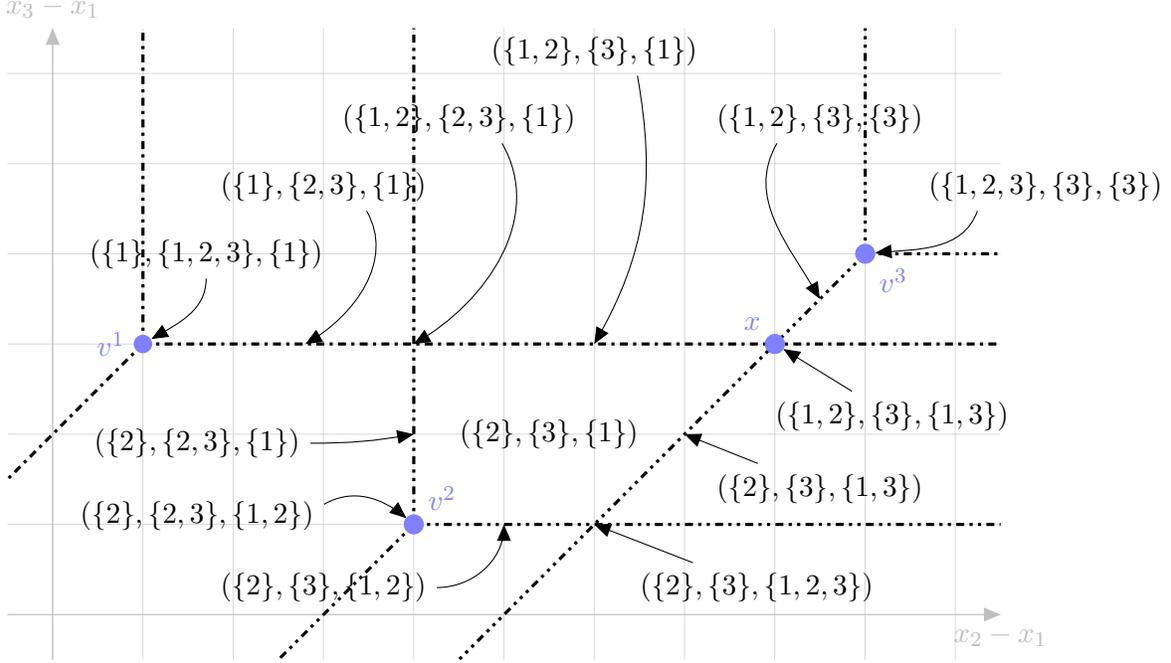
\begin{figure}
\begin{center}
\begin{tikzpicture}	[convex/.style={draw=lightgray,fill=lightgray,fill opacity=0.7},convexborder/.style={very thick},point/.style={blue!50},>=triangle 45,scale=1.2]
\draw[gray!30,very thin] (-0.5,-0.5) grid (10.5,6.5);
\draw[gray!50,->] (-0.5,0) -- (10.5,0) node[color=gray!50,below] {$x_2-x_1$};
\draw[gray!50,->] (0,-0.5) -- (0,6.5) node[color=gray!50,above] {$x_3-x_1$};

\coordinate (v1) at (1,3);
\coordinate (v1_1) at (-0.5,1.5);
\coordinate (v1_2) at (1,6.5);
\coordinate (v1_3) at (10.5,3);

\coordinate (v2) at (4,1);
\coordinate (v2_1) at (2.5,-0.5);
\coordinate (v2_2) at (4,6.5);
\coordinate (v2_3) at (10.5,1);

\coordinate (v3) at (9,4);
\coordinate (v3_1) at (4.5,-0.5);
\coordinate (v3_2) at (9,6.5);
\coordinate (v3_3) at (10.5,4);

\coordinate (x) at (8,3);

\draw[convexborder,dashdotted] (v1) -- (v1_1) (v1) -- (v1_2) (v1) -- (v1_3);
\path[point] (v1) node[left=3pt] {$v^1$};
\node[fill,point,circle,minimum size=7pt,inner sep=0pt] (v1') at (v1) {};

\draw[convexborder,dashdotdotted] (v2) -- (v2_1) (v2) -- (v2_2) (v2) -- (v2_3);
\filldraw[point] (v2) circle (3pt) node[above right=1.5pt] {$v^2$};
\node[fill,point,circle,minimum size=7pt,inner sep=0pt] (v2') at (v2) {};

\draw[convexborder,dashdotdotdotted] (v3) -- (v3_1) (v3) -- (v3_2) (v3) -- (v3_3);
\filldraw[point] (v3) circle (3pt) node[below right=1.5pt] {$v^3$};
\node[fill,point,circle,minimum size=7pt,inner sep=0pt] (v3') at (v3) {};

\filldraw[point] (x) circle (3pt) node[above left=1.5pt] {$x$};
\node[fill,point,circle,minimum size=7pt,inner sep=0pt] (x') at (x) {};

\node (t1_123_1) at (1.7,4) {$(\{1\},\{1,2,3\},\{1\})$};
\draw[->] (t1_123_1) to[out=-90,in=30] (v1');

\node (t1_23_1) at (3,4.75) {$(\{1\},\{2,3\},\{1\})$};
\draw[->] (t1_23_1.330) to[out=-60,in=30] (2.8,3);

\node (t12_23_1) at (4.5,5.5) {$(\{1,2\},\{2,3\},\{1\})$};
\draw[->] (t12_23_1.330) to[out=-60,in=30] (4,3);

\node (t12_3_1) at (6,6.25) {$(\{1,2\},\{3\},\{1\})$};
\draw[->] (t12_3_1.330) to[out=-80,in=60] (6,3);

\node (t12_3_3) at (8.5,5.5) {$(\{1,2\},\{3\},\{3\})$};
\draw[->] (t12_3_3.210) to[out=-120,in=120] (8.5,3.5);

\node (t123_3_3) at (11,4.75) {$(\{1,2,3\},\{3\},\{3\})$};
\draw[->] (t123_3_3.200) to[out=-110,in=5] (v3');

\node (t2_3_1) at (5.5,2) {$(\{2\},\{3\},\{1\})$};

\node (t2_23_1) at (1.6,1.9) {$(\{2\},\{2,3\},\{1\})$};
\draw[->] (t2_23_1) to[out=0,in=-170] (4,2);

\node (t2_23_12) at (1.6,1.1) {$(\{2\},\{2,3\},\{1,2\})$};
\draw[->] (t2_23_12.5) to[out=30,in=140] (v2');

\node (t2_3_12) at (3,0.3) {$(\{2\},\{3\},\{1,2\})$};
\draw[->] (t2_3_12) to[out=0,in=-90] (5,1);

\node (t2_3_123) at (7.8,0.3) {$(\{2\},\{3\},\{1,2,3\})$};
\draw[->] (t2_3_123) -- (6,1);

\node (t2_3_13) at (8.5,1.4) {$(\{2\},\{3\},\{1,3\})$};
\draw[->] (t2_3_13) -- (7,2);

\node (t12_3_13) at (9.3,2.2) {$(\{1,2\},\{3\},\{1,3\})$};
\draw[->] (t12_3_13) -- (x');
\end{tikzpicture}
\end{center}
\caption{The natural cell decomposition of 
$\projspace[2]$ induced by $v^1=(0,1,3)$, $v^2=(0,4,1)$ and $v^3=(0,9,4)$.}
\label{FigCellDecomp}
\end{figure}

Figure~\ref{FigCellDecomp} also provides the type 
(relative to $v^1$, $v^2$ and $v^3$) of 
any vector in the relative interior of each bounded cell. For instance, 
the type of $x=(0,8,3)$ is $(\{1,2\},\{ 3 \},\{ 1,3\})$, 
and so this vector is a vertex (the undirected graph 
$G_S$, where $S=\type(x)$, is connected). 
The line segment joining 
$x$ with $v^3$ is the cell $X_S$ for 
$S=(\{1,2\},\{ 3 \},\{3\})$, and the only bounded 
two-dimensional cell is $X_S$ for 
$S=(\{2\},\{ 3 \},\{1\})$. 

Comparing Figures~\ref{FigTropicalCone} and~\ref{FigCellDecomp}, it can be seen that the tropical cone generated by $v^1$, $v^2$ and $v^3$ is precisely the union of the bounded cells in the natural cell decomposition of $\projspace[2]$ induced by these vectors. 
\end{example}

\subsection{Tropical polar cones.}

As in classical convex analysis, the \myemph{polar} $\polar{\CC}$ of a tropical cone $\CC \subset \projmaxspace$ can be defined~\cite{katz08} to represent the set of all (tropical) linear inequalities satisfied by the vectors of $\CC$:  
\begin{equation}
\polar{\CC}\mydef \left\{ (u,u') \in \projmax^{2n-1} \mid u x \geq u' x ,\; \forall x\in \CC \right\} \; .  
\end{equation}
However, note that tropical linear forms must be considered on both sides 
of the inequality due to the absence of a ``minus sign''. 
This means that the polar of $\CC$ is a tropical cone of $\projmax^{2n-1}$. 

As a consequence of the separation theorem for tropical cones 
of~\cite{zimmerman77,shpiz,cgqs04}, 
a tropical polyhedral cone $\CC$ is characterized by its polar cone, 
\ie
\[
\CC =\left\{x\in \projmaxspace \mid u x \geq u' x\; , \; \forall (u,u')\in \polar{\CC} \right\}  \; .
\]
Moreover, when $\CC$ is polyhedral, 
it can be shown that $\polar{\CC}$ is also polyhedral,  
implying $\CC$ is the intersection of the (finite) set of 
tropical half-spaces associated with the extreme vectors of $\polar{\CC}$. 

An equivalent notion to the polar is that of the $j$th polar, 
see~\cite{AGK-10}. For $j\in \oneto{n}$, the \myemph{$j$th polar} $\ipolar[j]{\CC}$ 
of $\CC$ is defined as the tropical cone
\begin{equation} 
\ipolar[j]{\CC} \mydef \left\{ u \in \projmaxspace \mid \oplus_{i\in\compl{\{j\}}} u_i x_i \geq u_j x_j \; ,\; \forall x \in \CC \right\} \; , 
\end{equation}
which lies in $\projmaxspace$. As in the case of the polar, 
a tropical polyhedral cone $\CC$ is given by the (finite) intersection of the 
tropical half-spaces associated with the extreme vectors of $\ipolar[j]{\CC}$, 
for $j \in \oneto{n}$. Indeed, the set of extreme vectors of $\polar{\CC}$ precisely consists of the vectors $(e^i, e^i)$ ($i \in [n]$) and the extreme vectors of the $j$th polars of $\CC$ ($j \in [n]$), see~\cite{AGK-10}.

The extreme vectors of the polars of tropical polyhedral cones 
have been characterized in different ways, 
see~
\cite[Theorem~5]{GK09} 
or~\cite[Theorem~3]{AGK-10}. We shall need the following 
variant of Theorem~3 of~\cite{AGK-10}, 
which is more adapted to our setting.   

\begin{theorem}\label{TheoCharacExtremeIPolar}
Let $\CC$ be a real polyhedral cone generated by the set $\{v^r\}_{r \in [p]} \subset \projspace$, and let $u\in \ipolar[j]{\CC}$ be such that $u_j\neq \mpzero$. 

Then, $u$ is extreme in $\ipolar[j]{\CC}$ if, and only if, 
for each $i \neq j$ either $u_i =\mpzero$ or there exists $r \in \oneto{p}$ such that $u_i v_i^r = u_j v^r_j > \oplus_{k \in \compl{\{i, j \}}} u_k v^r_k$.
\end{theorem}

Observe that the case $u_j=\mpzero$ is not considered in 
Theorem~\ref{TheoCharacExtremeIPolar}. This is due to the fact that 
$\ipolar[j]{\CC}$ contains the unit vectors $e^i$, for $i\neq j$, 
and so they are the only extreme vectors $u$ of $\ipolar[j]{\CC}$ satisfying 
$u_j=\mpzero$. These extreme vectors of $\ipolar[j]{\CC}$ will be called \myemph{trivial}, 
because they represent tautological inequalities $x_i \geq \mpzero$, and so they play no role in the external representation of $\CC$. 

\subsection{Saturation and minimal half-spaces.}\label{subsec:saturation}

Let $\CC$ be the real polyhedral cone generated by the set $\{v^r\}_{r \in [p]} \subset \projspace$. A half-space is said to be \myemph{minimal} with respect to $\CC$ if it is minimal for inclusion among the set of half-spaces containing $\CC$. Gaubert and Katz have proved in~\cite{GK09} that any minimal half-space with respect to $\CC$ is non-degenerate, and its apex can be characterized in terms of the natural cell decomposition of $\projspace$ induced by the generating set $\{v^r\}_{r \in [p]}$. 

\begin{theorem}[\cite{GK09}, Theorem~4]\label{Theo-Charac-Minimal}
The half-space $\HH(a,I)$ is minimal with respect to the real polyhedral cone $\CC$ if, and only if, the following conditions are satisfied:
\begin{enumerate}[label={\normalfont (C\arabic*)}]
\item\label{item:C1} $\cup_{i \in I} S_i(a) = \oneto{p}$,
\item\label{item:C2} for each $j \in \compl{I}$ there exists $i \in I$ such that $S_i(a) \cap S_j(a) \neq \emptyset$,
\item\label{item:C3minimal} for each $i \in I$ there exists $j \in \compl{I}$ such that $S_i(a) \cap S_j(a) \not \subset \cup_{k \in I \setminus \left\{ i \right\}} S_k(a)$.
\end{enumerate}
where $(S_1(a),\dots,S_n(a))=\type(a)$ is the type of $a$ relative to 
the generating set $\{v^r\}_{r \in [p]}$. 
\end{theorem}

The apices of minimal half-spaces consequently form certain cells of the 
natural cell decomposition of $\projspace$ induced by the generators of $\CC$. 
It was shown in~\cite{GK09} that these cells need not be zero-dimensional, so the number of apices of minimal half-spaces can be infinite. Since Conditions~\ref{item:C2} and~\ref{item:C3minimal} above imply  $S_h(a) \neq \emptyset$ for all $h \in [n]$, we readily obtain the following corollary: 

\begin{corollary}\label{Coro-Minimal-Apex}
If $\HH$ is a minimal half-space with respect to the real polyhedral cone $\CC$, then its apex belongs to $\CC$.
\end{corollary}

\begin{remark}\label{remark:minimal}
The three conditions of Theorem~\ref{Theo-Charac-Minimal} do not depend on the choice of the generating set of $\CC$. For instance, Condition~\ref{item:C1} amounts to $\CC \subset \HH(a,I)$. Similarly, assuming $\CC \subset \HH(a,I)$, Condition~\ref{item:C2} is equivalent to the fact that, for each $j \in \compl{I}$, there exists $x \in \CC$ such that $\vectinverse{a_j} x_j = \mpplus_{k \in I } \vectinverse{a_k} x_k$. Observe that the latter is trivially satisfied when $a \in \CC$. 
\end{remark}

Given a (possibly degenerate) half-space $\HH$ containing $\CC$, there always exists a minimal half-space $\HH'$ such that $\CC \subset \HH' \subset \HH$, see~\cite[Theorem~3]{GK09}. 
Using Corollary~\ref{Coro-Minimal-Apex} and the fact that $\CC$ is a finite intersection of tropical half-spaces by (the conic form) of the tropical Minkowski-Weyl theorem~\cite{GK09}, we conclude that $\CC$ is a finite intersection of half-spaces with apices in $\CC$, and these half-spaces can be assumed to be minimal. 

Since Theorem~3 of~\cite{GK09} is not constructive, in this section we explain a simple method, referred to as \myemph{saturation}, to compute a half-space $\HH'$ satisfying $\apex(\HH')\in \CC$ and $\CC \subset \HH' \subset \HH$.
Suppose that $\HH = \{ x \in \projmaxspace \mid \mpplus_{i \in I} \mpinverse{\alpha_i} x_i \geq \mpplus_{j \in J} \mpinverse{\alpha_j} x_j\}$, where $I$ and $J$ are disjoint non-empty subsets of $\oneto{n}$ and $\alpha_h \in \R$ for all $h \in I\cup J$. Consider the half-space $\HH(b,I')$ whose apex $b = (\beta_1,\ldots ,\beta_n) \in \projspace$ and sectors $I'$ are defined as follows: 
\[
\beta_i \mydef \mpplus_{r \in \oneto{p}} \lambda_r v^r_i \quad \text{for all } i \in [n], \qquad I' \mydef \{ i \in I \mid \alpha_i = \beta_i\} \; ,
\]
with $\lambda_r \in \R$ being defined by
$\lambda_r \mydef \mpinverse{(\mpplus_{h \in I \cup J} \mpinverse{\alpha_h} v^r_h)}$. 
Then, the following proposition holds:
\begin{proposition}\label{Prop-Saturation}
The half-space $\HH(b,I')$ satisfies the following properties:
\begin{enumerate}[label=(\roman*)]
\item its apex $b$ belongs to $\CC$;
\item $\CC \subset \HH(b,I') \subset \HH$.
\end{enumerate}
\end{proposition}

\begin{proof}
The first property readily follows from $b = \mpplus_{r \in \oneto{p}} \lambda_r v^r$.

On the other hand, since 
$\mpinverse{\alpha_k} v_k^r \leq \mpplus_{h \in I\cup J} \mpinverse{\alpha_h} v^r_h = 
\mpinverse{\lambda_r}$ for all $k \in I\cup J$ and $r\in \oneto{p}$, it follows that  
$\beta_k = \mpplus_{r \in \oneto{p}} \lambda_r v_k^r \leq \alpha_k$ for all $k\in I\cup J$. 
Moreover, note that $\alpha_k = \beta_k$ if, and only if, there exists $r \in \oneto{p}$ such that $\mpinverse{\alpha_k} v^r_k = \mpplus_{h \in I\cup J} \mpinverse{\alpha_h} v^r_h$. 

Consider now any $s\in \oneto{p}$. Firstly, observe that there exists $i \in I$ such that 
$\mpinverse{\alpha_i} v^s_i = \mpplus_{h \in I\cup J} \mpinverse{\alpha_h} v^s_h$, 
because $v^s \in \CC \subset \HH$. Since in that case we have $\alpha_i=\beta_i$ by the discussion above, it follows that $i \in I'$ and so 
\[
\mpplus_{h\in I\cup J} \mpinverse{\alpha_h} v^s_h = 
\mpplus_{i \in I'} \mpinverse{\alpha_i} v^s_i \; .
\]
Now note that for any $j \in \compl{I'}$ we have
\[
\mpinverse{\beta_j} v^s_j = \mpinverse{(\mpplus_{r \in \oneto{p}} \lambda_r v_j^r)} v^s_j  
\leq \mpinverse{\lambda_s}  = \mpplus_{h\in I\cup J} \mpinverse{\alpha_h} v^s_h = 
\mpplus_{i\in I'} \mpinverse{\alpha_i} v^s_i = \mpplus_{i \in I'} \mpinverse{\beta_i} v^s_i \; ,
\]
and thus $v^s\in \HH(b,I')$. Since this holds for any $s\in \oneto{p}$, 
we conclude that $\CC \subset \HH(b,I')$.

Finally, if we assume $\mpplus_{i \in I'} \mpinverse{\beta_i} x_i \geq \mpplus_{j \in \compl{I'}} \mpinverse{\beta_j} x_j$, it follows that   
\[
\mpplus_{i \in I} \mpinverse{\alpha_i} x_i \geq 
\mpplus_{i \in I'} \mpinverse{\alpha_i} x_i =
\mpplus_{i \in I'} \mpinverse{\beta_i} x_i \geq 
\mpplus_{j \in \compl{I'}} \mpinverse{\beta_j} x_j \geq 
\mpplus_{j \in J} \mpinverse{\beta_j} x_j \geq
\mpplus_{j \in J} \mpinverse{\alpha_j} x_j \; ,
\]
because $I'\subset I$, $J \subset \compl{I'}$ and 
$\beta_h \leq \alpha_h$ for all $h \in I\cup J$, where the equality holds for $h \in I'$. Then, we conclude that $\HH(b,I') \subset \HH$. 
\end{proof}

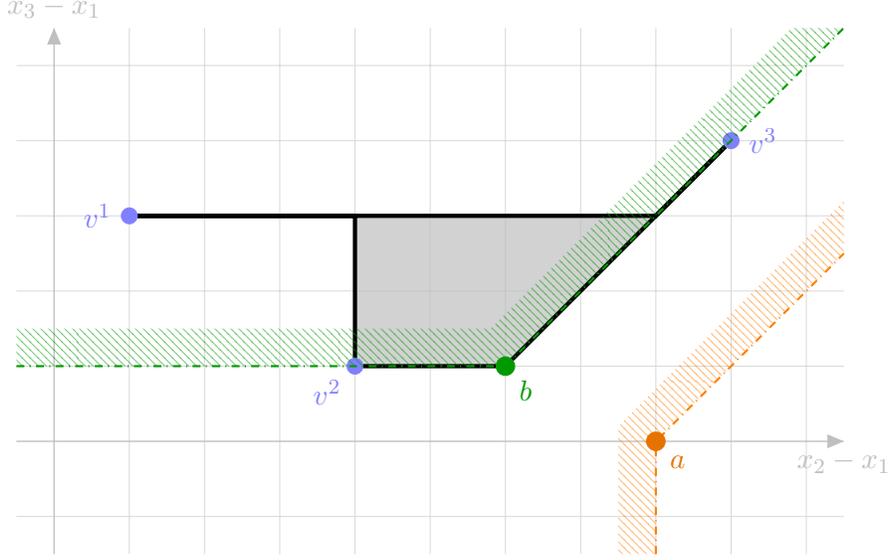
\begin{figure}
\begin{center}
\begin{tikzpicture}[convex/.style={draw=black,ultra thick,fill=lightgray,fill opacity=0.7},
point/.style={blue!50},>=triangle 45,hsborder/.style={orange,dashdotted,thick},
hs/.style={draw=none,pattern=north west lines,pattern color=orange,fill opacity=0.5},
apex/.style={orange!90!black},
shsborder/.style={green!60!black,dashdotted,thick},
shs/.style={draw=none,pattern=north west lines,pattern color=green!60!black,fill opacity=0.5},
sapex/.style={green!60!black},]
\draw[gray!30,very thin] (-0.5,-1.5) grid (10.5,5.5);
\draw[gray!50,->] (-0.5,0) -- (10.5,0) node[color=gray!50,below] {$x_2-x_1$};
\draw[gray!50,->] (0,-1.5) -- (0,5.5) node[color=gray!50,above] {$x_3-x_1$};

\node[coordinate] (v1) at (1,3) {};
\node[coordinate] (v12) at (4,3) {};
\node[coordinate] (v2) at (4,1) {};
\node[coordinate] (v23) at (6,1) {};
\node[coordinate] (v3) at (9,4) {};
\node[coordinate] (v31) at (8,3) {};

\filldraw[convex] (v1) -- (v12) -- (v2) -- (v23) -- (v3) -- (v31) -- cycle;
\filldraw[point] (v1) circle (3pt) node[left=3pt] {$v^1$};
\filldraw[point] (v2) circle (3pt) node[below left=1.5pt] {$v^2$};
\filldraw[point] (v3) circle (3pt) node[right=3pt] {$v^3$};

\draw[hsborder] (8,-1.5) -- (8,0) -- (10.5,2.5);
\filldraw[hs] (8,-1.5) -- (8,0) -- (10.5,2.5) -- (10.5,3.2) -- (7.5,0.2) -- (7.5,-1.5) -- cycle;
\filldraw[apex] (8,0) circle (3.5pt) node[below right=1.5pt] {$a$};

\draw[shsborder] (-0.5,1) -- (6,1) -- (10.5,5.5);
\filldraw[shs] (-0.5,1) -- (6,1) -- (10.5,5.5) -- (9.8,5.5) -- (5.8,1.5) -- (-0.5,1.5) -- cycle;
\filldraw[sapex] (6,1) circle (3.5pt) node[below right=1.5pt] {$b$};
\end{tikzpicture}	
\end{center}
\caption{Saturation of a half-space.}\label{fig:saturation}
\end{figure}

\begin{example}
Consider the cone of Figure~\ref{FigTropicalCone}, and the half-space $\{ x \in \projmaxspace[2] \mid x_1 \mpplus x_3 \geq (-8) x_2 \}$ with apex $a = (0,8,0)$, depicted in orange in Figure~\ref{fig:saturation}. It can be verified that $\lambda_1 = -3$, $\lambda_2 = -1$, and $\lambda_3 = -4$, thus $\beta_1 = (-3) v^1_1 \mpplus (-1) v^2_1 \mpplus (-4) v^3_1 = -1$. Similarly, $\beta_2 = 5$ and $\beta_3 = 0$, so that $b = (-1,5,0)$, and $I' = \{3\}$. The half-space $\HH(b,I')$ is represented in green in Figure~\ref{fig:saturation}. 
\end{example}

\begin{remark}
Note that when $I \cup J = \oneto{n}$, we have $\lambda_r=\max\{\lambda \in \maxplus \mid \lambda v^r \leq a\}$, where $a\mydef (\alpha_1,\ldots ,\alpha_n)$ is the apex of $\HH$ (here $\leq$ refers to the component-wise comparison over vectors of $\R^n$). Then, in that case, the apex $b$ can be seen as the projection of the vector $a$ onto the cone $\CC$. This projection is known to minimize the tropical projective Hilbert metric, \ie~for all $x \in \CC$, $d_H(a,b) \leq d_H(a,x)$, see~\cite{cgq00,cgq02} for details. 
\end{remark}

In general, the half-space $\HH(b,I')$ is not minimal with respect to $\CC$. However, we next show that $\HH(b,I')$ is minimal in the important special case where $\HH$ is the half-space associated with a non-trivial extreme vector of the $j$th polar of $\CC$ ($j \in [n]$).

\begin{proposition}\label{PropExtremeSaturationMinimal}
Let $\HH = \{ x \in \projmax^{n-1} \mid \mpplus_{i\in \compl{\{j\}}} u_i x_i \geq u_j x_j\}$ be the half-space associated with a non-trivial extreme vector $u$ of the $j$th polar of $\CC$. The half-space obtained by saturation of $\HH$ is minimal respect to $\CC$, and is of the form $\HH(b,I)$ with $b \in \projspace$ satisfying $u = \vectinverse{b_I} \mpplus \vectinverse{b_j} e_j$.
\end{proposition}

\begin{proof} 
Let $\HH(b,I')$ be the half-space obtained by saturation of $\HH$. 
Note that using the notation of Proposition~\ref{Prop-Saturation}, we have $\HH = \{ x\in \projmaxspace \mid \mpplus_{i \in I} \mpinverse{\alpha_i} x_i \geq \mpplus_{j \in J} \mpinverse{\alpha_j} x_j\}$ where $J=\{j\}$, $I=\{i \in \oneto{n} \mid  u_i \neq \mpzero, i\neq j\}$ and $\alpha_h=\mpinverse{u_h}$ for $h\in I\cup J= I\cup \{j\}$.

By Theorem~\ref{TheoCharacExtremeIPolar}, for each $i \in I$ there exists $r \in \oneto{p}$ such that 
\begin{equation}
u_i v^r_i = u_j v^r_j > \mpplus_{k \in I \setminus \{i\}} u_k v^r_k \;. \label{EqPropExtremeSaturationMinimal}
\end{equation}
As we have seen in the proof of Proposition~\ref{Prop-Saturation}, this implies $\beta_h = \alpha_h=\mpinverse{u_h}$ for all $h \in I \cup \{j\}$, and so in particular $I' = I$. Besides, by~\eqref{EqPropExtremeSaturationMinimal}, we obtain that 
\[
\vectinverse{b_i} v^r_i = \vectinverse{b_j} v^r_j > \mpplus_{k \in I \setminus \{i\}} \vectinverse{b_k} v^r_k \;.
\]
It follows that both $\vectinverse{b_i} v^r_i$ and $\vectinverse{b_j} v^r_j$ are maximal among the $\vectinverse{b_h} v^r_h$ for $h \in I \cup \{j \}$, and even among the $\vectinverse{b_h} v^r_h$ for $h \in [n]$, since $\vectinverse{b_{I'}} v^r \geq \vectinverse{b_{\compl{I'}}} v^r$ and $I = I'$. Thus $r \in S_i(b) \cap S_j(b)$. However, $r \not \in \cup_{k \in I \setminus \{i\}} S_k(b)$, and so Condition~\ref{item:C3minimal} holds for $\HH(b,I')$. 

Moreover, by Remark~\ref{remark:minimal}, Conditions~\ref{item:C1} and~\ref{item:C2} are satisfied as $\CC \subset \HH(b,I')$ and $b \in \CC$. Therefore, $\HH(b,I')$ is a minimal half-space with respect to $\CC$. 
\end{proof}
	
\section{A combinatorial criterion to determine whether a half-space is redundant}\label{sec:criterion}

Let $\Gamma$ be a set of (possibly degenerate) half-spaces. A half-space $\HH$ is said to be \myemph{redundant} with respect to $\Gamma$ if $\HH$ is implied by the half-spaces in $\Gamma$, meaning that their intersection $\cap_{\HH' \in \Gamma} \HH'$ is contained in $\HH$. 

In this section, we show that the redundancy of a non-degenerate half-space $\HH$ with respect to $\Gamma$ is a local property when the apex of $\HH$ is assumed to belong to all the half-spaces in $\Gamma$. As a consequence, under the same assumption, we show that the redundancy of a half-space in a finite set of half-spaces is equivalent to a reachability problem in directed hypergraphs.    

\begin{propositiondefinition}\label{Prop-Def-Locally-Redundant}
Let $\Gamma $ be a set of (possibly degenerate) half-spaces, and $\HH$ 
a non-degenerate half-space whose apex belongs to each half-space in $\Gamma$. Then, $\HH$ is redundant with respect to $\Gamma$ if, and only if, there exists a neighborhood $\NN$ of $\apex(\HH)$ such that $(\cap_{\HH' \in \Gamma} \HH') \cap \NN \subset \HH$. 

In the latter case, $\HH$ is said to be \myemph{locally redundant} with respect to $\Gamma$. 
\end{propositiondefinition}

\begin{proof}
The ``only if'' part is obvious. 

To prove the ``if'' part, let $a\mydef\apex(\HH)$, $I \mydef \sect(\HH)$ and $\DD \mydef \cap_{\HH' \in \Gamma} \HH'$. Assume there exists a neighborhood $\NN$ of $a$ such that $\DD \cap \NN \subset \HH$, but $\HH$ is non-redundant in $\Gamma$, \ie~$\DD \not \subset \HH$. Then, pick any $x \in \DD \setminus \HH$ and let $j \in \compl{I}$ be such that $\vectinverse{a_j} x_j >  \vectinverse{a_i} x_i$ 
for all $i \in I$. Define $\lambda$ as the maximal scalar such that 
$\lambda x_i \leq a_i$ for all $i \in [n]$. Let us denote by $R$ the (non-empty) set of the coordinates $r$ such that $\lambda x_r = a_r$. Note that for any $i \in I$, 
\[
\lambda x_i < \lambda \vectinverse{a_j} x_j  a_i \leq a_i \;,
\]
and so $R \cap  I= \emptyset$.  

Now, define $y \mydef a \mpplus \mu x$, where $\mu > \lambda$. 
Due to the definition of $R$, if we take $\mu$ close enough to $\lambda$, 
we have 
\[
y_r > a_r \iff r \in R \;.
\]
Then, since $R \cap I = \emptyset$, 
it follows that $y_i = a_i$ for all $i \in I$. 
As a consequence, $y \not \in \HH(a,I)$ while $y \in \DD$ 
(because $y$ is a tropical linear combination of $a,x \in \DD$ and $\DD$ is a tropical cone). 
However, this contradicts the fact that $\DD \cap \NN \subset \HH(a,I)$, 
because $y \in \NN$ for $\mu$ close enough to $\lambda$.
\end{proof}

To exploit the local characterization of Proposition~\ref{Prop-Def-Locally-Redundant}, we use the notion of tangent cone~\cite{AGG10,AllamigeonGaubertGoubaultDCG2013}. Given a vector $z \in \projspace$ of a tropical polyhedral cone $\DD \subset \projmaxspace$, the tangent cone of $\DD$ at $z$ provides a description of $\DD$ in a neighborhood of $z$. 
We say that a tropical half-space $\{ x \in \projmaxspace \mid \mpplus_{i \in I} \mpinverse{\alpha_i} x_i \geq \mpplus_{j \in J} \mpinverse{\alpha_j} x_j\}$ is \myemph{active} at $z$ if the following equality holds:
\[
\mpplus_{i \in I} \mpinverse{\alpha_i} z_i = \mpplus_{j \in J} \mpinverse{\alpha_j} z_j \; .
\]
\begin{definition}
Let $\DD = \cap_{\HH \in \Gamma} \HH\subset \projmaxspace$, where $\Gamma$ is a finite set of (possibly degenerate) half-spaces, and let $z \in \DD \cap \projspace$. With each half-space $\HH = \{ x \in \projmaxspace \mid \mpplus_{i \in I} \mpinverse{\alpha_i} x_i \geq \mpplus_{j \in J} \mpinverse{\alpha_j} x_j\}$ in $\Gamma$ active at $z$, we associate the inequality 
\begin{equation}
\max_{i \in M} y_i \geq \max_{j \in N} y_j \; ,  \label{eq:tangent_cone_description}
\end{equation}
where $M$ and $N$ are respectively the argument of the maxima $\mpplus_{i \in I} \mpinverse{\alpha_i} z_i$ and $\mpplus_{j \in J} \mpinverse{\alpha_j} z_j$.

Then, the \myemph{tangent cone} $\tangent(\DD ,z)$ of $\DD$ at $z$ is given by the set of vectors $y \in \projspace$ satisfying all the inequalities of the form~\eqref{eq:tangent_cone_description} associated with the (active) half-spaces in $\Gamma$. 
\end{definition}

The term \myemph{tangent cone} refers to the usual terminology used in optimization and convex analysis. In particular, the term \myemph{cone} refers here to the property that for all $y \in \tangent(\DD,z)$ and $\lambda > 0$, the vector $\lambda \times y$ belongs to the set $\tangent(\DD,z)$.

\begin{proposition}[\cite{AllamigeonGaubertGoubaultDCG2013}]\label{prop:tangent_cone_exactness}
Let $z \in \DD \cap \projspace$, where $\DD$ is a tropical polyhedral cone. There exists a neighborhood $\NN$ of $z$ such that for all $x \in \NN$, $x \in \DD$ if, and only if, $x \in z + \tangent(\DD,z)$.
\end{proposition}

We now introduce an equivalent encoding of tangent cones in terms of directed hypergraphs. Recall that directed hypergraphs are generalizations of directed graphs, in which the tail and the head of arcs may consist of several nodes. More precisely, a \myemph{directed hypergraph} on the node set $[n] = \{1, \dots, n\}$ consists of a set of \myemph{hyperarcs}, each of which is of the form $(T,H)$, where $T, H \subset \oneto{n}$. 

Reachability can be naturally extended to directed hypergraphs as follows. Given a directed hypergraph $\GG$ on the node set $\oneto{n}$, a node $j \in [n]$ is reachable from a set of nodes $I \subset \oneto{n}$ if one of the following two conditions holds:
\begin{enumerate}[label=(\roman*)] 
\item $j$ belongs to $I$,  
\item or there is a hyperarc $(T, H)$ in $\GG$ such that $j \in H$, and every $t \in T$ is reachable from $I$.
\end{enumerate}
By extension, given two sets of nodes $I, J \subset \oneto{n}$, $J$ is reachable from $I$ if each node in $J$ is reachable from $I$. Equivalently, $J$ is reachable from $I$ if there exists a \myemph{hyperpath} from $I$ to $J$, \ie~a sequence $(T_1, H_1),\ldots ,(T_q, H_q)$ of hyperarcs of $\GG$ such that:
\[
T_i \subset \cup_{0 \leq j \leq i-1} H_j \quad \text{for all} \ i \in \oneto{q+1} \;,
\]
with the convention $H_0 = I$ and $T_{q+1} = J$. 

\begin{remark}\label{remark:reachability_hypergraph}
Given $I \subset \oneto{n}$, the set of the subsets of $[n]$ reachable from $I$ admits a greatest element $R \subset \oneto{n}$, composed of all the nodes $j \in [n]$ reachable from $I$.
\end{remark}

A directed hypergraph consequently provides a concise representation, in terms of a set of hyperarcs, of a possibly large set of relations between subsets of $[n]$. This representation also allows to efficiently determine the relation between two subsets. Indeed, the reachability from $I$ to $J$ can be determined in linear time in the size $\sum_{(T,H) \in \GG} (\card{T} + \card{H})$ of the hypergraph, see for instance~\cite{GalloDAM93}.

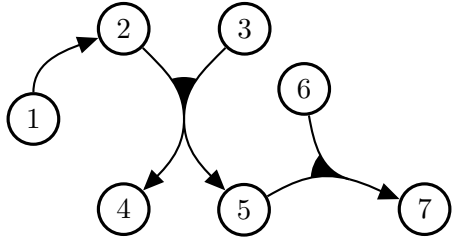
\begin{figure}
\begin{center}
\begin{tikzpicture}[>=triangle 45,scale=0.8, vertex/.style={circle,draw=black,very thick,minimum size=2ex}, hyperedge/.style={draw=black,thick}, simpleedge/.style={draw=black,thick}]
\node [vertex] (v1) at (-1.5,-1.5) {$1$};
\node [vertex] (v2) at (0,0) {$2$};
\node [vertex] (v3) at (2.01,0) {$3$};
\node [vertex] (v4) at (0,-3) {$4$};
\node [vertex] (v5) at (2,-3.01) {$5$};
\node [vertex] (v6) at (3,-1) {$6$};
\node [vertex] (v7) at (5,-3) {$7$};
\draw[simpleedge,->] (v1) edge[out=90,in=-160] (v2);
\hyperedge{v3,v2}{v4,v5};
\hyperedgewithangles{v6/-80,v5/30}{-15}{v7/158};
\end{tikzpicture}
\end{center}
\caption{A directed hypergraph.}\label{fig:hypergraph}
\end{figure}

\begin{example}
We provide an example of directed hypergraph on the node set $\{1, \dots, 7\}$ in Figure~\ref{fig:hypergraph}. It consists of the hyperarcs $(\{1\},\{2\})$, $(\{2,3\},\{4,5\})$, and $(\{5,6\},\{7\})$. Each hyperarc is represented as a bundle of arrows decorated by a solid disk sector. For instance, nodes $4$ and $5$ are both reachable from the set $\{1,3\}$, through a hyperpath formed by the first hyperarc (which leads to node $2$) and the second one. Similarly, the greatest set reachable from $\{1,3,6\}$ is the whole set of nodes $[7]$.
\end{example}

In our setting, directed hypergraphs are used to represent inequalities of the form~\eqref{eq:tangent_cone_description}.

\begin{definition}\label{def:tangent_hypergraph}
Let $\Gamma$ be a finite set of (possibly degenerate) half-spaces, and $z \in \projspace$ such that $z \in \HH$ for all $\HH$ in $\Gamma$. With each half-space $\HH = \{ x \in \projmaxspace \mid \mpplus_{i \in I} \mpinverse{\alpha_i} x_i \geq \mpplus_{j \in J} \mpinverse{\alpha_j} x_j\}$ in $\Gamma$ active at $z$, we associate the hyperarc $(M, N)$, where 
\[
M \mydef \argmax(\mpplus_{i \in I} \mpinverse{\alpha_i} z_i) \quad \text{and} \quad N \mydef \argmax(\mpplus_{j \in J} \mpinverse{\alpha_j} z_j) \;.
\]
The \myemph{tangent directed hypergraph at $z$ induced by $\Gamma$}, denoted by $\GG(\Gamma,z)$, is the directed hypergraph on the node set $\oneto{n}$ whose hyperarcs are the ones associated with the active half-spaces in $\Gamma$. 
\end{definition}

Observe that by definition, $\GG(\Gamma,z)$ depends on the set of half-spaces $\Gamma$. However, the following proposition shows that the reachability relations in $\GG(\Gamma,z)$ only depend on the tropical cone $\DD = \cap_{\HH \in \Gamma} \HH$. 

\begin{proposition}\label{prop:tangent_hypergraph}
Let $\DD \subset \projmaxspace$ be a tropical cone, and $z \in \DD \cap \projspace$. Assume $\DD = \cap_{\HH \in \Gamma} \HH$, where $\Gamma$ is a finite set of (possibly degenerate) half-spaces. Then, for any $I, J \subset \oneto{n}$, the following statements are equivalent:
\begin{enumerate}[label=(\roman*)]
\item\label{item:tangent_hypergraph1} $J$ is reachable from $I$ in the directed hypergraph $\GG(\Gamma,z)$,
\item\label{item:tangent_hypergraph2} the inequality $\max_{i \in I} y_i \geq \max_{j \in J} y_j$ is \myemph{valid} for $\tangent(\DD,z)$, meaning that it is satisfied for any $y \in \tangent(\DD,z)$.
\end{enumerate}
\end{proposition}

\begin{proof}
Assume $J$ is reachable from $I$ in $\GG(\Gamma,z)$. By definition, there exists a (possibly empty) hyperpath $(T_1, H_1),\ldots ,(T_q, H_q)$ from $I$ to $J$ in $\GG(\Gamma,z)$, meaning that $T_i \subset \cup_{0 \leq l \leq i-1} H_l$ for $i \in \oneto{q+1}$, where $H_0 = I$ and $T_{q+1} = J$. By definition, each hyperarc $(T_k, H_k)$ corresponds to an inequality 
\[
\max_{i \in T_k} y_i \geq \max_{j \in H_k} y_j
\]
which is valid for $\tangent(\DD ,z)$. This allows us to prove by induction on $k$ that 
\[
\max_{i \in I} y_i \geq \max_{j \in H_k} y_j
\]
is a valid inequality for $\tangent(\DD ,z)$ for $k = 1, \dots, q$. Since $J \subset H_0 \cup \dots \cup H_q$, we conclude that $\max_{i \in I} y_i \geq \max_{j \in J} y_j$ is also valid for $\tangent(\DD, z)$.

Now assume that for all $y \in \tangent(\DD ,z)$, the inequality $\max_{i \in I} y_i \geq \max_{j \in J} y_j$ holds. Let $R$ be the biggest subset of $\oneto{n}$ reachable from $I$ in $\GG(\Gamma,z)$.

Given $\epsilon > 0$, define the vector $y' \in \projspace$ by $y'_i = 0$ if $i \in R$, and $y'_i =\epsilon$ otherwise. Consider any active half-space $\HH= \{ x \in \projmaxspace \mid \mpplus_{i \in I'} \mpinverse{\alpha_i} x_i \geq \mpplus_{j \in J'} \mpinverse{\alpha_j} x_j\}$ in $\Gamma$, and let $M, N$ be as in Definition~\ref{def:tangent_hypergraph}. We claim that $y'$ satisfies the inequality 
\[
\max_{i \in M} y'_i \geq \max_{i \in N} y'_i
\]
associated with $\HH$. If $M \not \subset R$, then it is obviously satisfied. If $M \subset R$, then the set $M$, and subsequently the set $N$, are both reachable from $I$ in $\GG(\Gamma,z)$. Thus, $M \cup N \subset R$ and 
\[
\max_{i \in M} y'_i = 0 = \max_{i \in N} y'_i \;.
\]
As this holds for any active half-space $\HH$ in $\Gamma$, we conclude that $y'$ belongs to the tangent cone $\tangent(\DD ,z)$. Since the inequality $\max_{i \in I} y_i \geq \max_{j \in J} y_j$ is valid for $\tangent(\DD ,z)$ and $I\subset R$, we have 
\[
0 = \max_{i \in I} y'_i \geq \max_{j \in J} y'_j \;,
\]
implying $y'_j=0$ for all $j\in J$. This means that $J\subset R$, and so $J$ is reachable from $I$ in $\GG(\Gamma,z)$.
\end{proof}

We are going to use the reduction to local redundancy to characterize redundancy by means of the tangent hypergraph.

\begin{proposition}\label{prop:local_redundancy_hypergraph_characterization}
Let $\Gamma$ be a finite set of (possibly degenerate) half-spaces and $\HH(a,I)$ a half-space whose apex belongs to every half-space in $\Gamma$. Then, $\HH(a,I)$ is redundant with respect to $\Gamma$ if, and only if, $[n]$ is reachable from $I$ in the tangent directed hypergraph $\GG(\Gamma,a)$.
\end{proposition}

\begin{proof}
Let $\DD = \cap_{\HH' \in \Gamma} \HH'$. By Proposition~\ref{Prop-Def-Locally-Redundant}, $\HH(a,I)$ is redundant with respect to $\Gamma$ if, and only if, there exists a neighborhood $\NN$ of $a$ such that $\DD \cap \NN \subset \HH(a,I)$. By Proposition~\ref{prop:tangent_cone_exactness}, this is equivalent to the fact that 
\begin{equation}\label{eq:local_redundancy_hypergraph_characterization}
\tangent(\DD,a) \cap \NN' \subset \bigl\{ y \in \projspace \mid \max_{i \in I} y_i \geq \max_{j \in \compl{I}} y_j \bigr\} \; , 
\end{equation}
for some neighborhood $\NN'$ of the vector $(0,\ldots ,0)$. Besides, we claim that~\eqref{eq:local_redundancy_hypergraph_characterization} holds if, and only, if:
\[
\tangent(\DD,a) \subset \bigl\{ y \in \projspace \mid \max_{i \in I} y_i \geq \max_{j \in \compl{I}} y_j \bigr\} \;.
\]
To see this, assume~\eqref{eq:local_redundancy_hypergraph_characterization} holds, and let $y \in \tangent(\DD ,a)$. Then, $\lambda \times y \in \tangent(\DD ,a)$ for all $\lambda > 0$, and if $\lambda$ is sufficiently small, $\lambda \times y \in \NN'$. It follows that $\lambda \times y$, and consequently $y$, satisfies the inequality $\max_{i \in I} y_i \geq \max_{j \in \compl{I}} y_j$, proving the claim. 

Finally, using the first part of the proof and Proposition~\ref{prop:tangent_hypergraph}, we conclude that $\HH(a,I)$ is redundant with respect to $\Gamma$ if, and only if, the set $\compl{I}$, or equivalently $[n]$, is reachable from $I$ in $\GG(\Gamma,a)$.
\end{proof}

\begin{figure}
\begin{center}
\begin{tikzpicture}[convex/.style={draw=black,ultra thick,fill=lightgray,fill opacity=0.7},point/.style={blue!50},>=triangle 45,hsborder/.style={orange,dashdotted,thick},hs/.style={draw=none,pattern=north west lines,pattern color=orange,fill opacity=0.5}]
\draw[gray!30,very thin] (-0.5,-0.5) grid (9.5,6.5);
\draw[gray!50,->] (-0.5,0) -- (9.5,0) node[color=gray!50,below] {$x_2-x_1$};
\draw[gray!50,->] (0,-0.5) -- (0,6.5) node[color=gray!50,above] {$x_3-x_1$};

\node[coordinate] (v1) at (1,3) {};
\node[coordinate] (v12) at (4,3) {};
\node[coordinate] (v2) at (4,1) {};
\node[coordinate] (v23) at (6,1) {};
\node[coordinate] (v3) at (9,4) {};
\node[coordinate] (v31) at (8,3) {};

\filldraw[convex] (v1) -- (v12) -- (v2) -- (v23) -- (v3) -- (v31) -- cycle;

\draw[hsborder] (1,-0.5) -- (1,6.5);
\filldraw[hs,pattern=north east lines] (1,-0.5) -- (1,6.5) -- (1.2,6.5) -- (1.2,-0.5);

\draw[hsborder] (-0.5,3) -- (4,3) -- (4,-0.5);
\filldraw[hs] (-0.5,3) -- (4,3) -- (4,-0.5) -- (4.2,-0.5) -- (4.2,3.2) -- (-0.5,3.2) -- cycle;

\draw[hsborder] (-0.5,1) -- (6,1) -- (9.5,4.5);
\filldraw[hs] (-0.5,1) -- (6,1) -- (9.5,4.5) -- (9.5,4.8) -- (5.9,1.2) -- (-0.5,1.2) -- cycle;

\draw[hsborder] (-0.5,4) -- (9.5,4);
\filldraw[hs] (-0.5,4) -- (9.5,4) -- (9.5,3.8) -- (-0.5,3.8) -- cycle;

\draw[hsborder] (-0.5,3) -- (8,3) -- (9.5,4.5);
\filldraw[hs] (-0.5,3) -- (8,3) -- (9.5,4.5) -- (9.5,4.2) -- (8.1,2.8) -- (-0.5,2.8) -- cycle;

\draw[blue!50,very thick,dashed] (1,-0.5) -- (v1) -- (4.5,6.5);
\filldraw[hs,pattern color=blue!50,fill opacity=0.9] (1,-0.5) -- (v1) -- (4.5,6.5) -- (5,6.5) -- (1.4,2.8) -- (1.4,-0.5);
\filldraw[point] (v1) circle (3pt) node[above left=1.5pt] {$v^1$};
\end{tikzpicture}	
\end{center}
\caption{Determining the redundancy of the half-space $\HH(v^1,\{2\})$ with respect to the half-spaces in orange.}\label{FigRedundancyCriterion}
\end{figure}
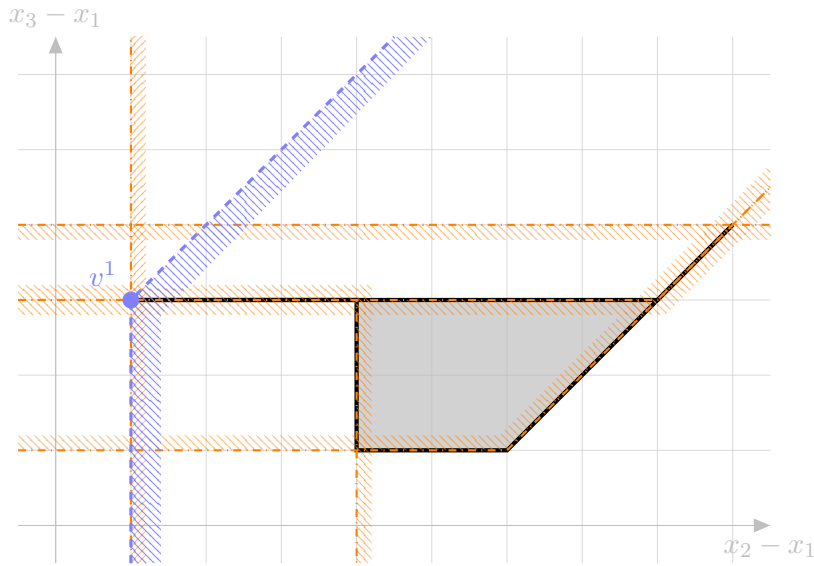

\begin{example}
The cone introduced in Figure~\ref{FigTropicalCone} can be expressed as the intersection of the collection $\Gamma$ of half-spaces given by the following inequalities:
\begin{equation}
\begin{aligned}
\underline{x_2} & \geq \underline{1 + x_1} \\
\max(-4+x_2,\underline{-3+x_3}) & \geq \underline{x_1} \\
-1+x_3 & \geq \max(x_1,-6+x_2) \\
x_1 & \geq -4+x_3 \\
\max(\underline{x_1},-8+x_2) & \geq \underline{-3+x_3} 
\end{aligned}
\label{eq:collection}
\end{equation}
These half-spaces are depicted in orange in Figure~\ref{FigRedundancyCriterion}. We illustrate Proposition~\ref{prop:local_redundancy_hypergraph_characterization} by establishing that the half-space $\HH(v^1,\{2\})$ (in blue in Figure~\ref{FigRedundancyCriterion}) is redundant with respect to $\Gamma$. Only the first two and last half-spaces of the list are active at $z$. For each of the corresponding inequalities $\mpplus_{i \in I} \mpinverse{\alpha}_i x_i \geq \mpplus_{j \in J} \mpinverse{\alpha}_j x_j$, the terms attaining the maxima $\mpplus_{i \in I} \mpinverse{\alpha_i} z_i$ and $\mpplus_{j \in J} \mpinverse{\alpha_j} z_j$ 
are underlined. The directed hypergraph $\GG(\Gamma,v^1)$ consequently consists of the hyperarcs $(\{2\},\{1\})$, $(\{3\},\{1\})$, and $(\{1\},\{3\})$.\footnote{Tangent directed hypergraphs can be computed with the library TPLib~\cite{tplib} (version 1.2 or later).} 
Node $1$ is reachable from $\{2\}$ through the first hyperarc, and then node $3$ is accessible through the last one. We conclude that the set $\{1,2,3\}$ is indeed reachable from $\{2\}$.
\end{example}

The interest of the criterion of Proposition~\ref{prop:local_redundancy_hypergraph_characterization} is not only theoretical, but also algorithmic, since it provides a polynomial-time method to eliminate superfluous half-spaces, assuming their apices belong to the other half-spaces:
\begin{corollary}
Given a finite set $\Gamma$ of (possibly degenerate) half-spaces, and a non-degenerate half-space $\HH$ such that $\apex(\HH) \in \HH'$ for all $\HH'$ in $\Gamma$, the redundancy of $\HH$ with respect to $\Gamma$ can be determined in time $O(n \card{\Gamma})$.
\end{corollary}
This result has to be compared with a criterion previously established in~\cite{AGK-10}, and expressed in terms of strategies for mean payoff games. Although the latter criterion applies to any half-space (without any assumption on the apex), it is not known whether it can be evaluated in polynomial time (the corresponding decision problem belongs to the complexity class $\mathsf{NP} \cap \mathsf{coNP}$).

\section{Non-redundant external representation of real polyhedral cones}\label{SectionNon-redundantRepre}

Throughout this section, $\CC \subset \projspace$ denotes a real polyhedral cone. Thanks to Proposition~\ref{Prop-Saturation}, we now focus on external representations of $\CC$ composed of (non-degenerate) half-spaces whose apices belong to $\CC$. 

We denote by $\Sigma$ the set of half-spaces containing $\CC$ and with apices in $\CC$, \ie 
\[
\Sigma\mydef \{ \HH \mid \CC \subset \HH \;,\;\apex(\HH) \in \CC  \}\; .
\]
To study the redundancy of a half-space in a set of half-spaces, 
it is convenient to introduce the function 
$\tau : 2^\Sigma \rightarrow 2^\Sigma$ defined by:
\begin{equation}\label{DefTau}
\tau(\Gamma) \mydef \{ \HH \in \Sigma \mid \cap_{\HH' \in \Gamma} \HH' \subset \HH \}\; .
\end{equation}
This function is a closure operator, meaning that for any $\Gamma, \Lambda \in 2^\Sigma$ the following properties hold:
\begin{enumerate}[label=(\roman*)]
\item\label{item:i} $\tau(\emptyset) = \emptyset$,
\item\label{item:ii} $\Gamma \subset \tau(\Gamma)$, 
\item\label{item:iii}  $\Gamma \subset \Lambda $ implies $\tau(\Gamma) \subset \tau(\Lambda)$,
\item\label{item:iv} $\tau(\tau(\Gamma)) = \tau(\Gamma)$.
\end{enumerate}

With this notation, a half-space $\HH \in \Sigma$ is redundant with respect to a set $\Gamma \subset \Sigma$ if, and only if, $\HH \in \tau(\Gamma)$ or, equivalently, $\tau(\Gamma) = \tau(\Gamma \cup \{\HH\})$. A finite set $\Gamma \subset \Sigma$ will be called a \myemph{non-redundant external representation} of $\CC$ if:
\[ 
\CC = \cap_{\HH \in \Gamma} \HH \ \text{(or equivalently,} \ \tau(\Gamma) = \Sigma \text{),} \  \text{and} \ \HH \not \in \tau(\Gamma \setminus \{\HH\}) \  \text{for each half-space} \  \HH \in \Gamma .
\]

This section is organized as follows. In Section~\ref{subsec:partial_antiexchange}, we prove a key result establishing that half-spaces with distinct apices satisfy an anti-exchange property. Section~\ref{subsec:non_redundant_apices} deals with non-redundant apices, and Section~\ref{subsec:non_redundant_with_same_apex} with non-redundant half-spaces with the same apex. These sections bring all the results to establish Theorem~\ref{Th-Main} in Section~\ref{subsec:proof}. 
Section~\ref{subsec:pure_case} is devoted to the particular case of non-redundant external representations of pure cones.

\subsection{The partial anti-exchange property}\label{subsec:partial_antiexchange}

We want to show the following partial anti-exchange property:
\begin{theorem}\label{prop:partial-anti-exchange} 
Let $\Gamma \subset \Sigma $ be a finite set of half-spaces and 
$\HH,\HH'\in \Sigma$ with distinct apices.
If $\HH' \not \in \tau(\Gamma )$ and 
$\HH'\in \tau(\{ \HH \}\cup \Gamma )$, 
then $\HH \not \in \tau(\{ \HH' \}\cup \Gamma )$.   
\end{theorem}

To prove this theorem, we shall use the following lemma:

\begin{lemma}\label{lemma:argmaxB}
Let $\Gamma \subset \Sigma $ be a finite set of half-spaces and 
$\HH(a,I) \in \tau (\Gamma )$. 
Then, for each non-empty subset $P$ of $\compl{I}$ there exists a half-space 
$\HH(b,J)$ in $\Gamma $ such that 
\[
\vectinverse{b_J} a = \vectinverse{b_{\compl{J}}} a 
,\;  
\argmax(\vectinverse{b_{J}} a) \cap P = \emptyset , \;
\makebox{and } \argmax(\vectinverse{b_{\compl{J}}} a) \cap P \neq \emptyset\; . 
\]
\end{lemma}

\begin{proof} 
Since $\HH(a,I) \in \tau (\Gamma)$ and $a \in \CC \subset \HH'$ for all $\HH' \in \Gamma$, we know by Proposition~\ref{prop:local_redundancy_hypergraph_characterization} that any subset of $[n]$ is reachable from $I$ in the tangent directed hypergraph $\GG(\Gamma,a)$. In particular, $P$ is reachable from $I$, thus the hypergraph $\GG(\Gamma,a)$ 
must contain a hyperarc $(T, H)$ such that $H \cap P \neq \emptyset$ 
and $T \subset \compl{P}$ (given a hyperpath $(T_1, H_1),\ldots ,(T_q, H_q)$ from $I$ to $P$, it suffices to set $(T,H)=(T_k, H_k)$, where $k \geq 1$ is the greatest integer such that $P \cap (\cup_{l = 0}^{k-1} H_l) = \emptyset$, recalling that $H_0=I$). By definition, this hyperarc is associated with a half-space $\HH(b,J)$ in $\Gamma$ active at $a$, meaning that 
\[
\vectinverse{b_J} a = \vectinverse{b_{\compl{J}}} a, \; T = \argmax(\vectinverse{b_{J}} a), \; \text{and} \; H = \argmax(\vectinverse{b_{\compl{J}}} a) \;.
\]
This provides the expected result.
\end{proof}

\begin{proof}[\proofname{} (Theorem~\ref{prop:partial-anti-exchange})]
Let $a \mydef \apex(\HH)$, $I \mydef \sect(\HH)$, $b \mydef \apex(\HH')$, 
and $J \mydef \sect(\HH')$. 

Since $\apex(\HH') \in \CC$, $\HH' \not \in \tau(\Gamma)$, and $\HH' \in \tau(\Gamma \cup \{\HH\})$, by Proposition~\ref{prop:local_redundancy_hypergraph_characterization} the set $[n]$ is reachable from $J$ in the hypergraph $\GG(\Gamma \cup \{\HH\},b)$, while it is not in the hypergraph $\GG(\Gamma ,b)$. Consequently, the two hypergraphs are not equal, which proves that the half-space $\HH$ necessarily provides a hyperarc in the hypergraph $\GG(\Gamma \cup \{\HH\},b)$, \ie~$\HH$ is active at $b$. More precisely, the hypergraph $\GG(\Gamma \cup \{\HH\},b)$ is obtained from $\GG(\Gamma,b)$ by adding the hyperarc $(M, N)$, where $M=\argmax(\vectinverse{a_I} b)$ and $N =\argmax(\vectinverse{a_{\compl{I}}} b)$. 

Let $R$ be the biggest subset of $[n]$ reachable from $J$ in $\GG(\Gamma,b)$. From the previous discussion, we have $R \subsetneq \oneto{n}$. Let $P$ be the complement of $R$ in $[n]$ (note that in particular $P \subset [n] \setminus J$ because $J\subset R$). As $[n]$ is reachable from $J$ in $\GG(\Gamma \cup \{\HH\},b)$, we necessarily have $N \cap P \neq \emptyset$ and $M \subset R$ (otherwise, the set $P$ would not be reachable from $J$ in $\GG(\Gamma \cup \{\HH\},b)$).  Hence, 
\begin{equation}\label{Relation-a-ap}
\vectinverse{a_I} b = \vectinverse{a_{\compl{I}}} b , \; 
\argmax(\vectinverse{a_I} b) \subset R,\; \text{and }
\argmax(\vectinverse{a_{\compl{I}}} b) \cap P \neq \emptyset \; . 
\end{equation}
Let $P' \mydef \argmax(\vectinverse{a_{\compl{I}}} b) \cap P$. 
As $\vectinverse{a_{\compl{I}}} b$ is equal to $\vectinverse{a_I} b$, it is also equal to $\vectinverse{a} b$, and so we have $P' \subset \argmax(\vectinverse{a} b) \cap (\compl{I})$.  

We shall prove that $\HH\not \in \tau (\{\HH'\}\cup \Gamma )$ 
by contradiction, so suppose that $\HH\in \tau (\{\HH'\}\cup \Gamma )$. 
Then, as $P' \subset {\compl{I}}$, by Lemma~\ref{lemma:argmaxB} 
we know that there exists a half-space 
$\HH''$ in $\{\HH'\}\cup \Gamma $, with apex $c$ and sectors $K$, such that 
\begin{equation}\label{Relation-a-app}
\vectinverse{c_{K}} a = \vectinverse{c_{\compl{K}}} a ,\; 
\argmax(\vectinverse{c_{K}} a) \cap P' = \emptyset,\; 
\text{and } \argmax(\vectinverse{c_{\compl{K}}} a) \cap P' \neq \emptyset \; . 
\end{equation}
Consider an arbitrary element 
$i \in \argmax(\vectinverse{c_{\compl{K}}} a) \cap P'$. Since $\vectinverse{c_{\compl{K}}} a = \vectinverse{c} a$, 
we have $i \in \argmax(\vectinverse{c} a) \cap (\compl{K})$.

Suppose that the half-space $\HH''$ coincides with 
$\HH'$, and so in particular $c=b$.  
Then, since $i \in \argmax(\vectinverse{c} a)=\argmax(\vectinverse{b} a)$, 
$\vectinverse{a_i} b_i$ is the minimum of 
$\vectinverse{a_h} b_h$ for $h \in \oneto{n}$. 
But as $i \in P'\subset \argmax(\vectinverse{a} b)$, 
$\vectinverse{a_i} b_i$ is also the maximum of 
$\vectinverse{a_h} b_h$ for $h \in \oneto{n}$.  
This is impossible unless $a$ and $b$ are identical (as elements of $\projspace$). As a consequence, the half-space $\HH''$ necessarily belongs to $\Gamma$.  

Now, since $i\in \argmax(\vectinverse{c} a)$ and 
$i \in P' \subset \argmax(\vectinverse{a} b)$, 
we have 
\[
\vectinverse{c_i} b_i = 
\vectinverse{c_i} a_i \vectinverse{a_i} b_i \geq 
\vectinverse{c_h} a_h \vectinverse{a_h} b_h = 
\vectinverse{c_h} b_h \quad \text{for any} \ h \in \oneto{n},
\] 
and thus $i \in \argmax(\vectinverse{c} b)$. Then, 
as $i \in \compl{K}$, we conclude that  
\[
i \in \argmax(\vectinverse{c_{\compl{K}}} b) \; 
\text{ and }\; \vectinverse{c_{K}} b = \vectinverse{c_{\compl{K}}} b \;,
\] 
because $\vectinverse{c_{K}} b \geq \vectinverse{c_{\compl{K}}} b$
according to the fact that $b\in \CC \subset \HH''$. 

Observe that for all $j \in K$, 
\[
(\vectinverse{c_{K}} a) (\vectinverse{a} b) \geq (\vectinverse{c_j} a_j) (\vectinverse{a_j} b_j) = \vectinverse{c_j} b_j \; .
\]
Moreover, the bound $(\vectinverse{c_{K}} a) (\vectinverse{a} b)$ is the maximum of 
$\vectinverse{c_j} b_j$ for $j \in K$. Indeed, 
\[
\mpplus_{j \in K} \vectinverse{c_j} b_j = \vectinverse{c_{K}} b = \vectinverse{c_{\compl{K}}} b = 
\vectinverse{c_i} b_i = (\vectinverse{c_i} a_i) (\vectinverse{a_i} b_i) \;,
\]
and since 
$i \in \argmax(\vectinverse{c_{\compl{K}}} a) \cap \argmax(\vectinverse{a} b)$, we 
have $\vectinverse{c_i} a_i =\vectinverse{c_{\compl{K}}} a =\vectinverse{c_{K}} a$ 
and $\vectinverse{a_i} b_i = \vectinverse{a} b$. It follows that 
\[
\argmax(\vectinverse{c_{K}} b)=\argmax (\vectinverse{c_{K}} a)\cap \argmax(\vectinverse{a} b) \; .
\] 

We are now going to show that $\argmax(\vectinverse{c_{K}} b)\subset R$. 
Given $h \in \argmax(\vectinverse{c_{K}} b)$, 
we either have $h \in \argmax(\vectinverse{a_{\compl{I}}} b)$ or 
$h \in \argmax(\vectinverse{a_I} b)$. In the latter case, 
by~\eqref{Relation-a-ap} we have 
$h\in \argmax(\vectinverse{a_I} b) \subset R$.   
Assume now that $h \in \argmax(\vectinverse{a_{\compl{I}}} b)$. 
Then, since 
\[
\argmax(\vectinverse{c_{K}} b) \cap \argmax(\vectinverse{a_{\compl{I}}} b) \cap P \subset \argmax(\vectinverse{c_{K}} a) \cap P' = \emptyset \;,
\]
it follows that $h\not \in P$, \ie{}\ $h\in R$. 
As a consequence, $\argmax(\vectinverse{c_{K}} b) \subset R$. 

Finally, since $\HH'' \in \Gamma$, and 
\[
\vectinverse{c_{K}} b = \vectinverse{c_{\compl{K}}} b, \  
\argmax(\vectinverse{c_{K}} b) \subset R \ \text{and} \  
i \in \argmax (\vectinverse{c_{\compl{K}}} b) \;, 
\]
we conclude that node $i$ is reachable from $J$ in the hypergraph $\GG(\Gamma,b)$, \ie{}\ $i \in R$. This contradicts the fact that $i \in P$, 
and completes the proof of the theorem.
\end{proof}

We shall need the following corollary of the partial anti-exchange property. 

\begin{corollary}\label{Coro:partial-anti-exchange} 
Let $\Gamma_1,\Gamma_2\subset \Sigma $ be two finite sets of half-spaces, and $\HH \in \Sigma $ be such that 
$\apex(\HH) \neq \apex(\HH')$ for all $\HH' \in \Gamma_2$.
If $\tau(\{\HH \} \cup \Gamma_1) = \tau(\Gamma_2)$, then $\HH \in \tau (\Gamma_1)$.    
\end{corollary}

\begin{proof} 
Let $\HH'$ be any half-space in $\Gamma_2$ and define 
$\Gamma'_2\mydef \Gamma_2\setminus \{\HH'\}$. 
Note that:
\[
\HH' \in \tau (\Gamma_2) = \tau(\{\HH \}\cup \Gamma_1) \subset 
\tau(\{\HH \} \cup \Gamma_1 \cup \Gamma'_2) \; ,
\]
and
\[
\HH \in \tau (\{\HH\} \cup \Gamma_1) = \tau (\Gamma_2) \subset 
\tau (\{\HH'\} \cup \Gamma_1 \cup \Gamma'_2) \;.
\]
Since $\HH$ and $\HH'$ have distinct apices, we conclude by Theorem~\ref{prop:partial-anti-exchange} that 
$\HH \in \tau (\Gamma_1\cup \Gamma'_2)$.
If the set $\Gamma'_2$ is non-empty, we can repeat the same argument by choosing a new half-space $\HH''$ in $\Gamma'_2$. Since $\Gamma_2$ is a finite set, this completes the proof. 
\end{proof}

\subsection{Apices of non-redundant external representations}\label{subsec:non_redundant_apices}

Note that the boundary $\mcB$ of $\CC$ is precisely the set of apices of half-spaces in $\Sigma$: 

\begin{lemma}
We have $\mcB = \{ \apex(\HH) \mid \HH \in \Sigma \}$.
\end{lemma}

\begin{proof}
Since no neighborhood of $\apex(\HH)$ is contained in $\HH$ for any half-space $\HH$, it readily follows that the apex of any half-space in $\Sigma$ does not belong to the interior of $\CC$. 

Conversely, consider $a \in \mcB$, and assume $\CC$ is not contained in any half-space with apex $a$. Then, for each $i \in \oneto{n}$ there exists $x^i \in \CC$ such that $\vectinverse{a_i} x^i_i > \vectinverse{a_{\compl{\{i\}}}} x^i$. Let $\epsilon \in \R$ be such that $\vectinverse{a_i} x^i_i > \mpinverse{\epsilon} \geq \vectinverse{a_{\compl{\{i\}}}} x^i$, and define $y^i \mydef a \mpplus \epsilon x^i$. Thus, $y^i$ satisfies $y^i_i > a_i$ and $y^i_j = a_j$ for all $j \neq i$. Now consider the cone $\NN$ generated by the vectors $y^i$ for $i \in \oneto{n}$. This cone forms a neighborhood of $a$ (it contains the Hilbert ball of center $a$ and radius $\rho = \min_{i \in \oneto{n}} \vectinverse{a_i} y^i_i > 0$). Besides, $\NN \subset \CC$ since $y^i \in \CC$ for all $i \in [n]$. Hence, $a$ is in the interior of $\CC$, which is a contradiction.
\end{proof}

For each $a\in \mcB$, we denote by $\Sigma_a$ the set of half-spaces with apex $a$ which contain $\CC$, \ie
\[
\Sigma_a \mydef \{ \HH(a,I) \mid \CC \subset \HH(a,I) \} \;.
\]
Obviously, $\Sigma=\cup_{a \in \mcB} \Sigma_a$. Now, define the function $\tau' : 2^{\mcB} \rightarrow 2^{\mcB}$ by 
\[
\tau'(X)\mydef \{a\in \mcB  \mid \Sigma_{a}\subset  
\tau( \cup_{b\in X} \Sigma_b ) \} \; .  
\]
Then, as in the case of $\tau$, we have: 

\begin{proposition}
The function $\tau'$ is a closure operator on $\mcB$.
\end{proposition}

\begin{proof}
First, $\tau'(\emptyset) = \emptyset$, 
as a consequence of the fact that $\tau(\emptyset) = \emptyset$. 
Similarly, for any $X \in 2^{\mcB }$, we have $X \subset \tau'(X)$ because 
$\Sigma_a \subset \tau(\cup_{b\in X} \Sigma_b)$ for $a\in X$. 

Besides, for $X, Y \in 2^{\mcB }$, we have 
\[
X \subset Y\implies \cup_{a\in X} \Sigma_a \subset \cup_{a\in Y} \Sigma_a 
\implies \tau(\cup_{a\in X} \Sigma_a) \subset \tau(\cup_{a\in Y} \Sigma_a) 
\implies \tau'(X) \subset \tau'(Y) \; .
\]

Finally, let us show that $\tau'(\tau'(X)) = \tau'(X)$ for all 
$X \in 2^{\mcB }$. If we define 
\[
\DD \mydef \cap_{\HH \in \Sigma_a,a \in X} \HH \quad \text{and} \quad \DD' \mydef \cap_{\HH \in \Sigma_a, a \in \tau'(X)} \HH \;,
\]
then $\DD' \subset \DD$ because $X \subset \tau'(X)$.  
Moreover, for any $\HH \in \cup_{a\in \tau'(X)} \Sigma_a$, 
we have $\DD \subset \HH$ and thus $\DD \subset \DD'$. 
Therefore, we conclude that $\DD = \DD'$. Note that $a \in \tau'(X)$ if, and only if, $\DD \subset \HH$ for all $\HH \in \Sigma_a$. Similarly, $a \in \tau'(\tau'(X))$ is equivalent to $\DD' \subset \HH$ for all $\HH \in \Sigma_a$. This implies $\tau'(\tau'(X)) = \tau'(X)$. 
\end{proof}

Unlike $\tau$, the closure operator $\tau'$ 
satisfies the anti-exchange property. 

\begin{proposition}
Let $X$ be a finite subset of $\mcB$, and $a, b \in \mcB$ two distinct elements of $\projspace$. Then, $b \not \in \tau'(X)$ and $b \in \tau'(X \cup \{a\})$ 
imply $a \not \in \tau'(X \cup \{b\})$.
\end{proposition}

\begin{proof}
Suppose that $a \in \tau'(X \cup \{b\})$. Then, 
if we define $\Gamma \mydef \cup_{c \in X} \Sigma_c$, 
we have 
\[
\tau(\Gamma \cup \Sigma_a) = \tau(\Gamma \cup \Sigma_a\cup \Sigma_b)= \tau(\Gamma \cup \Sigma_b) \;.
\]
Note that $b$ is 
distinct from the apices of half-spaces in 
$\Gamma = \cup_{c \in X} \Sigma_c$ 
because $b \not \in \tau'(X)$. Therefore, 
by successive applications of Corollary~\ref{Coro:partial-anti-exchange} 
to the half-spaces in $\Sigma_b$, we conclude that 
these half-spaces belong to $\tau(\Gamma )$. 
This implies $b \in \tau'(X)$, which is a contradiction.
\end{proof}

As a consequence of the previous two propositions, we obtain: 

\begin{corollary}
The pair $(\mcB ,\tau')$ is a convex geometry.
\end{corollary}

Recall that $X\subset \mcB$ is said to be a \myemph{spanning set} of $\mcB$ 
if $\tau'(X)= \mcB$. When the ground set $G$ of a convex 
geometry $(G,\tau')$ is finite, it is known that $G$ has a 
unique minimal spanning set, see for example~\cite{korte}.  
This minimal spanning set is composed of the \myemph{extreme elements} of $G$, 
which are the elements $a\in G$ such that 
$a\not \in \tau'(G \setminus \{ a\})$. 
Even if in our case the ground set $\mcB$ is infinite, 
we next show that it also admits a unique minimal finite spanning set.  

\begin{corollary}\label{Coro:Non-redundant-apices}
There exists a unique minimal finite subset 
$\mcA$ of $\mcB$ satisfying $\tau'(\mcA )= \mcB$.    
\end{corollary}

\begin{proof}
In the first place, observe that there exists a finite spanning set $X$ of $\mcB$. Indeed, as $\CC $ is a real polyhedral cone, there exists a finite set of half-spaces $\Gamma $, whose apices belong to $\CC$, such that $\CC=\cap_{\HH \in \Gamma} \HH$, see Section~\ref{subsec:saturation}. 
Then, we have $\tau'(\{\apex(\HH)\mid \HH \in \Gamma\})= \mcB$.

Assume now that $X$ and $Y$ are two distinct minimal finite spanning sets of $\mcB$, and let $a \in X \setminus Y$.
Let $\Gamma_1\mydef \cup_{b\in X\setminus \{ a\}} \Sigma_{b}$ and 
$\Gamma_2 \mydef  \cup_{b\in Y} \Sigma_{b}$. Then, 
since $\tau'(X)= \tau'(Y)=\mcB$, we have 
\[
\tau (\Sigma_a \cup \Gamma_1)= \tau (\Gamma_2)=\Sigma \;.
\] 
Now, as $a \not \in Y$, 
we can repeatedly apply Corollary~\ref{Coro:partial-anti-exchange} 
to the half-spaces in $\Sigma_a$ to conclude that 
$\Sigma_a\subset \tau (\Gamma_1)$, 
and so 
\[
\tau (\Gamma_1)= \tau (\Sigma_a \cup \Gamma_1)=\Sigma \;.
\] 
Therefore, 
$\tau'(X\setminus \{a\})=\mcB$ contradicting the fact that $X$ 
is a minimal spanning set of $\mcB$.  
\end{proof}

We can now establish the main theorem of this subsection, which shows that
the set $\mcA$ precisely characterizes the apices of the half-spaces in 
any finite non-redundant external representation of the cone $\CC$. As indicated in the introduction, such apices will be referred to as \myemph{non-redundant apices}.

\begin{theorem}\label{th:non_redundant_apices}
Let $\Gamma $ be any non-redundant external representation of 
$\CC$ (composed of finitely many half-spaces with apices in $\CC$). 
Then, $\mcA = \{ \apex(\HH) \mid \HH \in \Gamma\}$.  
\end{theorem} 

\begin{proof}
Since $\cap_{\HH \in \Gamma} \HH = \CC$, 
we have $\tau'(\{ \apex(\HH) \mid \HH \in \Gamma\})=\mcB$. So, by Corollary~\ref{Coro:Non-redundant-apices},  
\[
\mcA \subset \{ \apex(\HH) \mid \HH \in \Gamma\} \;.
\]
Now suppose that for some $\HH' \in \Gamma$, $\apex(\HH') \not \in \mcA$.
Since  $\tau(\Gamma )=\Sigma= \tau(\cup_{a\in \mcA} \Sigma_a)$, by Corollary~\ref{Coro:partial-anti-exchange} it follows that $\HH' \in \tau(\Gamma \setminus \{ \HH'\})$. This contradicts the fact that $\Gamma$ is a 
non-redundant external representation of $\CC$.
\end{proof}

\subsection{Non-redundant half-spaces with the same apex}\label{subsec:non_redundant_with_same_apex}

We now study those half-spaces which have the same apex $a \in \mcB$ 
in non-redundant external representations of $\CC$. With this aim, 
assume $\CC$ is given by the intersection of 
half-spaces $\{\HH(a,I_l)\}_{l \in \oneto{q}}\subset \Sigma_a$ with apex $a$,
and a tropical cone 
\[
\DD \mydef \cap_{\HH' \in \Lambda } \HH' \; , 
\]
where $\Lambda \subset \Sigma \setminus \Sigma_a$ is a finite set of half-spaces 
whose apices are 
distinct from $a$.
We want to characterize the minimal subsets $L$ of $\oneto{q}$ satisfying:
\begin{equation}\label{Eq-Nonredundant-Local}
\tau (\Lambda \cup \{\HH(a,I_l)\}_{l \in L})=
\tau (\Lambda \cup \{\HH(a,I_l)\}_{l \in \oneto{q}})= \Sigma \; .
\end{equation}
Observe that $a$ is a non-redundant apex if, and only if, such minimal subsets are non-empty. In principle, these subsets depend on the 
half-spaces composing the set $\Lambda$. 
However, we next show that indeed this is not the case. 

\begin{proposition}\label{Prop-CanonicalMaximalSCC}
Let $\Lambda_1$ and $\Lambda_2$ be two finite sets of half-spaces in $\Sigma$, 
whose apices are 
distinct from $a$, such that 
\begin{equation}\label{Property0-Prop-CanonicalMaximalSCC}
\tau (\Lambda_1 \cup \{\HH(a,I_l)\}_{l\in \oneto{q}})
=\tau (\Lambda_2 \cup \{\HH(a,I_l)\}_{l\in \oneto{q}}) \;.
\end{equation} 
Then, $L$ is a minimal subset of $\oneto{q}$ satisfying 
\begin{equation}\label{Property1-Prop-CanonicalMaximalSCC}
\tau (\Lambda_1 \cup \{\HH(a,I_l)\}_{l\in L})
=\tau (\Lambda_1 \cup \{\HH(a,I_l)\}_{l\in \oneto{q}})
\end{equation}
if, and only if, $L$ is a minimal subset of $\oneto{p}$ satisfying 
\begin{equation}\label{Property2-Prop-CanonicalMaximalSCC}
\tau (\Lambda_2 \cup \{\HH(a,I_l)\}_{l\in L})=
\tau (\Lambda_2 \cup \{\HH(a,I_l)\}_{l\in \oneto{q}}) \; .
\end{equation}
\end{proposition}

\begin{proof}
Observe that to prove the proposition, it is enough to show that a subset
$L$ of $\oneto{q}$ satisfies~\eqref{Property1-Prop-CanonicalMaximalSCC} 
only if it satisfies~\eqref{Property2-Prop-CanonicalMaximalSCC}. 

By the contrary, suppose that~\eqref{Property1-Prop-CanonicalMaximalSCC} 
is satisfied by some $L\subset \oneto{q}$ 
but~\eqref{Property2-Prop-CanonicalMaximalSCC} is not. 
In that case, we can always define a subset $L'$ of $\oneto{q}$ 
such that $L\subsetneq L'$, \eqref{Property2-Prop-CanonicalMaximalSCC} 
is satisfied with $L'$ instead of $L$, but 
\[
\tau (\Lambda_2 \cup \{\HH(a,I_l)\}_{l\in L'\setminus\{r\}})\subsetneq
\tau (\Lambda_2 \cup \{\HH(a,I_l)\}_{l\in \oneto{q}})
\] 
for some $r \in L'\setminus L$. 
Then, by~\eqref{Property1-Prop-CanonicalMaximalSCC} and the fact that 
$L\subset L'\setminus\{r\}$, we obtain: 
\[ 
\HH(a,I_{r})\in 
\tau (\Lambda_1 \cup \Lambda_2 \cup \{\HH(a,I_l)\}_{l\in L'\setminus\{r\}} ) 
\; .  
\] 
Moreover, given $\HH'\in \Lambda_1$, we have:
\[
\HH' \in \tau ((\Lambda_1\setminus \{\HH'\}) \cup \Lambda_2 \cup 
\{\HH(a,I_l)\}_{l\in L'}) 
\] 
by using~\eqref{Property0-Prop-CanonicalMaximalSCC} and the fact that~\eqref{Property2-Prop-CanonicalMaximalSCC} 
is satisfied with $L'$ instead of $L$.
As $\HH(a,I_r)$ and $\HH'$ have distinct apices, by Theorem~\ref{prop:partial-anti-exchange}, it follows that 
\[ 
\HH(a,I_{r})\in 
\tau ((\Lambda_1\setminus \{\HH'\}) \cup \Lambda_2 \cup 
\{\HH(a,I_l)\}_{l\in L'\setminus\{r\}} ) \; . 
\] 
Repeating this argument, we conclude that $\HH(a,I_{r})\in 
\tau (\Lambda_2 \cup \{\HH(a,I_l)\}_{l\in L'\setminus\{r\}})$. 
However, this is a contradiction, because it would imply  
\[
\tau (\Lambda_2 \cup \{\HH(a,I_l)\}_{l\in L'\setminus\{r\}} ) =
\tau (\Lambda_2 \cup \{\HH(a,I_l)\}_{l\in L'} )    
\]
and so~\eqref{Property2-Prop-CanonicalMaximalSCC} 
would be satisfied with $L'\setminus\{r\}$ instead of $L$.      
\end{proof}

We now introduce a directed graph $\reachgraph(\DD,a)$ defined as follows:
\begin{enumerate}[label=(\roman*)]
\item its nodes are the elements of the set 
\[
E_a \; \mydef \; 
\{[n] \} \, \cup \, \{ I \subset \oneto{n} \mid \HH(a,I) \in \Sigma_a \} 
\; = \; \{[n] \} \, \cup \, \{ I \subset \oneto{n} \mid \CC \subset \HH(a,I) \}  \; ,
\]
\item there is an arc from $I$ to $J$ if, and only if, $J$ is reachable from $I$ in the tangent directed hypergraph $\GG(\Lambda,a)$ at $a$ induced by $\Lambda$.
\end{enumerate}
Note that, by Proposition~\ref{prop:tangent_hypergraph}, the graph $\reachgraph(\DD,a)$ does not depend on the choice of the set $\Lambda$ of half-spaces representing the cone $\DD$. 

The following proposition shows that the redundancy of half-spaces with apex $a$ can be characterized using $\reachgraph(\DD,a)$. 

\begin{proposition}\label{Prop-Charac-SCC-Redundant}
Let $\{\HH(a,I_l)\}_{l \in L}$ be a non-empty subset of $\Sigma_a$, and $\HH(a,J) \in \Sigma_a$. Then, 
\begin{equation}\label{EqProp-Charac-SCC-Redundant}
\HH(a,J) \in \tau(\Lambda \cup \{\HH(a,I_l)\}_{l \in L})
\end{equation}
if, and only if, for some $r \in L$, $I_r$ is reachable from $J$ in the directed graph $\reachgraph(\DD,a)$.
\end{proposition}

\begin{proof}
In the first place, observe that the tangent directed hypergraph 
\[
\GG_L \mydef \GG(\Lambda \cup \{\HH(a,I_l)\}_{l \in L}, a)
\] 
is obtained by adding hyperarcs connecting $I_l$ with $\compl{I_l}$, for $l\in L$, to the tangent directed hypergraph $\GG(\Lambda ,a)$.

Assume $I_r$ is reachable from $J$ in $\reachgraph(\DD,a)$ for some $r \in L$. Then, $I_r$ is also reachable from $J$ in the hypergraph $\GG_L$. Since $\GG_L$ contains a hyperarc connecting $I_r$ with $\compl{I_r}$, we conclude that $[n]$ is reachable from $J$ in $\GG_L$. Therefore, by Proposition~\ref{prop:local_redundancy_hypergraph_characterization}, it follows that~\eqref{EqProp-Charac-SCC-Redundant} holds. 

Assume now that~\eqref{EqProp-Charac-SCC-Redundant} is satisfied. Then, by Proposition~\ref{prop:local_redundancy_hypergraph_characterization} 
we know that $[n]$ is reachable from $J$ in $\GG_L$. Consider a hyperpath connecting $J$ with $[n]$ in $\GG_L$. It is convenient to split the rest of the proof into two cases. 

If one of the hyperarcs in the hyperpath connecting $J$ with $[n]$ is associated with a half-space in $\{\HH(a,I_l)\}_{l\in L }$, let $\HH(a,I_r)$ be the half-space corresponding to the first occurrence of such an hyperarc in the hyperpath. Then, $I_r$ is reachable from $J$ in $\GG(\Lambda ,a)$, and consequently in $\reachgraph(\DD,a)$.

If no hyperarc in the considered hyperpath 
 is associated with a half-space in $\{\HH(a,I_l)\}_{l\in L }$, 
we conclude that $[n]$ is reachable from $J$ in 
$\GG(\Lambda ,a)$. Therefore, any subset of $\oneto{n}$ is reachable from $J$ in $\GG(\Lambda ,a)$, and so any node of $\reachgraph(\DD,a)$ is reachable from 
$J$ in $\reachgraph(\DD,a)$. This completes the proof of the proposition.  
\end{proof}

\begin{definition}\label{Def-Mutually-Redundant}
Two half-spaces $\HH, \HH' \in \Sigma$ are said to be \myemph{mutually redundant} with respect to $\Gamma \subset \Sigma$ if $\HH \in \tau (\Gamma \cup \{\HH' \})$ and $\HH' \in \tau (\Gamma \cup \{\HH \})$. 
\end{definition}
 
As an immediate corollary of Proposition~\ref{Prop-Charac-SCC-Redundant}, 
we obtain:  

\begin{corollary}\label{Coro-SameSCC}
The half-spaces $\HH(a,I)$ and $\HH(a,J)$ are mutually redundant with respect to $\Lambda $ if, and only if, $I$ and $J$ belong to the same strongly connected component of $\reachgraph(\DD,a)$.  
\end{corollary}

The reachability relation associated with the directed graph $\reachgraph(\DD,a)$ naturally induces a pre-order~$\reachleq$ on the elements of $E_a$, \ie{}\ $I \reachleq J$ if, and only if, $J$ is reachable from $I$. Considering the equivalence relation $I \equiv J$ defined by $I \reachleq J \reachleq I$, the pre-order can be turned into a partial order (still denoted $\reachleq$ by abuse of notation) over the quotient set $E_a / \mathord{\equiv}$ formed by the strongly connected components of $\reachgraph(\DD,a)$.

The abstract structure of half-spaces with the same apex $a$ is thus in relation to a poset convex geometry. A \myemph{poset convex geometry} is a pair $(G, \sigma)$, where $G$ is a ground set, and $\sigma : 2^G \rightarrow 2^G$ is the closure operator defined as 
\[
\sigma(X) \mydef \{ y \in G \mid y \reachleq x \makebox{ for some } x\in X\}
\]
for all $X \in 2^G$. Poset convex geometries arise from poset antimatroids~\cite{korte}, in the sense that the closed elements of a poset convex geometry are precisely the complements of the feasible elements of a poset antimatroid.

In our case, the poset convex geometry is associated with the partially ordered set formed by the strongly connected components of $\reachgraph(\DD,a)$. We can then verify that the quotient set $E_a / \mathord{\equiv}$ has a unique minimal spanning set, consisting of the strongly connected components which are maximal for the order~$\reachleq$. This leads to the following characterization:

\begin{theorem}\label{Th-MaximalSCC}
The following two properties hold:
\begin{enumerate}[label=(\roman*)]
\item\label{item:Th-MaximalSCC1} 
The apex $a$ is non-redundant if, and only if, the directed graph $\reachgraph(\DD,a)$ is not strongly connected.
\item\label{item:Th-MaximalSCC2} When $a$ is a non-redundant apex, $L$ is a minimal subset of $\oneto{q}$ satisfying~\eqref{Eq-Nonredundant-Local} if, and only if, $\{I_l \}_{l\in L}$ is composed of precisely one element of each maximal strongly connected component of~$\reachgraph(\DD,a)$.
\end{enumerate}
\end{theorem}

\begin{proof}
\begin{enumerate}[label=(\roman*),wide]
\item Assume the directed graph $\reachgraph(\DD,a)$ is strongly connected. Then, for each $I \in E_a$, node $\oneto{n}$ is reachable from $I$ in $\reachgraph(\DD,a)$, and consequently in the directed hypergraph $\GG(\Lambda,a)$. By Proposition~\ref{prop:local_redundancy_hypergraph_characterization}, we deduce that $\HH(a,I_l)$ is redundant with respect to $\Lambda$ for each $l \in \oneto{q}$. We conclude that $\Lambda$ is an external representation of the cone $\CC$, and since no half-space in $\Lambda$ has apex $a$, $a \not \in \mcA$ by Theorem~\ref{th:non_redundant_apices}.  

Suppose now that $a \not \in \mcA$. Then, $\HH(a,I) \in \tau(\Lambda)$ for all $I \in E_a\setminus \{\oneto{n} \}$. By Proposition~\ref{prop:local_redundancy_hypergraph_characterization}, node $[n]$ is reachable in $\GG(\Lambda,a)$ from any node $I \in E_a\setminus \{\oneto{n} \}$, hence in $\reachgraph(\DD,a)$. Since any node $I\in E_a$ is obviously reachable from node $\oneto{n}$ in $\reachgraph(\DD,a)$, we conclude that the directed graph $\reachgraph(\DD,a)$ is strongly connected.
\item First observe that by~\ref{item:Th-MaximalSCC1}, the nodes of the maximal strongly connected components of~$\reachgraph(\DD,a)$ are all distinct from the node $[n]$.

To prove the ``only if'' part, let $L$ be a minimal subset of $\oneto{q}$ satisfying~\eqref{Eq-Nonredundant-Local}. As $a$ is a non-redundant apex, we know that $L \neq \emptyset$. In the first place, consider any maximal strongly connected component of $\reachgraph(\DD,a)$, and let $J$ be any node of that component. Since 
\[
\HH(a,J) \in \Sigma = \tau (\Lambda \cup \{\HH(a,I_l)\}_{l \in L}) \;,
\]
then by Proposition~\ref{Prop-Charac-SCC-Redundant}, we know that for some $r\in L$, $ I_r$ is reachable from $J$ in $\reachgraph(\DD,a)$. As $J$ belongs to a maximal strongly connected component of $\reachgraph(\DD,a)$, $I_r$ must belong to the same component. As a consequence, $\{I_l\}_{l\in L}$ contains at least one element of 
each maximal strongly connected component of $\reachgraph(\DD,a)$. In the second place,  
by Corollary~\ref{Coro-SameSCC} the minimality of $L$
implies $\{I_l\}_{l\in L}$ is composed of precisely one node of each maximal strongly connected component of the digraph $\reachgraph(\DD,a)$.

Conversely, consider a (non-empty) subset $L \subset [q]$ such that $\{I_l\}_{l \in L}$ is composed of precisely one element of each maximal strongly connected component of $\reachgraph(\DD,a)$. 
Then, \eqref{Eq-Nonredundant-Local} is satisfied because 
\[
\HH(a,I_r) \in \tau (\Lambda \cup \{\HH(a,I_l)\}_{l\in L})
\] for all $r \in \oneto{q}$ by Proposition~\ref{Prop-Charac-SCC-Redundant}. We claim that $L$ is a minimal subset of $[q]$ satisfying~\eqref{Eq-Nonredundant-Local}. Indeed, if $L$ is a singleton, it is obviously minimal since $a$ is a non-redundant apex. Similarly, if $L$ has more than one element, by Proposition~\ref{Prop-Charac-SCC-Redundant} we have 
\[
\HH(a,I_r) \not \in \tau(\Lambda \cup \{\HH(a,I_l)\}_{l \in L \setminus \{r\}}) \quad \text{for all} \ r \in L \;.
\]
This shows the ``if'' part of the statement.\qedhere 
\end{enumerate}
\end{proof}

In particular, the following corollary follows from 
Theorem~\ref{Th-MaximalSCC} and 
Proposition~\ref{Prop-CanonicalMaximalSCC}.  

\begin{corollary}\label{Coro-Same-MaximalSCC}
Let $\Lambda_1$ and $\Lambda_2$ be two finite sets of half-spaces in $\Sigma$, 
whose apices are 
distinct from $a$, such that:
\[
\tau (\Lambda_1 \cup \{\HH(a,I_l)\}_{l\in \oneto{q}} )
=\tau (\Lambda_2 \cup \{\HH(a,I_l)\}_{l\in \oneto{q}} )=\Sigma \;.
\] 
If we define the tropical cones $\DD_1 \mydef \cap_{\HH \in \Lambda_1} \HH$ and 
$\DD_2 \mydef \cap_{\HH \in \Lambda_2} \HH$, then the maximal 
strongly connected components of the directed graphs $\reachgraph(\DD_1,a)$ and 
$\reachgraph(\DD_2,a)$ coincide. 
\end{corollary}

Theorem~\ref{Th-MaximalSCC} shows that, when $a$ is a non-redundant apex, a half-space $\HH(a,I)$ occurring in a non-redundant external representation of $\CC$ can be exchanged with another half-space $\HH(a,J)$ if, and only if, $I$ and $J$ belong to the same (maximal) strongly connected components of $\reachgraph(\DD,a)$. By Propositions~\ref{prop:tangent_cone_exactness} and~\ref{prop:tangent_hypergraph}, this implies that the equality constraint $\vectinverse{a_I} x = \vectinverse{a_J} x$ is satisfied for any $x \in \CC$ located in a certain neighborhood of $a$. This can be seen as analogous to the situation of a non-fully dimensional ordinary polytope (\ie{}\ whose affine hull is a proper subspace of $\mathbb{R}^n$). However, the difference here is that the exchange is due to the local shape of the polytope around $a$, and not to its global shape. Moreover, the vectors $x$ satisfying 
\[
\vectinverse{a_I} x = \vectinverse{a_J} x \geq \vectinverse{a_{\compl{(I\cup J)}}} x
\]
are included in a tropical hyperplane if, and only if, $I \cap J = \emptyset$. In contrast, in the case where $I$ is included in $J$, this constraint is equivalent to the inequality $\vectinverse{a_I} x \geq \vectinverse{a_{\compl{I}}} x$.

We finally study the structure of maximal strongly connected components of $\reachgraph(\DD,a)$. With this aim, recall that a \myemph{principal ideal} of $E_a$ is a subset of $E_a$ of the form $\{ I \in E_a \mid I \subset J\}$, for certain $J \in E_a$, which is called the \myemph{principal element} of this ideal. 

\begin{proposition}
Every maximal strongly connected component $C$ of $\reachgraph(\DD,a)$ is a principal ideal of $E_a$. Moreover, the principal element of $C$ is given by the biggest subset $R$ of $\oneto{n}$ reachable from $I$ in $\GG(\Lambda,a)$, for any $I\in C$. 
\end{proposition}

\begin{proof}
Consider any $I \in C$, and (recalling Remark~\ref{remark:reachability_hypergraph}) let $R$ be the biggest subset of $\oneto{n}$ reachable from $I$ in $\GG(\Lambda,a)$. Since $I \subset R$, we know that $R \in E_a$. We claim that for all $J \in E_a$, $J$ belongs to $C$ if, and only if, $J \subset R$. 

If $J$ belongs to $C$, then $J$ is reachable from $I$ in the graph $\reachgraph(\DD,a)$, and consequently in the hypergraph $\GG(\Lambda,a)$. By definition of $R$, it follows that $J \subset R$.

Conversely, if $J \subset R$, then $J$ is reachable from $I$ in $\GG(\Lambda,a)$, thus also in $\reachgraph(\DD,a)$ because we assume $J \in E_a$. As $C$ is a maximal strongly connected component of $\reachgraph(\DD,a)$ and $I\in C$, we conclude that $J \in C$.
\end{proof}

As a consequence, any maximal strongly connected component $C$ of $\reachgraph(\DD,a)$ is completely determined by its principal element. Besides, when the apex $a$ is non-redundant, the minimal elements (for inclusion) of $C$ correspond to minimal half-spaces containing $\CC$. To see this, assume that $I$ is a minimal element of $C$, and let $\HH$ be a minimal half-space containing $\CC$ such that $\CC \subset \HH \subset \HH(a,I)$. Let $\{I_l\}_{l \in L}$ be composed of precisely one element 
of each maximal strongly connected component of $\reachgraph(\DD,a)$, 
except $C$. Replacing $\HH(a,I)$ by $\HH$ in $\Lambda \cup \{\HH(a,I_l)\}_{l \in L}\cup \{\HH(a,I)\}$, we obtain another (finite) external representation of $\CC$ composed of half-spaces in $\Sigma$. Then, by Theorem~\ref{Th-MaximalSCC}, $\HH$ is necessarily 
of the form $\HH(a,J)$, where $J \in C$. Moreover, as $\HH(a,J) \subset \HH(a,I)$, we have $J \subset I$, which shows that $I = J$ because $I$ is a minimal element of $C$. Thus, $\HH(a,I)$ is a minimal half-space containing $\CC$. 

\subsection{Proof of Theorem~\ref{Th-Main} and illustrations}\label{subsec:proof} We have now all the ingredients to establish Theorem~\ref{Th-Main}.

Given $a \in \mcA$, let $\mcC_a \subset 2^{\Sigma_a}$ be composed of the sets of half-spaces $\{\HH(a,I) \mid I \in C\}$, where $C$ ranges over the maximal strongly connected components of the directed graph $\reachgraph(\DD,a)$, 
with $\DD$ defined as the intersection of the half-spaces with apices different from $a$ in some external representation of $\CC$. In the first place, observe that Corollary~\ref{Coro-Same-MaximalSCC} shows 
$\mcC_a$ is independent of the choice of the external representation of $\CC$. 

Now, let $\Gamma$ be a non-redundant external representation of $\CC$. By Theorem~\ref{th:non_redundant_apices}, $\Gamma$ is the union, for $a \in \mcA$, of non-empty sets $\Gamma_a \subset \Gamma$ of half-spaces with apex $a$. Let $\Lambda\mydef \cup_{b\in \mcA\setminus \{a\}}\Gamma_b $ and $\DD\mydef\cap_{\HH\in \Lambda}\HH$. If $\{I_l\}_{l\in \oneto{q}} \subset E_a$ is such that $\Gamma_a = \{\HH(a,I_l)\}_{l \in \oneto{q}}$, then $L = [q]$ is a minimal subset of $\oneto{q}$ satisfying~\eqref{Eq-Nonredundant-Local}. Therefore, by Theorem~\ref{Th-MaximalSCC}, $\{I_l \}_{l\in \oneto{q}}$ is composed of precisely one element of each maximal strongly connected component of $\reachgraph(\DD,a)$. Thus, according to the discussion above, $\Gamma_a$ is composed of precisely one half-space of each set in the collection $\mcC_a$. This proves the ``only if'' part of the property in Theorem~\ref{Th-Main}.  

To prove the ``if'' part of the property in Theorem~\ref{Th-Main}, 
we only need to note that, according to Corollary~\ref{Coro-SameSCC} and the discussion above, we can replace a half-space in $\Gamma_a$ by any other half-space in the same set of the collection $\mcC_a$. By Theorem~\ref{Th-MaximalSCC}, we still obtain a non-redundant external representation of $\CC$. 

\begin{figure}[tp]
\begin{minipage}[c]{0.49\textwidth}
\begin{center}
\includemovie[mimetype=model/u3d,
text={\includegraphics[scale=0.5]{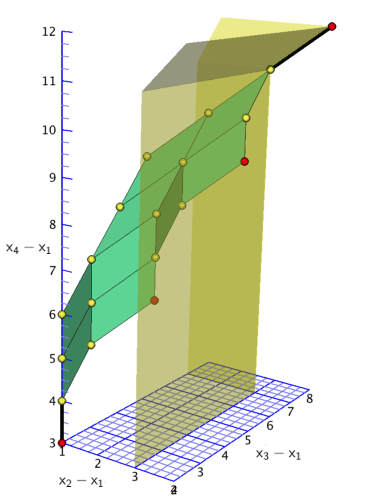}},
attach=true,
3Djscript=figures/Encompass.js,
3Dlights=CAD]{}{}{figures/cyclic_exchange1.u3d}
\end{center}
\end{minipage}
\begin{minipage}[c]{0.49\textwidth}
\begin{center}
\includemovie[mimetype=model/u3d,
text={\includegraphics[scale=0.5]{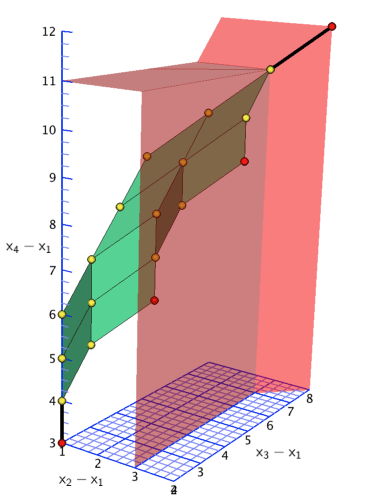}},
attach=true,
3Djscript=figures/Encompass2.js,
3Dlights=CAD]{}{}{figures/cyclic_exchange2.u3d}
\end{center}
\end{minipage}
\caption{Exchange of two mutually redundant half-spaces of the $4$th cyclic cone in $\projspace[3]$ (pictures have been drawn using {\normalfont\texttt{polymake}}~\cite{polymake}, JavaView, and JReality%
\protect\footnotemark).}\label{fig:cyclic_exchange}
\end{figure}

\footnotetext{Interactive 3D objects can be accessed by clicking on the pictures if U3D format is supported by the PDF reader, or by going to~\url{http://www.cmap.polytechnique.fr/~allamigeon/gallery.html}.}

\begin{example}\label{ex:cone_cyclic}
Let us illustrate the results of this section on the \myemph{$4$th cyclic cone} in $\projspace[3]$. For a given choice of four scalars $-\infty < t_1 < t_2 < t_3 < t_4$, this cone is generated by the vectors 
$(t_i \times 0,t_i \times 1,t_i \times 2, t_i \times 3)$ for $i = 1, \dots, 4$ (note that the product $t \times m$ corresponds to the tropical exponentiation $t^m$, which explains the name ``cyclic'' cone, see~\cite{blockyu06,AGK09}). Here we use the scalars $t_i \mydef i$ for $i = 1, \dots, 4$. The corresponding cone $\CC$ is depicted (in green) in Figure~\ref{fig:cyclic_exchange}. 
We start from a description $\Gamma$ of $\CC$ obtained by saturating the half-spaces associated with the non-trivial extreme vectors of the polar cones of $\CC$ (see Section~\ref{subsec:saturation}), and given by the following list of apices and sectors:
\[
\begin{array}{c@{\qquad}c@{\qquad}c@{\qquad}l}
(0,1,2,3), \{4\} 
& \textcolor{blue!70}{(0,3,7,11), \{1,3\}} 
& \textcolor{blue!70}{(0,3,7,11), \{1,4\}} 
& (0,1,2,6), \{3\} \\
\textcolor{red!80!black}{(0,1,3,5)}, \{1,4\} 
& (0,1,3,6), \{2,4\} 
& (0,1,3,7), \{1,3\} 
& (0,1,4,8), \{2,4\} \\
(0,1,5,9), \{2\} 
& \textcolor{red!80!black}{(0,2,4,7)}, \{1,4\}
& \textcolor{red!80!black}{(0,2,5,8)}, \{1,4\}
& (0,2,5,9), \{1,3\} \\
\textcolor{red!80!black}{(0,3,6,10)}, \{1,4\} 
& \textcolor{green!60!black}{(0,1,2,4), \{1,4\}} 
& \textcolor{green!60!black}{(0,1,2,4), \{2,4\}} 
& (0,4,8,12), \{1\} 
\end{array}
\tag{$\Gamma$}
\]
We number the corresponding half-spaces from $\HH_1$ to $\HH_{16}$ from left to right and top to bottom.

The set of non-redundant apices is given by the vectors which are not colored in red. For instance, it can be verified that the apex $(0,2,5,8)$ does not belong to $\mcA$. The tangent directed hypergraph $\GG(\Gamma \setminus \{\HH_{11}\},(0,2,5,8))$, generated by the set of half-spaces in $\Gamma$ having an apex distinct from $(0,2,5,8)$, is formed by the hyperarcs $(\{4\},\{3\})$, $(\{2\},\{3\})$, and $(\{1,3\},\{2\})$, see Figure~\ref{fig:reachgraph}. They are respectively associated with the half-spaces $\HH_{6}$ (or $\HH_{10}$), $\HH_8$, and $\HH_{12}$, which are active at $(0,2,5,8)$. As a consequence, the set $\{1,2,3,4\}$ is reachable from $\{1,4\}$. It can be verified that a half-space $\HH((0,2,5,8),I)$ contains $\CC$ if, and only if, $I \supset \{1,4\}$. In this case, the reachability graph $\reachgraph(\cap_{i \neq 11} \HH_i,a)$ consists of the nodes $\{1,4\}$, $\{1,2,4\}$, $\{1,3,4\}$, and $\{1,2,3,4\}$, and is necessarily strongly connected. By the first part of Theorem~\ref{Th-MaximalSCC}, we conclude that the vector $(0,2,5,8)$ is not a non-redundant apex. 

We now illustrate with apex $(0,3,7,11)$ the situation in which two half-spaces can be exchanged in a non-redundant representation. The tangent  directed hypergraph $\GG(\Gamma \setminus \{\HH_2,\HH_3\},(0,3,7,11))$ consists of the hyperarcs $(\{3\},\{4\})$, $(\{2\},\{3,4\})$, and $(\{4\},\{3\})$. A half-space $\HH$ with apex $(0,3,7,11)$ contains the cone $\CC$ if, and only if, there exists $i \in \{2,3,4\}$ such that $\{1, i\} \subset \sect(\HH)$. Let us denote by $\DD$ the cone $\cap_{\HH \in \Gamma \setminus \{\HH_2,\HH_3\}} \HH$ provided by the intersection of the half-spaces in $\Gamma$ with apices different from $(0,3,7,11)$. We deduce that the directed graph $\reachgraph(\DD,(0,3,7,11))$, depicted in Figure~\ref{fig:reachgraph}, is not strongly connected. Thus, the apex $(0,3,7,11)$ is non-redundant by Theorem~\ref{Th-MaximalSCC}. Besides, the graph $\reachgraph(\DD,(0,3,7,11))$ has only one maximal strongly connected component, composed of the subsets $\{1,3\}$, $\{1,4\}$, and $\{1,3,4\}$, the latter being the principal element of the component. It follows that the collection $\mcC_{(0,3,7,11)}$ of Theorem~\ref{Th-Main} is composed of a unique set of half-spaces $\{\HH((0,3,7,11),\{1,3\}), \HH((0,3,7,11),\{1,4\}), \HH((0,3,7,11),\{1,3,4\})\}$, and that only the first two half-spaces are minimal with respect to $\CC$ (see Figure~\ref{fig:cyclic_exchange}, where they are depicted in yellow and red respectively). By Theorem~\ref{Th-MaximalSCC}, we conclude that any non-redundant external representation obtained from $\Gamma$ 
contains precisely one of the half-spaces $\HH((0,3,7,11),\{1,3\})$ and $\HH((0,3,7,11),\{1,4\})$.

More generally, it can be verified that the non-redundant external representations 
which can be obtained from the representation $\Gamma$
are the ones containing all the half-spaces in black, and precisely one half-space in green and one in blue.
\end{example}

\begin{figure}[tp]
\begin{center}
\begin{tikzpicture}	[>=triangle 45,scale=0.8, vertex/.style={circle,draw=black,very thick,minimum size=2ex}, hyperedge/.style={draw=black,thick}, simpleedge/.style={draw=black,thick}]
\begin{scope}
\node [vertex] (v4) at (0,1) {$4$};
\node [vertex] (v1) at (2,2) {$1$};
\node [vertex] (v3) at (2,0) {$3$};
\node [vertex] (v2) at (5,1) {$2$};
\hyperedge{v1,v3}{v2};
\draw[simpleedge,->] (v4) -- (v3);	
\draw[simpleedge,->] (v2) to[out=-120,in=0] (v3);
\node at (2.5,-1) {(\refstepcounter{tikzfigures}\alph{tikzfigures}\label{fig:reachgraph1})};
\end{scope}
\begin{scope}[xshift=1cm,yshift=-7cm]
\node [vertex] (v1) at (1,3) {$1$};
\node [vertex] (v2) at (0,1) {$2$};
\node [vertex] (v3) at (3,2) {$3$};
\node [vertex] (v4) at (3,0) {$4$};
\hyperedge{v2}{v4,v3};
\draw[simpleedge,->] (v3) edge[out=-45,in=45] (v4);
\draw[simpleedge,->] (v4) edge[out=90,in=-90] (v3);
\node at (1.5,-1) {(\refstepcounter{tikzfigures}\alph{tikzfigures}\label{fig:reachgraph2})};
\end{scope}
\begin{scope}[xshift=12cm,vertex/.style={inner sep=3pt},yshift=-5cm]

\node [vertex] (v1234) at (0,6) {$\{1,2,3,4\}$};
\node [vertex] (v123) at (-2,4) {$\{1,2,3\}$};
\node [vertex] (v124) at (2,4) {$\{1,2,4\}$};
\node [vertex] (v12) at (0,2) {$\{1,2\}$};
\node [vertex] (v13) at (-4,-1) {$\{1,3\}$};
\node [vertex] (v14) at (4,-1) {$\{1,4\}$};
\node [vertex] (v134) at (0,-1) {$\{1,3,4\}$};

\draw[simpleedge,<->] (v1234) -- (v123);
\draw[simpleedge,<->] (v1234) -- (v124);
\draw[simpleedge,<->] (v1234) -- (v12); 
\draw[simpleedge,->] (v1234) to[out=-170,in=120] (v13);
\draw[simpleedge,->] (v1234) to[out=-10,in=60] (v14);
\draw[simpleedge,->] (v1234) to[out=-120,in=120] (v134);
\draw[simpleedge,<->] (v123) -- (v12);
\draw[simpleedge,->] (v123) to[out=-120,in=80] (v13);
\draw[simpleedge,->] (v123) to[out=-80,in=160] (v14);
\draw[simpleedge,->] (v123) to[out=-100,in=150] (v134);
\draw[simpleedge,<->] (v123) -- (v124);
\draw[simpleedge,<->] (v124) -- (v12);
\draw[simpleedge,->] (v124) to[out=-60,in=110] (v14);
\draw[simpleedge,->] (v124) to[out=-100,in=20] (v13);
\draw[simpleedge,->] (v124) to[out=-80,in=30] (v134);
\draw[simpleedge,->] (v12) -- (v13);
\draw[simpleedge,->] (v12) -- (v14);
\draw[simpleedge,->] (v12) -- (v134);
\draw[simpleedge,<->] (v13) -- (v134);
\draw[simpleedge,<->] (v14) -- (v134);
\draw[simpleedge,<->] (v13) to[out=-30,in=-150] (v14);

\node at (0,-3) {(\refstepcounter{tikzfigures}\alph{tikzfigures}\label{fig:reachgraph3})};
\end{scope}
\end{tikzpicture}
\end{center}
\caption{\eqref{fig:reachgraph1}~The tangent directed hypergraph $\GG(\{\HH_i\}_{i \neq 11},(0,2,5,8))$; \eqref{fig:reachgraph2}~The tangent directed hypergraph $\GG(\{\HH_i\}_{i \not \in \{2,3\}},(0,3,7,11))$; \eqref{fig:reachgraph3}~The reachability digraph $\reachgraph(\cap_{i \not \in \{2,3\}} \HH_i,(0,3,7,11))$ over the elements of $E_{(0,3,7,11)}$.}\label{fig:reachgraph}
\end{figure}
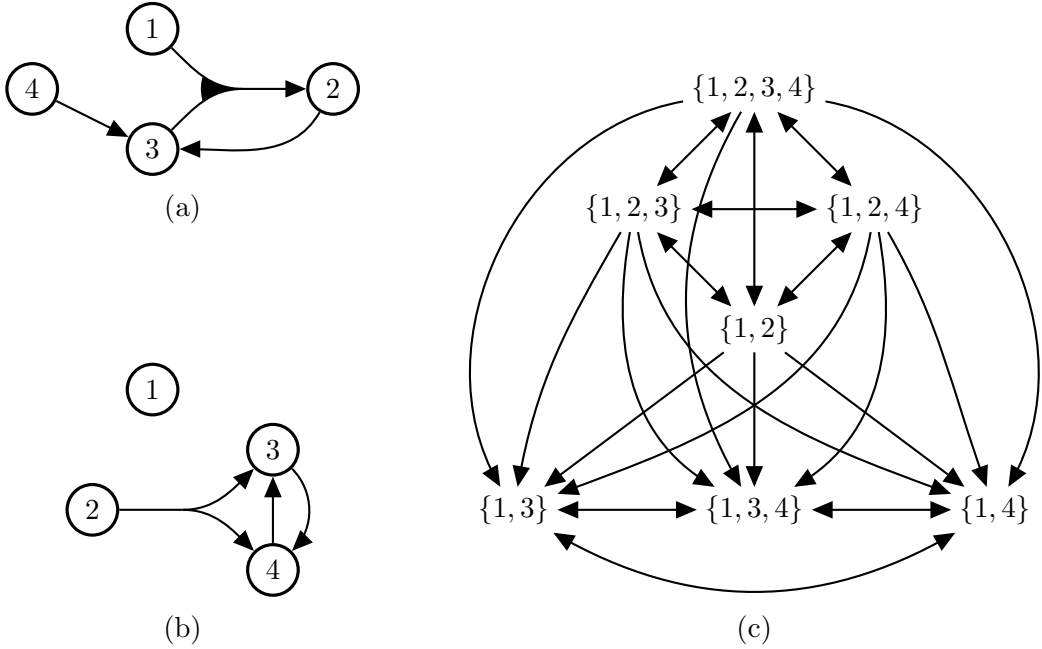

\subsection{The particular case of pure cones}\label{subsec:pure_case}

We say that the real polyhedral cone $\CC$ is \myemph{pure} when it coincides with the topological closure of its interior. This terminology originates from the equivalence between this definition, and the fact that the maximal (inclusion-wise) bounded cells of the natural cell decomposition of $\projspace$ induced 
by a generating set $\{v^r\}_{r\in \oneto{p}}$ of $\CC$ are all full-dimensional (the subcomplex of bounded cells is said to be \myemph{pure}). 

\begin{lemma}\label{Lemma-Pure-Case}
If the real polyhedral cone $\CC$ is pure, for any $a \in \mcB$ there exists at most one minimal half-space with respect to $\CC$ with apex $a$.
\end{lemma}

\begin{proof}
In the first place, we claim that if $\HH(a,I)$ and $\HH(a,J)$ 
contain $\CC$ and $I \cap J \neq \emptyset$, then $\HH(a,I\cap J)$ 
also contains $\CC$. Otherwise, \ie{}\ if $\HH(a,I\cap J)$ does not contain $\CC$ (so $I$ and $J$ are not comparable), there exists $y \in \CC$ such that 
\[
\vectinverse{a_{\compl{(I\cup J)}}} y \mpplus \vectinverse{a_{J\setminus I}} y \mpplus \vectinverse{a_{I\setminus J}} y > \vectinverse{a_{I\cap J}} y \; .
\]
Besides,
\[
\begin{aligned}
\vectinverse{a_{I\setminus J}} y \mpplus  \vectinverse{a_{I\cap J}} y & \geq \vectinverse{a_{J\setminus I}} y \mpplus \vectinverse{a_{\compl{(I\cup J)}}} y \; ,\\
\vectinverse{a_{J\setminus I}} y \mpplus  \vectinverse{a_{I\cap J}} y & \geq \vectinverse{a_{I\setminus J}} y \mpplus \vectinverse{a_{\compl{(I\cup J)}}} y \;.
\end{aligned}
\] 
Thus, we have:
\[
\vectinverse{a_{J\setminus I}} y = \vectinverse{a_{I\setminus J}} y > \vectinverse{a_{I\cap J}} y \;.
\]
Since $\CC$ is pure, we can find $x$ in the interior of $\CC$ close enough to $y$ so that 
\[
\vectinverse{a_{J\setminus I}} x > \vectinverse{a_{I\cap J}} x, \ \vectinverse{a_{I\setminus J}} x > \vectinverse{a_{I\cap J}} x, \ \text{and} \ \vectinverse{a_{J\setminus I}} x \neq \vectinverse{a_{I\setminus J}} x \;.
\]
If, for instance, $\vectinverse{a_{J\setminus I}} x > \vectinverse{a_{I\setminus J}} x$, then $\vectinverse{a_{J\setminus I}} x > \vectinverse{a_I} x$. However, this contradicts that $x$ belongs to $\HH(a,I)$, and so also to $\CC$. Therefore, $\HH(a,I\cap J)$ must contain $\CC$, proving our claim.

According to the first part of the proof, 
if we suppose that $\HH(a,I)$ and $\HH(a,J)$ 
are two different minimal half-spaces with respect to $\CC$, 
then necessarily $I \cap J=\emptyset$. 
Since in that case for all $x\in \CC$ we have 
\[ 
\vectinverse{a_I} x  \geq \vectinverse{a_J} x \mpplus \vectinverse{a_{\compl{(I\cup J)}}} x 
\quad \text{and} \quad
\vectinverse{a_J} x  \geq \vectinverse{a_I} x \mpplus \vectinverse{a_{\compl{(I\cup J)}}} x \;,
\]
it follows that:
\[
\vectinverse{a_I} x = \vectinverse{a_J} x \geq \vectinverse{a_{\compl{(I\cup J)}}} x \;, \quad \text{for all} \ x \in \CC \;.
\]
As $I \cap J=\emptyset$, this implies $\CC$ is contained in the tropical hyperplane whose apex is $a$, which contradicts the fact that $\CC$ is pure (hence, it has a non-empty interior).
\end{proof}

\begin{proposition}\label{Prop-Pure-Case}
If the real polyhedral cone $\CC$ is pure, every non-redundant external representation of $\CC$ consists of precisely one half-space $\HH(a,I)$ for each non-redundant apex $a$. In particular, there exists a unique non-redundant external representation composed of minimal half-spaces.
\end{proposition}

\begin{proof}
Assume that $\reachgraph(\DD,a)$ contains two different maximal strongly connected components $C$ and $C'$, and let $I$ and $J$ be minimal elements of $C$ and $C'$. Then, as explained above, $\HH(a,I)$ and $\HH(a,J)$ are different minimal half-spaces containing $\CC$, which is impossible by Lemma~\ref{Lemma-Pure-Case}. Thus, the graph $\reachgraph(\DD,a)$ has only one maximal strongly connected component. Besides, using the same argument, we can also conclude that this component has only one minimal element. This implies there is a unique non-redundant external representation of $\CC$ composed of minimal half-spaces.
\end{proof}
   
\section{Apices of non-redundant half-spaces and cell decomposition}\label{sec:cell_decomposition}
	
In this section we show that the non-redundant apices 
associated with a real polyhedral cone come from a small 
set of candidates. In the sequel, $\CC$ denotes a real polyhedral cone 
generated by the set of vectors $\{v^1, \dots, v^p\}\subset \projspace$. 

As already observed in Section~\ref{subsec:saturation}, 
the real polyhedral cone $\CC$ has an external representation composed of 
minimal half-spaces with respect to $\CC$. Besides, the apex of each minimal half-space belongs to $\CC$. In consequence, the non-redundant apices $\mcA$ associated with $\CC$ are apices of minimal half-spaces, and so they 
belong to the cells of the natural cell decomposition of $\projspace$ induced 
by the generators $\{v^r\}_{r\in \oneto{p}}$ of $\CC$ characterized by 
the conditions of Theorem~\ref{Theo-Charac-Minimal}. We next show that the elements of $\mcA$ satisfy a stronger property, 
implying they are special vertices (recall that cells composed of apices of minimal half-spaces need not be zero-dimensional).

\begin{definition}\label{Def-Distinguished}
Let $I$ be a non-empty proper subset of $\oneto{n}$ and  $j \in \compl{I}$. 
A vector $a \in \projspace$ is said to be an \myemph{$(I,j)$-vertex} of 
$\CC$ if
\begin{equation*}
S_i(a) \cap S_j(a) \not \subset 
\cup_{k \in I \setminus \left\{ i \right\} } S_k(a) \makebox{ for all } i\in I \; , \tag{C4}\label{eq:C4} 	
\end{equation*} 
and Conditions~\ref{item:C1} and~\ref{item:C2} are satisfied, 
where $(S_1(a),\dots,S_n(a))=\type(a)$ is the type of $a$ relative to 
the generating set $\{v^r\}_{r \in [p]}$ of $\CC$. 
\end{definition}

\begin{remark}
As in the case of Conditions~\ref{item:C1} and~\ref{item:C2}, Condition~\eqref{eq:C4} above is independent of the choice of the generating set of $\CC$. Indeed, assuming $\CC \subset \HH(a,I)$, it is equivalent to:
\begin{equation*}
\text{for each } i \in I \text{ there exists } x \in \CC \text{ such that } \vectinverse{a_i} x_i = \vectinverse{a_j} x_j > \mpsum_{k \in I \setminus \{i\}} \vectinverse{a_k} x_k \; . \tag{C4'}\label{eq:C4bis}
\end{equation*}
This provides a geometric interpretation of Condition~\eqref{eq:C4}: $\HH(a,I)$ separates the cone $\CC$ from the sector $\SS(a,j)$, and this separation is ``tight'', since $\SS(a,i) \cap \SS(a,j)$ has a non-empty interior for all $i \in I$.
\end{remark}

Observe that any vector $a$ satisfying the conditions of Definition~\ref{Def-Distinguished} is a vertex of the natural cell decomposition of $\projspace$ induced by $\{v^r\}_{r\in \oneto{p}}$. 
Let $G_S$ be the graph associated with the cell $X_S$, where $S=\type(a)$ (see Section~\ref{SectionBasicNotions}). 
Note that Condition~\eqref{eq:C4} above implies $S_i(a) \cap S_j(a)\neq \emptyset$ for all $i\in I$, and so node $j$ is connected in $G_S$ with each node of $I$. Moreover, by Condition~\ref{item:C2}, each node of $\compl{I}$ is connected in $G_S$ with some node of $I$. Therefore, the graph $G_S$ associated with an $(I,j)$-vertex of $\CC$ is always connected. This shows that $X_S$ is a zero-dimensional cell.

Also observe that if $a$ is an $(I,j)$-vertex of $\CC$, then $\HH(a,I)$ is a minimal half-space with respect to $\CC$, because Condition~\eqref{eq:C4} above is stronger than Condition~\ref{item:C3minimal}. 

The following proposition shows that $(I,j)$-vertices are associated with non-trivial extreme vectors of the $j$th polar of $\CC$. 

\begin{proposition}\label{Prop-Equiv-Distinguished-MinimalH}
The following three properties are equivalent:
\begin{enumerate}[label=(\roman*)]
\item\label{Prop-Equiv-Distinguished-MinimalH-Item1} the vector $a$ is an $(I,j)$-vertex of $\CC$,
\item\label{Prop-Equiv-Distinguished-MinimalH-Item2} $\HH(a,I)$ is a minimal half-space with respect to $\CC$ and $\vectinverse{a_I} \mpplus \vectinverse{a_j} e^j$ is a non-trivial extreme vector of the $j$th polar of $\CC$, 
\item\label{Prop-Equiv-Distinguished-MinimalH-Item3} $\HH(a,I)$ can be obtained by the saturation of a half-space associated with a non-trivial extreme vector of the $j$th polar of $\CC$.
\end{enumerate}
\end{proposition}

\begin{proof}
\begin{description}[wide,font=\normalfont,parsep=0.5ex]
\item[\ref{Prop-Equiv-Distinguished-MinimalH-Item1} $\implies$ \ref{Prop-Equiv-Distinguished-MinimalH-Item2}] Assume the three conditions of Definition~\ref{Def-Distinguished} are satisfied. Thus, as we observed above, $\HH(a,I)$ is a minimal half-space with respect to $\CC$. This implies $\vectinverse{a_I} \mpplus \vectinverse{a_j} e^j$ belongs to the $j$th polar of $\CC$, because the inequality $\vectinverse{a_I} x \geq \vectinverse{a_j} x_j$ holds for all $x \in \CC$. Then, since Condition~\eqref{eq:C4} is satisfied, by Theorem~\ref{TheoCharacExtremeIPolar} we conclude that $\vectinverse{a_I} \mpplus \vectinverse{a_j} e^j$ is a non-trivial extreme vector of the $j$th polar of $\CC$.  

\item[\ref{Prop-Equiv-Distinguished-MinimalH-Item2} $\implies$ \ref{Prop-Equiv-Distinguished-MinimalH-Item1}] Conversely, if $\HH(a,I)$ is a minimal half-space with respect to $\CC$, Conditions~\ref{item:C1} and~\ref{item:C2} are satisfied. Besides, if $\vectinverse{a_I} \mpplus \vectinverse{a_j} e^j$ is a non-trivial extreme vector of the $j$th polar of $\CC$, 
Condition~\eqref{eq:C4bis}, and subsequently Condition~\eqref{eq:C4},
are satisfied by Theorem~\ref{TheoCharacExtremeIPolar}. 
Thus, $a$ is an $(I,j)$-vertex of $\CC$. 

\item[\ref{Prop-Equiv-Distinguished-MinimalH-Item2} $\implies$ \ref{Prop-Equiv-Distinguished-MinimalH-Item3}]
Suppose that $u \mydef \vectinverse{a_I} \mpplus \vectinverse{a_j} e^j$ is a non-trivial extreme vector of the $j$th polar of $\CC$ and $\HH(a,I)$ is minimal with respect to $\CC$. Let $\HH(b,J)$ be the half-space obtained by the saturation of $\{x \in \projmaxspace \mid \mpplus_{i \in [n] \setminus \{j\}} u_i x_i \geq u_j x_j\}$. From Proposition~\ref{PropExtremeSaturationMinimal}, it follows that $J=I$ and $u = \vectinverse{b_I} \mpplus \vectinverse{b_j} e^j$, and so $a_i = b_i$ for all $i \in I$. Besides, since $\HH(b,I)$ is minimal with respect to $\CC$, we have $b \in \CC$ by Corollary~\ref{Coro-Minimal-Apex}. Thus, $\vectinverse{a} b = \vectinverse{a_I} b = 0$, which shows that $b_h \leq a_h$ for all $h \in [n]$. The symmetric inequality can be obtained by exchanging $\HH(a,I)$ and $\HH(b,I)$ (since the former is also minimal), which proves $a = b$.
\item[\ref{Prop-Equiv-Distinguished-MinimalH-Item3} $\implies$ \ref{Prop-Equiv-Distinguished-MinimalH-Item2}] 
Straightforward by Proposition~\ref{PropExtremeSaturationMinimal}.\qedhere

\end{description}
\end{proof}

We are now ready to prove one of the main results of this section.   

\begin{theorem}\label{Theorem-Nonredundant-Distinguised} 
For any non-redundant apex $a\in \mcA$ there exist a non-empty proper subset 
$I$ of $\oneto{n}$ and $j\in \compl{I}$ such that $a$ is an $(I,j)$-vertex of $\CC$. 
\end{theorem}

\begin{proof}
Let $\Gamma \subset \Sigma$ be a finite set of minimal half-spaces with respect to $\CC$ such that $\CC = \cap_{\HH \in \Gamma} \HH$. Up to extracting a subset of half-spaces, we may assume $\Gamma$ is a non-redundant external representation of $\CC$.

Given a non-redundant apex $a$, let $\HH(a,I) \in \Gamma$ be a half-space with apex $a$. By assumption, $\HH(a,I)$ is non-redundant in $\Gamma$, so there exists a vector $x$ such that $x \not \in \HH(a,I)$ and $x \in \HH'$ for all $\HH' \in \Gamma \setminus \{\HH(a,I)\}$. For $r \in [p]$, let $w^r \in \projspace$ be the vector defined by:
\[
w^r \mydef (\vectinverse{a_{\compl{I}}} x) v^r \mpplus (\vectinverse{a_I} v^r) x \; . 
\]
This vector $w^r$ is located on the tropical (projective) segment joining $v^r$ and $x$. Moreover, since $x\not \in \HH(a,I)$ and $v^r\in \CC \subset\HH(a,I)$, we have $\vectinverse{a_I} x < \vectinverse{a_{\compl{I}}} x$ and  $\vectinverse{a_I} v^r \geq \vectinverse{a_{\compl{I}}} v^r$. Then, 
\[
\vectinverse{a_I} w^r = (\vectinverse{a_{\compl{I}}} x) (\vectinverse{a_I} v^r) \mpplus (\vectinverse{a_I} v^r) (\vectinverse{a_I} x) = (\vectinverse{a_{\compl{I}}} x) (\vectinverse{a_I} v^r)  
\]
and  
\[
\vectinverse{a_{\compl{I}}} w^r = (\vectinverse{a_{\compl{I}}} x) (\vectinverse{a_{\compl{I}}} v^r) \mpplus (\vectinverse{a_I} v^r) (\vectinverse{a_{\compl{I}}} x) = (\vectinverse{a_I} v^r) (\vectinverse{a_{\compl{I}}} x) \; . 
\]
Thus, $\vectinverse{a_I} w^r= \vectinverse{a_{\compl{I}}} w^r$, which means that $w^r$ lies on the boundary of $\HH(a,I)$. It follows that $w^r$ belongs to $\CC$, because $w^r$ also belongs to $\HH'$ for any $\HH' \in \Gamma \setminus \{\HH(a,I)\}$ (as a tropical linear combination of vectors of $\HH'$).

We now claim that $\argmax(\vectinverse{a_I} w^r) = \argmax(\vectinverse{a_I} v^r)$ and $\argmax(\vectinverse{a_{\compl{I}}} w^r) \supset \argmax(\vectinverse{a_{\compl{I}}} x)$ for all $r \in \oneto{p}$.  

Indeed, observe that for any $i \in I$, 
\[
(\vectinverse{a_I} v^r) (\vectinverse{a_i} x_i) < (\vectinverse{a_I} v^r) (\vectinverse{a_{\compl{I}}} x) = \vectinverse{a_I} w^r \;.
\]
Then, as $\vectinverse{a_i} w^r_i = (\vectinverse{a_{\compl{I}}} x) (\vectinverse{a_i} v^r_i) \mpplus (\vectinverse{a_I} v^r) (\vectinverse{a_i} x_i)$, we have $i\in \argmax(\vectinverse{a_I} w^r)$ if, and only if, $\vectinverse{a_i} v^r_i = \vectinverse{a_I} v^r$. This proves that $\argmax(\vectinverse{a_I} w^r)$ coincides with $\argmax(\vectinverse{a_I} v^r)$.

Similarly, let $j \in \argmax(\vectinverse{a_{\compl{I}}} x)$. Since 
$\vectinverse{a_I} v^r \geq \vectinverse{a_{\compl{I}}} v^r \geq 
\vectinverse{a_j} v^r_j$ and $\vectinverse{a_{\compl{I}}} x = \vectinverse{a_j} x_j $, it follows that:
\[\vectinverse{a_j} w^r_j = (\vectinverse{a_{\compl{I}}} x) (\vectinverse{a_j} v^r_j) \mpplus (\vectinverse{a_I} v^r) (\vectinverse{a_j} x_j) = 
(\vectinverse{a_I} v^r) (\vectinverse{a_j} x_j) = (\vectinverse{a_I} v^r) (\vectinverse{a_{\compl{I}}} x) \;.
\]
Thus, $\vectinverse{a_j} w^r_j$ is equal to $ \vectinverse{a_{\compl{I}}} w^r$, which shows that $j \in \argmax(\vectinverse{a_{\compl{I}}} w^r)$. This completes the proof of the claim. 

Now, consider any $j \in \argmax(\vectinverse{a_{\compl{I}}} x)$. Then, $j$ also belongs to $\argmax(\vectinverse{a_{\compl{I}}} w^r)$ for all $r \in [p]$. Besides, since the half-space $\HH(a,I)$ is minimal, Condition~\ref{item:C3minimal} is satisfied, so for each $i \in I$ there exists $r_i \in [p]$ such that $r_i \in S_i(a)$ and $r_i \not \in \cup_{k \in I \setminus \{i\}} S_k(a)$. 
Equivalently, $\argmax(\vectinverse{a_I} w^{r_i}) = \argmax(\vectinverse{a_I} v^{r_i})$ is reduced to the singleton $\{i\}$. As $\vectinverse{a_I} w^{r_i} = \vectinverse{a_{\compl{I}}} w^{r_i}$, we conclude that $w^{r_i}$ satisfies:
\[
\vectinverse{a_i} w_i^{r_i} = \vectinverse{a_j} w_j^{r_i} > \mpsum_{k \in I \setminus \{i\}} \vectinverse{a_k} w_k^{r_i} \; .
\]
This shows that Condition~\eqref{eq:C4bis} above is satisfied. Moreover, 
as $\HH(a,I)$ is a minimal half-space with respect to $\CC$, we know that Conditions~\ref{item:C1} and~\ref{item:C2} are also satisfied. In consequence, the non-redundant apex $a$ is an $(I, j)$-vertex of $\CC$.
\end{proof}

We are now going to study a sufficient condition for an $(I, j)$-vertex $a$ to be a non-redundant apex. We first show that this condition implies node $j$ does not belong to the head of the hyperarcs associated with half-spaces different from $\HH (a,I)$.

\begin{lemma}\label{lemma:hyperarc_head}
Let $a$ be an $(I, j)$-vertex of $\CC$ satisfying 
\begin{equation*}
S_i(a) \cap S_j(a) \not \subset 
\cup_{k \in \compl{\{j,i\}}} S_k(a) \makebox{ for all } i \in I \; . \tag{C5} \label{eq:C5} 
\end{equation*}
If $b$ is a $(K, l)$-vertex of $\CC$ such that the half-spaces $\HH(a,I)$ and $\HH(b,K)$ are different, and $\HH(b,K)$ is active at $a$, then $j \not \in \argmax(\vectinverse{b}_{\compl{K}} a)$.
\end{lemma}

\begin{proof}
By contradiction, assume that $j \in \argmax(\vectinverse{b}_{\compl{K}}a) $. Then, we have $\vectinverse{b_j}a_j = \vectinverse{b_K}a$ because $\HH(b,K)$ is active at $a$.

In the first place, assume $I \not \subset \argmax(\vectinverse{b}_K a)$. Given $i \in I \setminus \argmax(\vectinverse{b}_K a)$, since $a$ satisfies Condition~\eqref{eq:C5}, we know that there exists $r \in [p]$ such that $r \in S_i(a) \cap S_j(a)$, and $r \not \in S_k(a)$ for all $k \not \in \{i, j\}$. Equivalently, $\vectinverse{a_i} v^r_i = \vectinverse{a_j} v^r_j > \vectinverse{a_k} v^r_k$ for all $k \not \in \{i, j\}$. Consider $\eta $ such that 
\[
\mpplus_{k \not \in \{i, j\}} \vectinverse{a_k} v^r_k < \mpinverse{\eta} < \vectinverse{a_i} v^r_i=\vectinverse{a_j} v^r_j \;.
\]
Then, the vector $x \mydef a \mpplus \eta v^r$ satisfies $x_i = \eta v^r_i$, $x_j = \eta v^r_j$, and $x_k = a_k$ for all $k \not \in \{i, j\}$. 
In particular, we have $x_k = a_k$ for all $k \in \argmax(\vectinverse{b}_K a)$ because $i \not \in \argmax(\vectinverse{b}_K a)$.  Besides, choosing $\eta$ such that $\mpinverse{\eta}$ is close enough to $\vectinverse{a_i} v^r_i$, we can also suppose that $\argmax(\vectinverse{b}_K x) \subset \argmax(\vectinverse{b}_K a)$. Then, we have 
\[
\vectinverse{b_j} x_j = \vectinverse{b_j} \eta v^r_j > \vectinverse{b_j} a_j = \vectinverse{b_K} a = \mpsum_{k \in \argmax(\vectinverse{b}_K a)} \vectinverse{b_k} x_k \geq \mpsum_{k \in \argmax(\vectinverse{b_K} x)} \vectinverse{b_k} x_k = \vectinverse{b_K} x \; .
\]
This shows that $x$ does not belong to the half-space $\HH(b,K)$. This is a contradiction because $x \in \CC$ (as a tropical linear combination of two elements of $\CC$) and $\CC \subset \HH(b,K)$ (by Condition~\ref{item:C1} applied to $\HH(b,K)$). 

Now assume $I\subset \argmax(\vectinverse{b_K}a)$. Then, since $\vectinverse{b_j}a_j = \vectinverse{b_K}a$, we have $\vectinverse{b_i} a_i = \vectinverse{b_j} a_j$ for all $i \in I$. It is convenient to split the rest of the proof into two cases:
\begin{itemize}
\item[$I \subsetneq K$:] Let $k \in K \setminus I$. Since $b$ is a $(K, l)$-vertex of $\CC$, by Condition~\eqref{eq:C4} there exists $r \in \oneto{p}$ such that $\vectinverse{b_k} v^r_k = \vectinverse{b_l} v^r_l > \mpplus_{h \in K \setminus \{k\}} \vectinverse{b_h} v^r_h$. Then, we have 
\begin{equation}\label{eq:hyperarc_head}
\vectinverse{a}_I v^r = \vectinverse{a_j} b_j (\vectinverse{b}_I v^r) \leq \vectinverse{a_j} b_j (\vectinverse{b}_{K \setminus \{k\}} v^r) < \vectinverse{a_j} b_j \vectinverse{b_k} v^r_k \leq \vectinverse{a_k} b_k \vectinverse{b_k} v^r_k = \vectinverse{a_k} v^r_k \; , 
\end{equation}
where the last inequality follows from $j \in \argmax(\vectinverse{b} a)$, because it implies $\vectinverse{b_j} a_j \geq \vectinverse{b_k} a_k$.
Since $k \in \compl{I}$, we conclude from~\eqref{eq:hyperarc_head} that $v^r \not \in \HH(a,I)$, which is a contradiction.
\item[$I = K$:] We know that $\HH(a,I)$ and $\HH(b,K) = \HH(b,I)$ are both minimal half-spaces with respect to $\CC$. Then, since $a \in \CC \subset \HH(b,I)$ by Corollary~\ref{Coro-Minimal-Apex}, it follows that for all $h \in \compl{I}$, 
\[
\vectinverse{b_j}a_j = \vectinverse{b_I}a \geq \vectinverse{b_{\compl{I}}} a \geq \vectinverse{b_h} a_h \;.
\]
Symmetrically, it can be proved that $\vectinverse{a_j} b_j \geq \vectinverse{a_h} b_h$ for all $h \in \compl{I}$, 
using the fact that $\vectinverse{a_j}b_j = \vectinverse{a_I}b$ and $b \in \HH(a,I)$. 
We conclude that $a$ and $b$ are identical (as elements of the projective space), which contradicts that $\HH(a,I)$ and $\HH(b,K)$ are different half-spaces.\qedhere
\end{itemize} 
\end{proof}

\begin{theorem}\label{th:sufficient_condition_non_redundant_apices}
If $a$ is an $(I,j)$-vertex of $\CC$ satisfying Condition~\eqref{eq:C5}, then $a$ is a non-redundant apex of $\CC$.
\end{theorem}

\begin{proof}
It suffices to consider the external representation $\Gamma$ of $\CC$ composed of the half-spaces provided by Proposition~\ref{Prop-Saturation}, when applied to the non-trivial extreme vectors of the $j$th polar of $\CC$, for all $j\in \oneto{n}$. 

By Proposition~\ref{Prop-Equiv-Distinguished-MinimalH}, a half-space $\HH(b,K)$ belongs to $\Gamma$ if, and only if, its apex $b$ is a $(K,l)$-vertex of $\CC$ for some $l \in \compl{K}$. Then, since $a$ is assumed to be an $(I,j)$-vertex of $\CC$ satisfying Condition~\eqref{eq:C5}, in particular we have $\HH(a,I) \in \Gamma$. Moreover, Lemma~\ref{lemma:hyperarc_head} ensures that $j \not \in H$ for any hyperarc $(T, H)$ in the tangent directed hypergraph $\GG(\Gamma \setminus \{\HH(a,I)\},a)$, because such hyperarc is associated with a half-space $\HH(b,K)$ such that $b$ is a $(K, l)$-vertex of $\CC$ for some $l\in \oneto{n}\setminus K$. As a consequence, the set $[n]$ cannot be reachable from $I$ in $\GG(\Gamma \setminus \{\HH(a,I)\},a)$. From Proposition~\ref{prop:local_redundancy_hypergraph_characterization}, we conclude that $\HH(a,I)$ is not redundant in $\Gamma$, and then $a \in \mcA$ by Theorem~\ref{th:non_redundant_apices}.
\end{proof}

\begin{remark}
When $\CC \subset \projspace[2]$, Theorems~\ref{Theorem-Nonredundant-Distinguised} and~\ref{th:sufficient_condition_non_redundant_apices} allow us to establish that the non-redundant apices of $\CC$ are precisely the vectors $a\in\projspace[2]$ such that $a$ is an $(I,j)$-vertex of $\CC$ for some non-empty proper subset $I$ of $\oneto{3}$ and $j \not \in I$. 

To see this, in the first place assume $I = \{i_1,i_2\}$, with $i_1 \neq i_2$. Then, Condition~\eqref{eq:C5} amounts to $S_{i_1}(a) \cap S_j(a) \not \subset S_{i_2}(a)$ and $S_{i_2}(a) \cap S_j(a) \not \subset S_{i_1}(a)$, which is equivalent to Condition~\eqref{eq:C4}. Thus, Theorem~\ref{th:sufficient_condition_non_redundant_apices} ensures that $a$ is a non-redundant apex.

Assume now $I$ consists of only one element $i$, and let $k \neq j$ be the second element of $\compl{I}$. Let $\Gamma$ be a non-redundant external representation of $\CC$, and assume it does not contain any half-space with apex $a$. Since the half-space $\HH(a,\{i\})$ is redundant with respect to $\Gamma$, the tangent directed hypergraph $\GG(\Gamma,a)$ must necessarily contain a hyperarc from $\{i\}$ to $\{j\}$ or to $\{k\}$ (but not to $\{j,k\}$ because no half-space in $\Gamma$ has apex $a$). Suppose, for instance, that $\{i\}$ is connected with $\{j\}$ by a hyperarc associated with a half-space in $\Gamma$, and let $b$ be its apex. Thus $\mpinverse{b}_i a_i = \mpinverse{b_j} a_j > \mpinverse{b_k} a_k$. We get a contradiction, since $\mpinverse{a}_i b_i < \mpinverse{a_j} b_j \mpplus \mpinverse{a_k} b_k$ while $b \in \CC \subset \HH(a,\{i\})$. 
\end{remark}

We now exhibit a class of real polyhedral cones for which the non-redundant apices are precisely the vertices satisfying Definition~\ref{Def-Distinguished}. 

\begin{definition}\label{def:generic_extremities}
The real polyhedral cone $\CC$ is said to have \myemph{generic extremities} if for each of its generators $v^r$ there exists a non-trivial (\ie~of  
positive radius) Hilbert ball containing $v^r$ and included in $\CC$.
\end{definition}

\begin{remark}\label{remark:generic_extremities}
Definition~\ref{def:generic_extremities} does not depend on the choice of the generating set $\{v^r\}_{r \in \oneto{p}}$. Indeed, $\CC$ has generic extremities if, and only if, for each $x \in \CC$ there exists a non-trivial Hilbert ball $\BB$ such that $x \in \BB \subset \CC$.

To see this, let $x = \mpplus_{r \in \oneto{p}} \lambda_r v^r$ be an arbitrary element of $\CC$. 
For each $r\in \oneto{p}$, suppose that the Hilbert ball with center $c^r$ and radius $\epsilon > 0$ is contained in $\CC$ and contains $v^r$. Without loss of generality, we can assume $\min_{i \in \oneto{n}} (v^r_i - c^r_i) = 0$ for all $r \in \oneto{p}$. Define $c \mydef \mpplus_{r \in \oneto{p}} \lambda_r c^r$, and let $\BB$ be the Hilbert ball with center $c$ and radius $\epsilon$.

First, let us show that $x \in \BB$. For each $i \in \oneto{n}$, there exists $r \in \oneto{p}$ such that $x_i = \lambda_r v^r_i$. Since $c_i \geq \lambda_r c^r_i$, we have $x_i - c_i \leq v^r_i - c^r_i \leq \epsilon$. Similarly, let $s \in \oneto{p}$ such that $c_i = \lambda_s c^s_i$. Then, $x_i - c_i \geq v^s_i - c^s_i \geq 0$ by assumption. It follows that $d_H(x,c) \leq \epsilon$. 

It remains to prove that $\BB \subset \CC$. Consider any $y \in \BB$ and, without loss of generality, assume that $\min_{i \in [n]}(y_i - c_i) = 0$. For each $i \in \oneto{n}$, define $\mu_i\mydef \mpinverse{(\epsilon c^{r_i}_i)}y_i$, where $r_i\in \oneto{p}$ is such that $c_i=\mpplus_{r \in \oneto{p}} \lambda_r c^r_i=\lambda_{r_i} c^{r_i}_i$. Observe that $\mpinverse{\epsilon}\lambda_{r_i} \leq \mu_i \leq \lambda_{r_i}$. We claim that 
\[
y = \mpplus_{j \in \oneto{n}} \mu_j (c^{r_j} \mpplus \epsilon c^{r_j}_j e^j) \;.
\]
Given $i \in \oneto{n}$, the equality $y_i = \mu_i \epsilon c^{r_i}_i = \mu_i (c^{r_i} \mpplus \epsilon c^{r_i}_i e^i)_i$ holds by definition of $\mu_i$. Besides, for $j \neq i$, we have 
\[
\mu_j (c^{r_j} \mpplus \epsilon c^{r_j}_j e^j)_i = \mu_j c^{r_j}_i \leq \lambda_{r_j} c^{r_j}_i \leq \lambda_{r_i} c^{r_i}_i \leq \epsilon \mu_i c^{r_i}_i\;.\]
This shows that $y_i$ is the maximum of the $\mu_j (c^{r_j} \mpplus \epsilon c^{r_j}_j e^j)_i$ for $j \in \oneto{n}$, proving the claim. Finally, since each $c^{r_i} \mpplus \epsilon c^{r_i}_i e^i$ belongs the Hilbert ball with center $c^{r_i}$ and radius $\epsilon$, which is contained in $\CC$, we conclude that $y \in \CC$. 
\end{remark}

Note that, from Remark~\ref{remark:generic_extremities},  
we conclude that any real polyhedral cone which has generic extremities is pure. 

The term \myemph{generic extremities} originates from the fact the aforementioned property holds if, and only if, each extreme vector of $\CC$ belongs to a non-trivial Hilbert ball contained in $\CC$. This enforces that around each of its extreme vectors, the cone has the shape of a Hilbert ball, ensuring a certain ``genericity''.

\begin{remark}
It can be shown that the cone $\CC$ has generic extremities as soon as the following two conditions hold:
\begin{enumerate}[label=(\roman*)] 
\item $\CC$ is pure;
\item the $2 \times 2$-minors 
\[
v^r_i v^s_j \mpplus v^r_j v^s_i = \max\{v^r_i + v^s_j, v^r_j + v^s_i\} \qquad (i,j \in [n], \ r, s \in [p], \ i \neq j, \ r \neq s)
\]
are non-singular, \ie\ the maximum in the right-hand side is reached exactly once.
\end{enumerate}
In particular, the latter condition is satisfied when the vectors $v^1, \dots, v^p$ are in general position in the sense of~\cite{RGST}.
\end{remark}

We say that a cone $\DD$ \myemph{approximates} the cone $\CC$ with precision $\epsilon > 0$ if the Hausdorff distance between $\CC$ and $\DD$ (derived from the metric $d_H$) is bounded by $\epsilon$.
Observe that the real polyhedral cone $\CC$ can be approximated with an arbitrary precision by another one having generic extremities: given $\epsilon > 0$, it suffices to define the tropical cone $\CC_\epsilon$ as the one generated by the Hilbert balls $\BB^r$ with center $v^r$ and radius $\epsilon$, for $r \in [p]$, \ie{}\ the set of tropical linear combinations of the form:
\[
\lambda_1 x^1 \mpplus \dots \mpplus \lambda_p x^p \;, \quad \text{where} \ \lambda_r \in \maxplus \ \text{and}\  x^r \in \BB^r \ \text{for}\  r \in [p].
\]
Since any Hilbert ball with center $c$ and radius $\epsilon$ is polyhedral (its extreme vectors are the vectors $c \mpplus \epsilon c_i e^i$ for $i\in \oneto{n}$), $\CC_\epsilon$ is a real polyhedral cone. Moreover, it can be shown that $\CC_\epsilon$ approximates $\CC$ with precision $\epsilon$. 
Also note that other deformations are possible, for instance choosing balls with different radii for each generator, or approximating each generator by a generic polytrope\footnote{A \myemph{polytrope} is a tropical cone which is also convex in the classical sense, see~\cite{JoswigKulas2010} for further details. We say that it is \myemph{generic} when its extreme vectors are in general position.} containing it. 

\begin{example}
Three Hilbert balls of radius $\frac{1}{2}$ centered at the 
generators $v^1=(0,1,3)$, $v^2=(0,4,1)$ and $v^3=(0,9,4)$ of the 
tropical cone of Figure~\ref{FigTropicalCone} 
are depicted on the left-hand side of Figure~\ref{FigGenericExt}.
Due to the shape of these non-trivial Hilbert balls of $\projspace[2]$, 
it is geometrically clear that the tropical cone of 
Figure~\ref{FigTropicalCone} does not have generic extremities. 
 
A tropical cone with generic extremities 
is shown on the right-hand side of Figure~\ref{FigGenericExt}. 
This cone is generated by the Hilbert balls on the left. 
Observe that letting the radii of these Hilbert balls tend to zero, 
the tropical cone of Figure~\ref{FigTropicalCone} can be 
approximated as much as we want. 
   
\begin{figure}
\begin{center}
\begin{tikzpicture}[convex/.style={draw=lightgray,fill=lightgray,fill opacity=0.7},convexborder/.style={ultra thick},>=triangle 45,scale=0.7]
\begin{scope}[point/.style={black}]
\draw[gray!30,very thin] (-0.5,-0.5) grid (9.5,6.5);
\draw[gray!50,->] (-0.5,0) -- (9.5,0) node[color=gray!50,below] {$x_2-x_1$};
\draw[gray!50,->] (0,-0.5) -- (0,6.5) node[color=gray!50,above] {$x_3-x_1$};	

\node[coordinate] (b1_center) at (1,3) {};
\path (b1_center) ++ (-0.5,-0.5) node[coordinate] (b1_1) {};
\path (b1_center) ++ (0,-0.5) node[coordinate] (b1_2) {};
\path (b1_center) ++ (0.5,0) node[coordinate] (b1_3) {};
\path (b1_center) ++ (0.5,0.5) node[coordinate] (b1_4) {};
\path (b1_center) ++ (0,0.5) node[coordinate] (b1_5) {};
\path (b1_center) ++ (-0.5,0) node[coordinate] (b1_6) {};

\filldraw[convex] (b1_1) -- (b1_2) -- (b1_3) -- (b1_4) -- (b1_5) -- (b1_6) -- cycle;
\draw[convexborder] (b1_1) -- (b1_2) -- (b1_3) -- (b1_4) -- (b1_5) -- (b1_6) -- cycle;
\filldraw[point] (b1_center) circle (1.5pt);

\node[coordinate] (b2_center) at (4,1) {};
\path (b2_center) ++ (-0.5,-0.5) node[coordinate] (b2_1) {};
\path (b2_center) ++ (0,-0.5) node[coordinate] (b2_2) {};
\path (b2_center) ++ (0.5,0) node[coordinate] (b2_3) {};
\path (b2_center) ++ (0.5,0.5) node[coordinate] (b2_4) {};
\path (b2_center) ++ (0,0.5) node[coordinate] (b2_5) {};
\path (b2_center) ++ (-0.5,0) node[coordinate] (b2_6) {};

\filldraw[convex] (b2_1) -- (b2_2) -- (b2_3) -- (b2_4) -- (b2_5) -- (b2_6) -- cycle;
\draw[convexborder] (b2_1) -- (b2_2) -- (b2_3) -- (b2_4) -- (b2_5) -- (b2_6) -- cycle;
\filldraw[point] (b2_center) circle (1.5pt);

\node[coordinate] (b3_center) at (9,4) {};
\path (b3_center) ++ (-0.5,-0.5) node[coordinate] (b3_1) {};
\path (b3_center) ++ (0,-0.5) node[coordinate] (b3_2) {};
\path (b3_center) ++ (0.5,0) node[coordinate] (b3_3) {};
\path (b3_center) ++ (0.5,0.5) node[coordinate] (b3_4) {};
\path (b3_center) ++ (0,0.5) node[coordinate] (b3_5) {};
\path (b3_center) ++ (-0.5,0) node[coordinate] (b3_6) {};

\filldraw[convex] (b3_1) -- (b3_2) -- (b3_3) -- (b3_4) -- (b3_5) -- (b3_6) -- cycle;
\draw[convexborder] (b3_1) -- (b3_2) -- (b3_3) -- (b3_4) -- (b3_5) -- (b3_6) -- cycle;
\filldraw[point] (b3_center) circle (1.5pt);
\end{scope}
\begin{scope}[xshift=11.8cm,point/.style={blue!50},ball/.style={green!60!black,ultra thick,densely dotted},ballcenter/.style={green!60!black},hsborder/.style={orange,dashdotted,thick},hs/.style={draw=none,pattern=north west lines,pattern color=orange,fill opacity=0.5},apex/.style={orange!90!black}]
	
\draw[gray!30,very thin] (-0.5,-0.5) grid (9.5,6.5);
\draw[gray!50,->] (-0.5,0) -- (9.5,0) node[color=gray!50,below] {$x_2-x_1$};
\draw[gray!50,->] (0,-0.5) -- (0,6.5) node[color=gray!50,above] {$x_3-x_1$};	

\node[coordinate] (b1_center) at (1,3) {};
\path (b1_center) ++ (-0.5,-0.5) node[coordinate] (b1_1) {};
\path (b1_center) ++ (0,-0.5) node[coordinate] (b1_2) {};
\path (b1_center) ++ (0.5,0) node[coordinate] (b1_3) {};
\path (b1_center) ++ (0.5,0.5) node[coordinate] (b1_4) {};
\path (b1_center) ++ (0,0.5) node[coordinate] (b1_5) {};
\path (b1_center) ++ (-0.5,0) node[coordinate] (b1_6) {};

\node[coordinate] (b2_center) at (4,1) {};
\path (b2_center) ++ (-0.5,-0.5) node[coordinate] (b2_1) {};
\path (b2_center) ++ (0,-0.5) node[coordinate] (b2_2) {};
\path (b2_center) ++ (0.5,0) node[coordinate] (b2_3) {};
\path (b2_center) ++ (0.5,0.5) node[coordinate] (b2_4) {};
\path (b2_center) ++ (0,0.5) node[coordinate] (b2_5) {};
\path (b2_center) ++ (-0.5,0) node[coordinate] (b2_6) {};

\node[coordinate] (b3_center) at (9,4) {};
\path (b3_center) ++ (-0.5,-0.5) node[coordinate] (b3_1) {};
\path (b3_center) ++ (0,-0.5) node[coordinate] (b3_2) {};
\path (b3_center) ++ (0.5,0) node[coordinate] (b3_3) {};
\path (b3_center) ++ (0.5,0.5) node[coordinate] (b3_4) {};
\path (b3_center) ++ (0,0.5) node[coordinate] (b3_5) {};
\path (b3_center) ++ (-0.5,0) node[coordinate] (b3_6) {};

\path let \p1=(b1_1), \p2 = (b2_1) in node[coordinate] (b1_1_b2_1) at (\x2,\y1) {};
\coordinate (b2_1_b3_3) at (6,0.5);
\coordinate (b3_5_b1_5) at (8,3.5);
\filldraw[convex] (b1_5) -- (b1_6) -- (b1_1) -- (b1_1_b2_1) -- (b2_1) -- (b2_1_b3_3) -- (b3_3) -- (b3_4) -- (b3_5) -- (b3_5_b1_5) -- cycle;
\draw[convexborder,very thick] (b1_5) -- (b1_6) -- (b1_1) -- (b1_1_b2_1) -- (b2_1) -- (b2_1_b3_3) -- (b3_3) -- (b3_4) -- (b3_5) -- (b3_5_b1_5) -- cycle;

\draw[ball] (b1_1) -- (b1_2) -- (b1_3) -- (b1_4) -- (b1_5) -- (b1_6) -- cycle;
\draw[ball] (b2_1) -- (b2_2) -- (b2_3) -- (b2_4) -- (b2_5) -- (b2_6) -- cycle;
\draw[ball] (b3_1) -- (b3_2) -- (b3_3) -- (b3_4) -- (b3_5) -- (b3_6) -- cycle;

\draw[hsborder] (0.5,-0.5) -- (0.5,3) -- (4,6.5);
\filldraw[hs] (0.5,-0.5) -- (0.5,3) -- (4,6.5) -- (4.3,6.5) -- (0.7,2.9) -- (0.7,-0.5) -- cycle;

\draw[hsborder] (-0.5,2.5) -- (3.5,2.5) -- (3.5,-0.5);
\filldraw[hs] (-0.5,2.5) -- (3.5,2.5) -- (3.5,-0.5) -- (3.7,-0.5) -- (3.7,2.7) -- (-0.5,2.7) -- cycle;

\draw[hsborder] (-0.5,0.5) -- (6,0.5) -- (9.5,4);
\filldraw[hs] (-0.5,0.5) -- (6,0.5) -- (9.5,4) -- (9.5,4.3) -- (5.9,0.7) -- (-0.5,0.7) -- cycle;

\draw[hsborder] (-0.5,4.5) -- (9.5,4.5) -- (9.5,-0.5);
\filldraw[hs] (-0.5,4.5) -- (9.5,4.5) -- (9.5,-0.5) -- (9.3,-0.5) -- (9.3,4.3) -- (-0.5,4.3) -- cycle;

\draw[hsborder] (-0.5,3.5) -- (8,3.5) -- (9.5,5);
\filldraw[hs] (-0.5,3.5) -- (8,3.5) -- (9.5,5) -- (9.5,4.7) -- (8.1,3.3) -- (-0.5,3.3) -- cycle;

\filldraw[ballcenter] (b1_center) circle (1.5pt);
\filldraw[point] (b1_5) circle (3.5pt) node[above left=1.5pt] {$w^5$};
\filldraw[apex] (b1_6) circle (3.5pt);
\filldraw[point] (b1_1) circle (3.5pt) node[below left=1.5pt] {$w^1$};
\filldraw[apex] (b1_1_b2_1) circle (3.5pt);
\filldraw[ballcenter] (b2_center) circle (1.5pt);
\filldraw[point] (b2_1) circle (3.5pt) node[below left=1.5pt] {$w^2$};
\filldraw[apex] (b2_1_b3_3) circle (3.5pt);
\filldraw[ballcenter] (b3_center) circle (1.5pt);
\filldraw[point] (b3_3) circle (3.5pt) node[right=3pt] {$w^3$};
\filldraw[apex] (b3_4) circle (3.5pt);
\filldraw[point] (b3_5) circle (3.5pt) node[above=3pt] {$w^4$};
\filldraw[apex] (b3_5_b1_5) circle (3.5pt);
\end{scope}
\end{tikzpicture}
\end{center}
\caption{Three Hilbert balls (left) and a tropical cone with generic extremities (right).}
\label{FigGenericExt}
\end{figure}
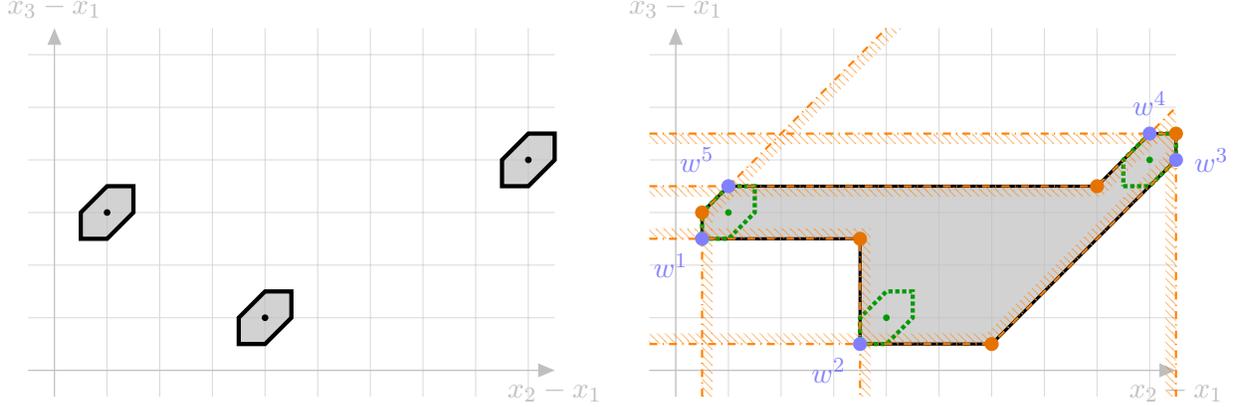
\end{example}

The following theorem shows that if $\CC$ has generic extremities, then the vertices of the cell decomposition introduced in Definition~\ref{Def-Distinguished} are precisely the non-redundant apices. Besides, it proves that they provide the unique non-redundant external representation composed of minimal half-spaces.
\begin{theorem}\label{th:generic_extremities}
If the real polyhedral cone $\CC$ has generic extremities, the non-redundant apices of $\CC$ are precisely the vectors $a$ for which there exist a non-empty proper subset $I$ of $\oneto{n}$ and $j \in \compl{I}$ such that $a$ is an $(I,j)$-vertex of $\CC$. 

Moreover, each such set $I$ is uniquely determined, and the collection of the half-spaces $\HH(a,I)$ is the unique non-redundant external representation of $\CC$ composed of minimal half-spaces.  
\end{theorem}

\begin{proof}
In the first place, we prove that any $(I,j)$-vertex of $\CC$ satisfies Condition~\eqref{eq:C5}. 

By the contrary, assume $a$ is an $(I,j)$-vertex of $\CC$ for which Condition~\eqref{eq:C5} does not hold. Then, for some $i \in I$, given any extreme vector $v$ of $\CC$ 
satisfying  $\vectinverse{a_i} v_i=\vectinverse{a_j} v_j=\vectinverse{a} v =\vectinverse{a_I}v$, there exists $k \in \compl{\{i,j\}}$ such that $\vectinverse{a_j} v_j = \vectinverse{a_k} v_k = \vectinverse{a_I} v$. Since Condition~\eqref{eq:C4} holds, $v$ can be chosen so that $\vectinverse{a_i} v_i=\vectinverse{a_j} v_j > \vectinverse{a_h} v_h$ for all $h\in I\setminus \{ i\}$, and so $k\in \compl{I}$. 

Let $\BB$ be a Hilbert ball with center $c$ and radius $\epsilon > 0$ such that $v \in \BB \subset \CC$. Since $v$ is extreme in $\CC$, it is also extreme in $\BB$. Recalling that the extreme vectors of $\BB$ are the vectors $c\oplus \epsilon c_i e^i$ for $i\in \oneto{n}$, it follows that there exists $l \in \oneto{n}$ such that $v_l = \epsilon c_l$, and $v_h = c_h$ for $h \neq l$. Then, for any $0< \eta < \epsilon$, the vector $x$ defined by $x_l = \mpinverse{\eta} v_l$, and $x_h = v_h$ for $h \neq l$, is in the interior of $\BB$. 

Now, suppose that $l \neq j$. Since $x$ is in the interior of $\BB$, there exists $\eta' > 0$ such that the vector $y$ defined by $y_j = \eta' x_j$, and $y_h = x_h$ for $h \neq j$, belongs to $\BB$. Then, we have
\[
\vectinverse{a_j} y_j > \vectinverse{a_j} x_j = \vectinverse{a_j} v_j = \vectinverse{a_I} v \geq \vectinverse{a_I} x = \vectinverse{a_I} y \; .
\]
However, this is impossible, because $y \in \BB \subset \CC \subset \HH(a,I)$. Thus, $l$ must be equal to $j$. The same reasoning holds with $k$ instead of $j$, and leads to $l = k$. Since $j$ and $k$ are distinct, we obtain a contradiction. Therefore, every $(I,j)$-vertex of $\CC$ must satisfy Condition~\eqref{eq:C5}. As a consequence, we conclude from Theorems~\ref{Theorem-Nonredundant-Distinguised} and~\ref{th:sufficient_condition_non_redundant_apices} that the non-redundant apices are precisely those vertices $a$ for which there exist a non-empty proper subset $I$ of $\oneto{n}$ and $j \in \compl{I}$ such that $a$ is an $(I,j)$-vertex of $\CC$.  

Finally, assume $a$ is both an $(I,j)$-vertex and an $(I',j')$-vertex of $\CC$. Then, since $\HH(a,I)$ and $\HH(a,I')$ are minimal half-spaces with respect to $\CC$ and $\CC$ is pure, by Lemma~\ref{Lemma-Pure-Case} we necessarily have $I = I'$. This proves that $I$ is indeed uniquely determined. Moreover, by Proposition~\ref{Prop-Pure-Case}, there is a unique non-redundant external representation composed of minimal half-spaces. According to Lemma~\ref{Lemma-Pure-Case} and Proposition~\ref{Prop-Equiv-Distinguished-MinimalH}, $\HH(a,I)$ is the only half-space with apex $a$ appearing in such representation.
\end{proof}

\begin{example}
The non-redundant apices of the cone of Figure~\ref{FigGenericExt} (right) are the vectors $(0,\frac{1}{2},3)$, $(0,\frac{7}{2},\frac{5}{2})$, $(0,6,\frac{1}{2})$, $(0,\frac{19}{2},\frac{9}{2})$, and $(0,8,\frac{7}{2})$, which are respectively $(\{2\},\cdot)$-, $(\{2,3\},1)$-, $(\{3\},\cdot)$-, $(\{1\},\cdot)$-, and $(\{1,2\},3)$-vertices (the notation $(I,\cdot)$ stands for any couple $(I,j)$ with $j \not \in I$). They are depicted in orange together with the corresponding minimal half-spaces, while the extreme vectors $w^1,\ldots ,w^5$ of the cone are represented in blue. 
\end{example}

\begin{figure}
\begin{center}
\includemovie[mimetype=model/u3d,
text={\includegraphics[scale=0.5]{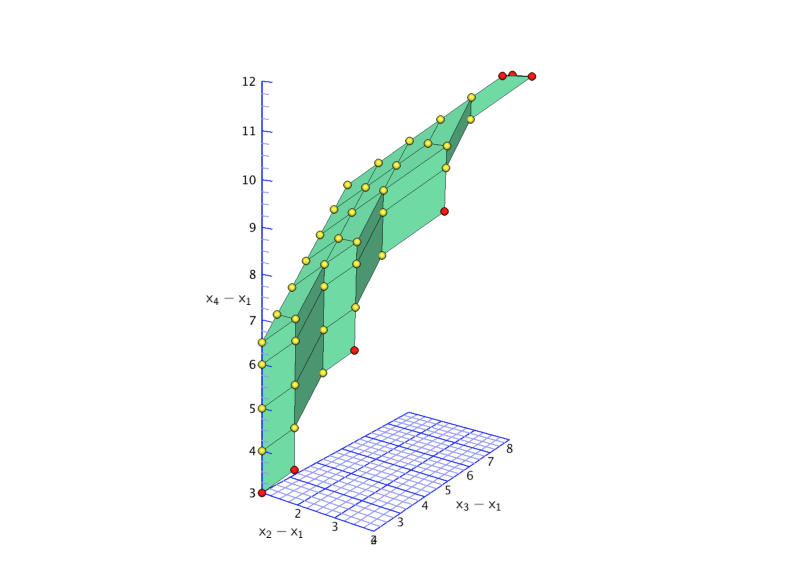}},
attach=true,
3Djscript=figures/Encompass3.js,
3Dlights=CAD]{}{}{figures/cyclic_pure_perturbation.u3d}
\end{center}
\caption{Perturbation of the $4$th cyclic cone of $\projspace[3]$ into a pure cone.}\label{fig:cyclic_pure_perturbation}
\end{figure}

\begin{remark}
Theorem~\ref{th:generic_extremities} cannot be generalized to the case of pure cones. As an example, consider the perturbation of the $4$th cyclic cone in $\projspace[3]$ generated by the following vectors:
\[
\begin{array}{c@{\qquad}c@{\qquad}c@{\qquad}c@{\qquad}c}
(0,1,2,3) & (0,1,\frac{5}{2},\frac{7}{2}) & (0,\frac{3}{2},\frac{5}{2},\frac{7}{2}) & (0,\frac{3}{2},4,6) & (0,2,4,6) \\
(0,\frac{5}{2},6,9) & (0,3,6,9) & (0,\frac{7}{2},\frac{15}{2},12) & (0,\frac{7}{2},8,12) & (0,4,8,12)  
\end{array}
\]
This cone can be verified to be pure (see Figure~\ref{fig:cyclic_pure_perturbation}), for instance by testing that the subcomplex of bounded cells of the natural cell decomposition induced by its generators is pure and full-dimensional. For the sake of completeness, we provide a {\normalfont\texttt{polymake}}\footnote{Version 2.12 or later.} script allowing to check this property:
\begin{footnotesize}
\begin{verbatim}
application "tropical";
$gen = new Matrix<Rational>([[0,1,2,3],[0,1,5/2,7/2],[0,3/2,5/2,7/2],[0,3/2,4,6],[0,2,4,6],
                             [0,5/2,6,9],[0,3,6,9],[0,7/2,15/2,12],[0,7/2,8,12],[0,4,8,12]]);
$p = new TropicalPolytope<Rational>(POINTS=>-$gen);
$trunc_vertices = $p->PSEUDOVERTICES->minor(All,range(1,$p->AMBIENT_DIM));
$n_vertices = scalar(@{$trunc_vertices});
$all_ones = new Vector<Rational>([ (1)x$n_vertices ]);
$vertices = ($all_ones|$trunc_vertices);
$max_cells = $p->ENVELOPE->BOUNDED_COMPLEX->MAXIMAL_POLYTOPES;
$cell_complex = new fan::PolyhedralComplex(VERTICES=>$vertices,MAXIMAL_CELLS=>$max_cells);
if ($cell_complex->FULL_DIM && $cell_complex->PURE) { 
    print "The cone is pure." 
} else { 
    print "The cone is not pure." 
};
\end{verbatim}
\end{footnotesize}
The external representation obtained by saturating the half-spaces associated with the non-trivial extreme vectors of the polar cones contains a half-space with apex $a = (0,1,\frac{5}{2},\frac{9}{2})$. The type of the latter vector is $S(a) = (\{1,2\},\{1,2,3\},\{2,4,5\},\{4,5,\dots,10\})$, so that $S_2(a) \cap S_3(a) \not \subset S_4(a)$ and $S_4(a) \cap S_3(a) \not \subset S_2(a)$. The vector $a$ is consequently a $(\{2,4\},3)$-vertex. However, by Theorem~\ref{th:non_redundant_apices}, $a$ is not a non-redundant apex since the following list of half-spaces provides a non-redundant external representation of the cone:
\[
\begin{array}{c@{\quad}c@{\quad}c@{\quad}c}
(0,1,2,3), \{4\} & (0,1,2,13/2), \{3\} & (0,1,7/2,13/2), \{2,4\} & (0,1,9/2,17/2), \{2,4\} \\
(0,1,11/2,19/2), \{2\} & (0,3/2,5/2,9/2), \{1,4\} & (0,3/2,7/2,8), \{1,3\} & (0,2,4,7), \{1,4\} \\
(0,2,5,19/2), \{1,3\} & (0,3,6,10), \{1,4\} & (0,3,7,23/2), \{1,3\} & (0,4,8,12), \{1\} 
\end{array}
\]

\end{remark}

\section*{Acknowledgements}

Both authors are very grateful to St{\'e}phane Gaubert for many helpful discussions on tropical convexity, and for his suggestion to study the special case of tropical cones with generic extremities. The authors also thank the two anonymous reviewers for their comments and suggestions which helped to improve the presentation of the results. 

\bibliographystyle{alpha}

\begin{thebibliography}{CDQV85}

\bibitem[AGG10]{AGG10}
X.~Allamigeon, S.~Gaubert, and {\'E}.~Goubault.
\newblock The tropical double description method.
\newblock In J.-Y. Marion and Th. Schwentick, editors, {\em Proceedings of the
  27th International Symposium on Theoretical Aspects of Computer Science
  (STACS 2010)}, volume~5 of {\em Leibniz International Proceedings in
  Informatics (LIPIcs)}, pages 47--58, Dagstuhl, Germany, 2010. Schloss
  Dagstuhl--Leibniz-Zentrum fuer Informatik.

\bibitem[AGG13]{AllamigeonGaubertGoubaultDCG2013}
X.~Allamigeon, S.~Gaubert, and {\'E}.~Goubault.
\newblock Computing the vertices of tropical polyhedra using directed
  hypergraphs.
\newblock {\em Discrete \& Computational Geometry}, 49(2):247--279, 2013.
\newblock E-print \arxiv{0904.3436v4}.

\bibitem[AGK11a]{AGK09}
X.~Allamigeon, S.~Gaubert, and R.~D. Katz.
\newblock The number of extreme points of tropical polyhedra.
\newblock {\em Journal of Combinatorial Theory, Series A}, 118(1):162--189,
  2011.
\newblock E-print \arxiv{0906.3492}.

\bibitem[AGK11b]{AGK-10}
X.~Allamigeon, S.~Gaubert, and R.~D. Katz.
\newblock Tropical polar cones, hypergraph transversals, and mean payoff games.
\newblock {\em Linear Algebra Appl.}, 435(7):1549--1574, 2011.
\newblock E-print \arxiv{1004.2778}.

\bibitem[All09]{tplib}
X.~Allamigeon.
\newblock {TPLib}: Tropical polyhedra library in {OC}aml.
\newblock Available at \url{https://gforge.inria.fr/projects/tplib}, 2009.

\bibitem[BH04]{BriecHorvath04}
W.~Briec and C.~Horvath.
\newblock $\mathbb{B}$-convexity.
\newblock {\em Optimization}, 53:103--127, 2004.

\bibitem[BHR05]{BriecHorvRub05}
W.~Briec, C.~Horvath, and A.~Rubinov.
\newblock Separation in {$\B$}-convexity.
\newblock {\em Pacific Journal of Optimization}, 1:13--30, 2005.

\bibitem[BY06]{blockyu06}
F.~Block and J.~Yu.
\newblock Tropical convexity via cellular resolutions.
\newblock {\em J. Algebraic Combin.}, 24(1):103--114, 2006.
\newblock E-print \arxiv{math.MG/0503279}.

\bibitem[CDQV85]{cohen85a}
G.~Cohen, D.~Dubois, J.~P. Quadrat, and M.~Viot.
\newblock A linear system theoretic view of discrete event processes and its
  use for performance evaluation in manufacturing.
\newblock {\em IEEE Trans. on Automatic Control}, AC--30:210--220, 1985.

\bibitem[CGQ99]{ccggq99}
G.~Cohen, S.~Gaubert, and J.~P. Quadrat.
\newblock Max-plus algebra and system theory: where we are and where to go now.
\newblock {\em Annual Reviews in Control}, 23:207--219, 1999.

\bibitem[CGQ01]{cgq00}
G.~Cohen, S.~Gaubert, and J.~P. Quadrat.
\newblock Hahn-{B}anach separation theorem for max-plus semimodules.
\newblock In J.~L. Menaldi, E.~Rofman, and A.~Sulem, editors, {\em Optimal
  Control and Partial Differential Equations}, pages 325--334. IOS Press, 2001.

\bibitem[CGQ04]{cgq02}
G.~Cohen, S.~Gaubert, and J.~P. Quadrat.
\newblock Duality and separation theorems in idempotent semimodules.
\newblock {\em Linear Algebra Appl.}, 379:395--422, 2004.
\newblock E-print \arxiv{math.FA/0212294}.

\bibitem[CGQS05]{cgqs04}
G.~Cohen, S.~Gaubert, J.~P. Quadrat, and I.~Singer.
\newblock Max-plus convex sets and functions.
\newblock In G.~L. Litvinov and V.~P. Maslov, editors, {\em Idempotent
  Mathematics and Mathematical Physics}, volume 377 of {\em Contemporary
  Mathematics}, pages 105--129. American Mathematical Society, 2005.
\newblock Also ESI Preprint 1341, \arxiv{math.FA/0308166}.

\bibitem[DS04]{DS}
M.~Develin and B.~Sturmfels.
\newblock Tropical convexity.
\newblock {\em Doc. Math.}, 9:1--27 (electronic), 2004.
\newblock E-print \arxiv{math.MG/0308254}.

\bibitem[DY07]{DevelinYu}
M.~Develin and J.~Yu.
\newblock Tropical polytopes and cellular resolutions.
\newblock {\em Experimental Mathematics}, 16(3):277--292, 2007.
\newblock E-print \arxiv{math.CO/0605494}.

\bibitem[GJ00]{polymake}
E.~Gawrilow and M.~Joswig.
\newblock polymake: a framework for analyzing convex polytopes.
\newblock In G.~Kalai and G.~M. Ziegler, editors, {\em Polytopes ---
  Combinatorics and Computation}, pages 43--74. Birkh\"auser, 2000.

\bibitem[GK06]{GK06a}
S.~Gaubert and R.~D. Katz.
\newblock Max-plus convex geometry.
\newblock In R.~A. Schmidt, editor, {\em Proceedings of the 9th International
  Conference on Relational Methods in Computer Science and 4th International
  Workshop on Applications of Kleene Algebra (RelMiCS/AKA 2006)}, volume 4136
  of {\em Lecture Notes in Comput. Sci.}, pages 192--206. Springer, 2006.

\bibitem[GK07]{GK}
S.~Gaubert and R.~D. Katz.
\newblock The {M}inkowski theorem for max-plus convex sets.
\newblock {\em Linear Algebra Appl.}, 421(2-3):356--369, 2007.
\newblock E-print \arxiv{math.GM/0605078}.

\bibitem[GK09]{katz08}
S.~Gaubert and R.~D. Katz.
\newblock The tropical analogue of polar cones.
\newblock {\em Linear Algebra Appl.}, 431(5-7):608--625, 2009.
\newblock E-print \arxiv{0805.3688}.

\bibitem[GK11]{GK09}
S.~Gaubert and R.~D. Katz.
\newblock Minimal half-spaces and external representation of tropical
  polyhedra.
\newblock {\em J. Algebraic Combin.}, 33(3):325--348, 2011.
\newblock E-print \arxiv{0908.1586}.

\bibitem[GLPN93]{GalloDAM93}
G.~Gallo, G.~Longo, S.~Pallottino, and S.~Nguyen.
\newblock Directed hypergraphs and applications.
\newblock {\em Discrete Appl. Math.}, 42(2-3):177--201, 1993.

\bibitem[JK10]{JoswigKulas2010}
M.~Joswig and K.~Kulas.
\newblock Tropical and ordinary convexity combined.
\newblock {\em Advances in Geometry}, 10:333--352, 2010.
\newblock E-print \arxiv{0801.4835}.

\bibitem[Jos05]{joswig04}
M.~Joswig.
\newblock Tropical halfspaces.
\newblock In {\em Combinatorial and computational geometry}, volume~52 of {\em
  Math. Sci. Res. Inst. Publ.}, pages 409--431. Cambridge Univ. Press,
  Cambridge, 2005.
\newblock E-print \arxiv{math.CO/0312068}.

\bibitem[KLS91]{korte}
B.~Korte, L.~Lov\'asz, and R.~Schrader.
\newblock {\em Greedoids}.
\newblock Springer, 1991.

\bibitem[LMS01]{litvinov00}
G.~L. Litvinov, V.~P. Maslov, and G.~B. Shpiz.
\newblock Idempotent functional analysis: an algebraic approach.
\newblock {\em Math. Notes}, 69(5):696--729, 2001.
\newblock E-print \arxiv{math.FA/0009128}.

\bibitem[NS07]{NiticaSinger07a}
V.~Nitica and I.~Singer.
\newblock Max-plus convex sets and max-plus semispaces. {I}.
\newblock {\em Optimization}, 56(1--2):171--205, 2007.

\bibitem[RGST05]{RGST}
J.~Richter-Gebert, B.~Sturmfels, and T.~Theobald.
\newblock First steps in tropical geometry.
\newblock In {\em Idempotent Mathematics and Mathematical Physics}, volume 377
  of {\em Contemporary Mathematics}, pages 289--317. American Mathematical
  Society, Providence, RI, 2005.
\newblock E-print \arxiv{math.AG/0306366}.

\bibitem[Sin97]{ACA}
I.~Singer.
\newblock {\em Abstract convex analysis}.
\newblock Wiley, 1997.

\bibitem[SS92]{shpiz}
S.~N. Samborski{\u\i} and G.~B. Shpiz.
\newblock Convex sets in the semimodule of bounded functions.
\newblock In {\em Idempotent analysis}, pages 135--137. American Mathematical
  Society, Providence, RI, 1992.

\bibitem[Zim77]{zimmerman77}
K.~Zimmermann.
\newblock A general separation theorem in extremal algebras.
\newblock {\em Ekonomicko-matematicky Obzor}, 13(2):179--201, 1977.

\end{thebibliography}
\def\cprime{$'$} \def\cprime{$'$}

\end{document}